\documentclass[11pt]{amsart}

\usepackage{amsmath, amsthm, amssymb}
\usepackage[pdftex]{hyperref}
\usepackage{upgreek}
\usepackage{graphicx}
\usepackage{enumerate}
\usepackage{hyperref}
\usepackage{comment}
\pdfoutput=1

\usepackage{fullpage}
\usepackage{amsfonts}
\usepackage{MnSymbol}

\usepackage[usenames,dvipsnames]{xcolor}

        \definecolor{darkspringgreen}{rgb}{0.09, 0.45, 0.27}
        \definecolor{lava}{rgb}{0.81, 0.06, 0.13}

\usepackage{todonotes}

\newtheorem{thm}{Theorem}[section]
\newtheorem{lem}[thm]{Lemma}
\newtheorem{cor}[thm]{Corollary}
\newtheorem{prop}[thm]{Proposition}

\newtheorem{defn}[thm]{Definition}
\newtheorem{rmk}[thm]{Remark}

\numberwithin{equation}{section}

\newcommand{\RR}{\mathbb{R}}
\newcommand{\sfd}{{\sf d}}

\newcommand{\RCD}{\mathsf{RCD}}
\newcommand{\RCDst}{\mathsf{RCD}^*}
\newcommand{\supp}{\mathsf{supp}}
\newcommand{\Lip}{{\rm Lip}}
\def\square{\hfill${\vcenter{\vbox{\hrule height.4pt \hbox{\vrule width.4pt
height7pt \kern7pt \vrule width.4pt} \hrule height.4pt}}}$}

\newcommand{\op}{\operatorname}

\newcommand{\Vol}{\mbox{Vol}}

\newcommand{\Ric}{\mbox{Ric}}

\newcommand{\ra}{\rightarrow}

\begin{document}

\title{Maximal Volume Entropy Rigidity \\ for $\RCDst(-(N-1),N)$ Spaces}


\author[Connell]{Chris Connell}
\address[Chris Connell]{115 Rawles Hall, Indiana University, Bloomington, IN 47405 }
\email{connell@indiana.edu}

\author[Dai]{Xianzhe Dai}
\address[Xianzhe Dai]{Department of Mathematics, UCSB, Santa Barbara CA 93106}
\email{dai@math.ucsb.edu}

\author[N\'u\~nez-Zimbr\'on]{Jes\'us N\'u\~nez-Zimbr\'on}
\address[Jes\'us N\'u\~nez-Zimbr\'on]{Centro de Investigaci\'on en Matem\'aticas CIMAT, Guanajuato, Mexico}
\email{jesus.nunez@cimat.mx}

\author[Perales]{Raquel Perales}
\address[Raquel Perales]{Instituto de Matem\'aticas, Universidad Nacional Aut\'onoma de M\'exico, Oaxaca, Mexico}
\email{raquel.perales@matem.unam.mx}

\author[Su\'arez-Serrato]{Pablo Su\'arez-Serrato}
\address[Pablo Su\'arez-Serrato]{Instituto de Matem\'aticas, Universidad Nacional Aut\'onoma de M\'exico, Mexico City and UCSB, Santa Barbara CA 93106}
\email{pablo@im.unam.mx}

\author[Wei]{Guofang Wei}
\address[Guofang Wei]{Mathematics Department, University of California, Santa Barbara, CA 93106}
\email{wei@math.ucsb.edu}


\thanks{The authors thank The University of California Institute for Mexico and the United States (UC MEXUS) for support for this project through the grant CN-16-43. CC was partially supported by a grant from the Simons Foundation \#210442. XD was partially supported by NSF DMS 1611915 and the Simons Foundation. JNZ was partially supported by a UCMEXUS postdoctoral grant under the project 'Alexandrov Geometry' and a MIUR SIR-grant 'Nonsmooth Differential Geometry' (RBSI147UG4). PSS was partially supported by DGAPA-UNAM PAPIIT  IN102716 and DGAPA UNAM PASPA program. GW was partially supported by NSF DMS grant 1506393, 1811558.}

\maketitle
\date{\today}


\begin{abstract}
For $n$-dimensional Riemannian manifolds  $M$ with Ricci curvature bounded below by $-(n-1)$, the volume entropy is bounded above by $n-1$. If $M$ is compact, it is known that the equality holds if and only if $M$ is hyperbolic. We extend this result to $\RCDst(-(N-1),N)$ spaces.  While the upper bound is straight forward, the rigidity case is quite involved due to the lack of a smooth structure in $\RCDst$ spaces. As an application, we obtain an almost rigidity result which partially recovers a result by Chen-Rong-Xu  for Riemannian manifolds.   
\end{abstract}


\section{Introduction}\label{intro}
Volume entropy is  a fundamental geometric invariant, related to the topological entropy of geodesic flows, minimal volume, simplicial volume, bottom spectrum of the Laplacian of the universal cover, among others.   For a compact Riemannian manifold $(M^n, g)$, the volume entropy is defined as,
\[
h(M,g) = \lim_{ R \ra \infty} \frac{\ln \Vol (B(x, R))}{R}.
\]
Here $B(x, R)$ is a ball in the universal cover $\tilde{M}$ of $M$. For $M$ compact, the limit exists and is independent of the base point $x \in \tilde{M}$ \cite{Mann}.  Thus, the volume entropy  measures the exponential growth rate of the volume of balls in the universal cover. It is non-zero if and only if the fundamental group $\pi_1(M)$ has exponential growth.  

When $\Ric_M \ge -(n-1)$, the Bishop-Gromov volume comparison gives the upper bound $h(M,g) \le  n-1$, which is the volume entropy of any hyperbolic $n$-manifold. Ledrappier-Wang \cite{Ledrappier-Wang2010} showed that if $h(M,g) =  n-1$, then $M$ is isometric to a hyperbolic manifold. This is called the maximal volume entropy rigidity. Liu found a simpler proof \cite{Liu2011}, and recently Chen-Rong-Xu gave a quantitative version of this rigidity result \cite{Chen-Rong-Xu}. 

In this paper we will show the same kind of maximal entropy rigidity holds for a class of metric measure spaces---known by now as \textit{$\RCDst(K,N)$ spaces}---that is of interest in both optimal transport and in the theory of limits of Riemannian manifolds with bounded Ricci curvature (known as {\em Ricci limit spaces}).

Alexandrov geometry can be seen as a synthetic approach to the spaces that occur as limits of smooth manifolds with sectional curvature bounded below. In this same spirit, $\RCDst(K,N)$ spaces can be thought of as the synthetic analog to Ricci curvature being bounded below by $K$, for dimension at most $N$. These spaces include Ricci limit spaces and Alexandrov spaces \cite{Pet}, and have been studied extensively, see Section 2 for details. 

The last-named author jointly with Mondino proved that the universal cover of an $\RCDst(K, N)$ space with $1<N<\infty$ exists and is also an $\RCDst(K, N)$ space \cite{MW}. This allows us to define the volume entropy similarly for compact $\RCDst(K, N)$ spaces.
 
That is, let $(X,d,m)$ be a compact $\RCDst(K, N)$ space, and  $(\tilde{X},\tilde{d}, \tilde{m})$ its universal cover.   We define the \textit{volume growth entropy}  of $(X,d,m)$ as
\[
h(X,d,m):= \limsup_{R \rightarrow \infty} \frac{1}{R} \ln \tilde{m} ( B_{\tilde{X}} (x,R)).
\]
The volume growth entropy is well defined, and it is independent of $x$ and the measure $m$, (see \cite{Rev, BCGS}). Observe that if $(M,g)$ is a Riemannian manifold then with the induced distance $d=d_g$ and  the volume measure $m=d{\rm vol}_g$, both definitions coincide. 

Our main results are: 
 
 \begin{thm}\label{thm-main1}
Let $1 < N < \infty$ and $(X, d, m)$ be a compact  $\RCDst(-(N-1), N)$ space. Then $h(X) \le N-1$. Furthermore,  the equality holds if and only if $N$ is an integer and the universal cover $(\tilde{X},\tilde{d}, \tilde{m})$ is isomorphic to the $N$-dimensional real hyperbolic space $(\mathbb H^{N}, d_{\mathbb H^N},  c_1{\mathcal H^{N}})$ for some $c_1 \in (0, \infty)$. Here ${\mathcal H^{N}}$ denotes the Hausdorff measure.
\end{thm}

As in the smooth case the compactness of $X$ is essential here.  For $N>1$ the well known smooth metric measure space $((0, \infty),\, |\, \cdot \, |, \,  \sinh^{N-1}(x) \, dx)$ is an $\mathsf{RCD}^*(-(N-1),N)$ space with volume entropy exactly $N-1$. This  example does not contradict our theorem as  it is not the universal cover of a compact $\mathsf{RCD}^*(-(N-1),N)$ space. 

The key step in proving the above theorem is the following result, which is of independent interest. In the statement we use the language of differential calculus developed by Gigli. We refer to Section \ref{Preliminaries} for definitions and more details.

\begin{thm}\label{thm-main2}
Let $1<N < \infty$ and $(X, d, m)$ be a complete $\RCDst(-(N-1), N)$ space. If there exists a function $u$ in $D_{loc}(\Delta)$ such that $|\nabla u| =1$ $m$-a.e. and   $\Delta u =N-1$, then $X$ is isomorphic to a warped product  space $\mathbb R \times_{e^t} X'$, where $X'$ is an $\RCDst (0, N)$ space. 
\end{thm}

An immediate consequence of  Theorem \ref{thm-main1} and the inequality provided in  \cite[Theorem 5]{Stu3},
$\sqrt {\lambda_{(\tilde X,d,m)}}  \leq  \frac 1 2 \limsup_{R \rightarrow \infty} \frac{1}{R} \ln {m} ( B_{{\tilde X}} (x,R))$,   is the following corollary. 

\begin{cor}
Let $1<N < \infty$ and $(X, d, m)$ be a compact $\RCDst(-(N-1), N)$ space.  If
\[ 
\lambda_{(\tilde{X},d,m)} := \inf \left\{\frac {\int_{\tilde X} |\nabla f|^2\, \mathrm{d}m} {\int_{\tilde X}  f^2 \, \mathrm{d}m} \mid f \in \mathrm{Lip}(\tilde X,d) \cap L^2(\tilde X,m), \, \int_{\tilde X}  f^2\, \mathrm{d}m\neq 0 \right\} = \frac{(N-1)^2}{4},
\]
then $\tilde X$ is isomorphic  to  the $N$-dimensional real hyperbolic space up to a scaling of the measure.
\end{cor}

The corresponding results for Alexandrov spaces have recently been proved by Jiang \cite{Jiang}.  

Rigidity results for $\RCDst$ spaces often imply almost rigidity results given that $\RCDst$ spaces are closed under measured Gromov-Hausdorff convergence. Theorem~\ref{thm-main1} implies an almost rigidity result assuming the volume entropy is continuous under measured Gromov-Hausdorff convergence, which is true when (the first) systole is uniformly bounded from below \cite[Proposition 38]{Rev}, c.f.  Proposition \ref{reviron-prop-38}. As a result we have: 

\begin{thm}  \label{stability}
	Let  $1< N < \infty$, $s>0,\ D>0$. There exists $\epsilon(N, s, D)>0$ such that for $0 < \epsilon < \epsilon(N, s, D)$, if $(X, d, m)$ is a compact  $\RCDst(-(N-1), N)$  space, satisfying ${\rm diam}(X) \le D, \ h(X) \ge N-1 -\epsilon,\  \rm{sys} (X, \sfd) \ge s$, then $X$ is homeomorphic and $\Psi(\epsilon | N, s, D)$ measured Gromov-Hausdorff close to an $N$-dimensional hyperbolic manifold.  
\end{thm}

 When $X$ is a Riemannian manifold this is proved without the systole condition in \cite[Theorem D]{Chen-Rong-Xu}, as the continuity of the entropy is proven for non-collapsing sequences of Riemannian manifolds with Ricci curvature bounded from below and diameter bounded from above converging to a manifold \cite[Theorem 0.5]{Chen-Rong-Xu}. The volume entropy is not necessarily continuous when the limit is a non-collapsing Ricci limit space, as the fundamental group could jump from having exponential growth for the sequence to  a trivial one for the limit space, see \cite[Remark 6.2]{Pan-Wei}.  We still conjecture that Theorem~\ref{stability} is true without the  systole  condition.

 The strategy and techniques used in proving our results are inspired by those of Gigli's Splitting Theorem  in the non-smooth context \cite{Gig}, as well as the ``Volume cone implies metric cone" Theorem by De Philippis-Gigli \cite{DePG}.
One of the key ideas for proving these results is to work at the level of the Sobolev spaces. 
In this way we overcome obstacles that appear due to the lack of analytical tools available in the smooth category. 
Once a result is obtained at this level it can be {\em transported} to a statement at the level of the metric measure space itself. 

We now present a summary of our strategy. In order to show that the universal cover $(\tilde{X}, \tilde{d}, \tilde{m})$ of  an $\RCDst(-(N-1), N)$ space $(X, d, m)$  with maximal volume entropy is isomorphic--- i.e. via a measure preserving isometry---to a real hyperbolic space (up to a scaling of the measure), it is sufficient to show that $\tilde{X}$ is isomorphic to a warped product space of the form $X'\times_{e^{t}} \mathbb{R}$, and then show that $X'$ is regular enough. 
At this point an analogy with \cite{DePG} becomes clear, as now our problem can be considered as a warped splitting theorem under the assumption of maximality of volume entropy.

To obtain a metric measure space which is a candidate for the role of $X'$, we reconstruct in our context Liu's ideas \cite{Liu2011} and build a Busemann-type function $u:\tilde{X}\to\mathbb{R}$ in $D_{loc}(\Delta)$, which is regular enough to admit a Regular Lagrangian Flow  $F:\mathbb{R}\times \tilde{X}\to \tilde{X}$ associated to $\nabla u$  (in the sense of Ambrosio-Trevisan \cite{AT}). The issue with the noncompact space here is dealt with by making use of the good cut-off functions of \cite{MN} and the local uniqueness of the Regular Lagrangian Flow. The trajectories $F_{(\cdot)}(x)$ of our flow induce a partition of $\tilde{X}$. The high regularity of $u$ provides useful information on how the reference measure $\tilde{m}$ changes under the flow. Still, the regularity of the Regular Lagrangian Flow takes serious effort, using the heat flow to regularize first and some uniform estimate following the ideas of \cite{DePG}. With the regularity issue addressed, an analysis of how the Cheeger energy of Sobolev functions changes once composed with the flow shows that a representative of $F$ can be chosen such that the maps $F_t$ are bi-Lipschitz. Then we proceed to obtain estimates of the local Lipschitz constants of F. 

Therefore, the natural candidate for $X'$ is $u^{-1}(0)$, the slice at time $0$ of the partition induced by $F$, endowed with the natural intrinsic metric and an appropriately defined measure which agrees with the data provided by $F$.   Given that $X' $ is non compact the measure defined on it is written in a similar way to \cite{Gig} and not as in \cite{DePG}. The proof that it is a complete, separable and geodesic space is more involved than in \cite{Gig} and \cite{DePG}. In \cite{Gig},  the distance in $X'$ can be seen as  the restriction of the metric of $\tilde d$ and in \cite{DePG} $X'$ is compact. We also have to show that $X'$ is locally doubling and not doubling as in \cite{DePG}.

At this point we need to show that the natural maps from and into $\tilde{X}$ and $\mathbb{R}\times_{e^t}X'$ are isomorphisms of metric measure spaces. As mentioned above, we obtain this at the level of the Sobolev spaces. The relation between the Sobolev spaces $W^{1,2}(\tilde{X})$ and $W^{1,2}(\mathbb{R}\times_{e^t} X')$ is explained by studying the metric speeds of curves in $\tilde{X}$ in relation with those in $X'$. This leads to a relationship between the minimal weak upper gradients of Sobolev functions in $X'$ and $\tilde{X}$.  Putting everything together, and combining them with the work of Gigli-Han \cite{GigHan} on the structure of Sobolev spaces of warped products, we obtain the desired isomorphism. 

Finally, the structure of a warped product space naturally implies via Bochner's inequality and a limiting argument that $X'$ is an $\RCDst(0,N)$ space. To complete the proof of the main theorem, we adapt Chen-Rong-Xu's argument \cite{Chen-Rong-Xu} and make use of the structure result of \cite{MN} to show that $\mathbb{R}\times_{e^t} X'$ is isomorphic to the $N$-dimensional hyperbolic space up to scaling of the measure. 

The article is organized as follows.   In Section \ref{Preliminaries},   we review definitions and properties of metric measure spaces and, in particular, $\RCDst$ spaces that will be needed in the paper.  In Section \ref{sec-Busemann}, using the Bishop-Gromov volume comparison theorem we provide the upper estimate of the volume entropy for $\RCDst(-(N-1),N)$ spaces. For the rigidity case, we construct the Busemann function $u$, calculate its Hessian and construct a Regular Lagrangian Flow $F$ associated to $\nabla u$. In Section \ref{sec-CheegerE} we estimate the minimal weak upper gradient of functions of the form  $ f \circ F_t$ for $f \in W^{1,2}(\tilde X, \tilde d, \tilde m)$. In the next section we use this to improve the regularity of the Regular Lagrangian Flow $F$, define the metric measure space quotient $(X',d',m')$ and estimate the minimal weak upper gradients of functions $g \in W^{1,2}(X')$ in terms of functions in $W^{1,2}(\tilde X)$. Moreover, we prove that $(X',d',m')$ is an infinitesimally Hilbertian space.  In Section  \ref{sec-Isom}  we use Gigli's Contraction By Local Duality Lemma, and his proposition on isomorphisms via duality with Sobolev norms, to show that the warped product space $\mathbb R \times_{e^t} X'$
 is isometric to $(\tilde X, \tilde d, \tilde m)$.  In Section \ref{sec-RCDcond} we prove that $(X',d',m')$ is an $\RCDst(0, N)$ space.  In the final section we see that $N \in \mathbb N$ and  $\mathbb R \times_{e^t} X'$ is isometric to the hyperbolic space $\mathbb H^N$, and prove the stability result, Theorem~\ref{stability}. 

On a complementary direction, the work of Besson-Courtois-Gallot \cite{BCGlong, BCG-acta} treated the {\it minimal} entropy of smooth manifolds and established major rigidity results for locally symmetric spaces of negative curvature.  Their work implies that negatively curved locally symmetric Riemannian metrics with given total volume cannot be perturbed to non-symmetric ones without increasing the volume entropy. A number of important corollaries in geometric rigidity and applications to dynamics then follow. We have also extended these barycenter techniques to $\RCDst$ spaces in \cite{RCD-Bary}.

\medskip

{\bf Acknowledgements.} The authors deeply thank Nicola Gigli for numerous very helpful communications on the theory of non-smooth differential geometry as a whole and, in particular, on his work on the Splitting Theorem in non-smooth context, and the ``Volume cone implies metric cone" theorem. We also thank Luigi Ambrosio and Dario Trevisan for explaining the main aspects of their work on the theory of Regular Lagrangian Flows on metric measure spaces, Christian Ketterer for explaining his work on cones over $\RCDst$ spaces, Lina Chen for communication on their work \cite{Chen-Rong-Xu}, Fabio Cavalletti for clarifying the equivalence between $\RCDst$ and $\RCD$ spaces, and Jaime Santos and Gerardo Sosa for commenting on this paper. This work was carried out while the third named author was visiting the Department of Mathematics of the University of California Santa Barbara and the Scuola Internazionale Superiore di Studi Avanzati. He would like to thank both institutions for their hospitality and excellent research conditions. The fifth named author thanks the warm hospitality of UC Santa Barbara Department of Mathematics during the period in which this work was produced. Finally we thank the anonymous referee for pointing out errors in the earlier version and very helpful comments and suggestions.

\tableofcontents


\section{Preliminaries}\label{Preliminaries}
The following is a review of the necessary definitions and results. First we recall the concepts pertaining to first order calculus on metric spaces, we refer readers to \cite{Gig15, Gig} for further details. 

\subsection{Calculus on metric measure spaces} \label{ssec-calculus}

We will consider a proper metric space $(X,d)$. Let $C([0,1];X)$ be the set of continuous curves in $(X,d)$. A curve $\gamma\in C([0,1];X)$ is said to be \textit{absolutely continuous} if there exists an integrable function $f$ on $[0,1]$ such that for every $0 \leq t<s \leq 1$,
\[
d(\gamma_t,\gamma_s)\leq \dint_t^s \! f(r)\, \mathrm{d}r.
\]
Absolutely continuous curves $\gamma$ have a well defined \textit{metric speed},
\[
|\dot{\gamma}_t|:=\lim_{t\to 0}\frac{d(\gamma_{t+h},\gamma_t)}{|h|},
\]
which is a function in $L^1([0,1]).$  The set of absolutely continuous curves in $(X,d)$ will be denoted by  $AC([0,1];X)$. 

Let $m$ be a non-negative Radon measure on $X$ and 
$\mathcal{P}(C([0,1];X))$ be the space of probability measures on $C([0,1];X)$. A measure $\pi\in \mathcal{P}(C([0,1];X))$ is called a \textit{test plan} if there exists $C> 0$ such that for every $t\in [0,1]$, 
\[
(e_t)_{\sharp}\pi\leq Cm
\]
and
\[
\dint\!\dint_{0}^1|\dot{\gamma}_t|^2 \, \mathrm{d}t\, \mathrm{d}\pi(\gamma)<\infty.
\]
Here, $e_t:C([0,1];X)\to X$ is the evaluation map $e_t(\gamma)=\gamma_t$. 

\bigskip

The \textit{Sobolev class} $S^2(X):=S^2(X,d,m)$ (respectively $S^2_{loc}(X):=S^2_{loc}(X,d,m)$) is the space of all Borel functions $f:X\to\mathbb{R}$ such that there exists a non-negative function $G\in L^2(X):=L^2(X,m)$ (respectively $G\in L^2_{loc}(X):=L^2_{loc}(X,m)$)---called \textit{weak upper gradient}---such that for any test plan $\pi$ the following inequality is satisfied
\[
\dint|f(\gamma_1)-f(\gamma_0)|\, \mathrm{d}\pi(\gamma)\leq \dint\!\!\!\dint_0^1 G(\gamma_t)|\dot\gamma_t| \, \mathrm{d}t \, \mathrm{d}\pi(\gamma).
\]
 
It is possible to prove that there exists a minimal $G$, which we denote by $|\nabla f|$, called the \textit{minimal weak upper gradient} of $f$. We now recall the following fundamental result.

\begin{prop}\cite[Definition 5.6, Proposition 5.7]{AGS14}
, \cite[Definition B.2, Theorem B.4]{Gig15}
\label{prop-weakUg}
Let $f,G: X  \to \mathbb R$  be two functions. The following are equivalent, 
\begin{itemize}
\item[(i)] $f \in S^2(X)$ and $G$ is a weak upper gradient.
\item[(ii)] For every test plan $\pi$ the following holds: For $\pi$-a.e. $\gamma$ the function $t \mapsto f(\gamma_t)$ is equal at $t=0$, $t=1$ and almost everywhere else on $[0,1]$ to an absolutely continuous function $f_\gamma:[0,1]\to\mathbb{R}$ whose derivative  for a.e. $t \in [0,1]$ satisfies $|f_\gamma'|(t) \leq G(\gamma_t) |\dot\gamma_t|$. 
\end{itemize}
\end{prop}

A local version of the Sobolev class is produced in the following manner: A function $f:\Omega\subset X \to \mathbb{R}$, with $\Omega$ an open set, is an element of $S^2_{loc}(\Omega):=S^2_{loc}(\Omega,d,m)$ if for any Lipschitz function $\chi:X\to\mathbb{R}$ with $\mathrm{supp}(\chi)\subset \Omega$
we have that $f\chi\in S^2_{loc}(X)$. In this case $|\nabla f|:\Omega\to\mathbb{R}$ is given by 
\[
|\nabla f|:=|\nabla(f \chi)| \ \ m- \,a.e. \ {\rm on} \ \chi=1. 
\]
Then, the set $S^2(\Omega)$ is defined as the subset of $S^2_{loc}(\Omega)$ of functions $f$ such that $|\nabla f|\in L^2(\Omega,m)$.

The \textit{Sobolev space} is defined as
\[
W^{1,2}(X,d,m):=L^2(X,m)\cap S^2(X,d,m) 
\]
 endowed with the norm
\[
||f||^{2}_{W^{1,2}(X)}:= ||f||^{2}_{L^2(X)} + |||\nabla f|||^{2}_{L^2(X)}= \dint_{X} (f^2 + |\nabla f|^2) \, \mathrm{d}m.
\]

We say that a proper metric measure space $(X,d,m)$ is \textit{infinitesimally Hilbertian} if $W^{1,2}(X)$ is a Hilbert space, i.e., if  $||\cdot ||^{2}_{W^{1,2}(X)}$ is induced by an inner product. This happens if and only if the parallelogram rule is satisfied, that is
\[
|||\nabla(f+g)|||^{2}_{L^2(X)} + |||\nabla(f-g)|||^{2}_{L^2(X)} = 2\left( |||\nabla f|||^{2}_{L^2(X)} + |||\nabla g|||^{2}_{L^2(X)}\right)
\]
for all $f,g\in S^2(X)$.
On an infinitesimally Hilbertian metric measure space $(X,d,m)$, for $\Omega\subset X$ open  and any $f,g\in S^2_{loc}(\Omega)$ the functions $D^{\pm}:\Omega\to \mathbb{R}$ defined $m$-a.e. by 
\begin{eqnarray*}
D^+f(\nabla g)=\inf_{\varepsilon >0} \frac{|\nabla(g+\varepsilon f)|^2-|\nabla g|^2}{2\varepsilon}\\
D^-f(\nabla g)=\sup_{\varepsilon >0} \frac{|\nabla(g+\varepsilon f)|^2-|\nabla g|^2}{2\varepsilon}
\end{eqnarray*}
coincide $m$-a.e. on $\Omega$. We denote the common value by $\left\langle \nabla f, \nabla g\right\rangle$. 

An important tool is the following \textit{first differentiation formula} (see (1.11) in \cite{Gig15}). Recall that a test plan $\pi$ is said to represent the gradient of $f\in S^2(X)$ if 
\[
\liminf_{t\downarrow 0} \dint\frac{f(\gamma_t)-f(\gamma_0)}{t}\,\mathrm{d}\pi(\gamma)\geq \frac{1}{2} \dint |\nabla f|^2(\gamma_0)\, \mathrm{d}\pi(\gamma) + \frac{1}{2} \limsup_{t\downarrow 0} \dint\dint_0^t |\dot{\gamma}_s|^2\,\mathrm{d}s\,\mathrm{d}\pi(\gamma).
\]
In the case that $f\in S^2(\Omega)$ for some open set $\Omega\subset X$, one adds to the definition the requirement that $(e_t)_{\#}\pi$ is concentrated on $\Omega$ for every $t\in [0,1]$ sufficiently small. Then, given $f,g\in S^2(\Omega)$ with $\Omega \subset X$ open and a test plan $\pi$ representing the gradient of $f$ it holds that
\begin{equation}
\label{eq-first-differentiation-formula}
\lim_{t\downarrow 0} \dint_{X} \frac{g(\gamma_t)-g(\gamma_0)}{t}\, \mathrm{d}\pi(\gamma) = \dint_{X} \left\langle \nabla f, \nabla g \right\rangle(\gamma_0)\, \mathrm{d}\pi(\gamma). 
\end{equation}

Let $(X,d,m)$ be an infinitesimally Hilbertian metric measure space and $\Omega\subset X$ an open set. Let $g: \Omega \to \mathbb{R}$ be a locally Lipschitz function. We say that $g$ \textit{has a measure valued Laplacian}, provided there exists a Radon measure $\mu$ on $\Omega$ such that 
\[
-\dint_{\Omega} \left\langle \nabla f, \nabla g\right\rangle \, \mathrm{d}m = \dint_{\Omega} f\, \mathrm{d}\mu
\] 
for all $f: \Omega \to \mathbb R$ Lipschitz and compactly supported in $\Omega$. In this case $\mu$ is the measure valued Laplacian of $g$, and it is denoted by $\mathbf{\Delta} g|_{\Omega}$. The set of all locally Lipschitz functions $g$ admitting a measure valued Laplacian is denoted by $D(\mathbf{\Delta},\Omega)$. A particular instance of the notation is that $D(\mathbf{\Delta}, X)= D(\mathbf{\Delta})$ and then $\mathbf{\Delta} g|_{X}=\mathbf{\Delta} g$.

A different definition is that of the $L^2$-Laplacian operator defined as follows. The domain $D(\Delta)$ of the $L^2$-Laplacian is the subset of $W^{1,2}(X)$ of all $g$ such that for some $h\in L^2(X)$,
\begin{equation} \label{Lap}
-\dint\left\langle \nabla f, \nabla g\right\rangle \, \mathrm{d}m = \dint fh \, \mathrm{d}m
\end{equation} 
 for all $f\in W^{1,2}(X)$, written as $\Delta g=h$. Both definitions agree in the sense that $g\in D(\Delta)$ if and only if $g\in W^{1,2}(X)\cap D(\mathbf{\Delta})$ and $\mathbf{\Delta} g= h\, m$ with $h \in L^2(X)$ (see \cite[Definition 4.6]{Gig}). We similarly define $D_{loc}(\Delta)$ to be the corresponding subset of $W^{1,2}_{loc}(X) := L_{loc}^2(X,m) \cap S_{loc}^2 (X,d,m)$, namely the subset of all $g\in W^{1,2}_{loc}(X)$ such that (\ref{Lap})  holds for all $f\in {\rm Test}_{bs}(X)$. Here  
 \begin{equation}  \label{eq-test}
 \mathrm{Test}(X):= \left\{f\in D(\Delta)\cap L^{\infty}(X,m) \mid |\nabla f|\in L^{\infty}(X,m)\ \text{and} \ \Delta f\in W^{1,2}(X)  \right\},
 \end{equation}
 and $\mathrm{Test}_{bs}(X)$ is the subset of  $\mathrm{Test}(X)$ consisting of functions with bounded support.

\subsection{Tangent and cotangent modules} We will now give a brief account of some of the tools of the tangent and cotangent modules as defined and developed in detail by Gigli \cite{Gig2} (see also the section on preliminaries of \cite{DePG}).

Given an infinitesimally Hilbertian metric measure space $(X,d,m)$, recall that there is a unique couple $(L^2(T^*X),d)$ (up to isomorphism) where $L^2(T^*X)$ is an $L^2(m)$-normed $L^{\infty}(m)$-module (see \cite[Definition 1.2.10]{Gig2}) and $d:S^2(X)\to L^2(T^*X)$ is a linear operator such that  the following two conditions hold
\begin{itemize}
\item[(i)] $|df| = |\nabla f|$ $m$-a.e.\ for every $f\in S^2(X)$. Here $|df|$ denotes the pointwise norm of $df$ in $L^2(T^*X)$.
\item[(ii)] $L^2(T^*X)$ is spanned by $\{ df \mid f\in S^2(X)\}$. 
\end{itemize}

The module $L^2(T^*X)$ is called the \textit{cotangent module of $X$} and $d$ is the \textit{differential}. Note that we abuse the notation slightly by using $d$ for the differential of a function and the distance of the space. 

The tangent module of $X$, denoted by $L^2(TX)$ is defined as the dual module of $L^2(T^*X)$ and the \textit{gradient} $\nabla f\in L^2(TX)$ of a function $f\in W^{1,2}(X)$ is the unique element associated to $df$ via the Riesz isomorphism. 

Let $(Y,d_Y,m_Y)$ be a metric measure space. We will say that a map $F:Y\to X$ has \textit{bounded compression} if $F_{\sharp}m_Y\leq Cm$ for some $C>0$. Given an $L^2$-normed $L^{\infty}$-module $\mathcal{M}$ over $X$, the \textit{pullback module $F^*\mathcal{M}$} is an $L^2$-normed $L^{\infty}$-module over $Y$  carrying a \textit{pullback operator} $F^*:\mathcal{M}\to F^*\mathcal{M}$ defined (uniquely up to isomorphism) in the following way: $F^*$ is linear and satisfies the following,
\begin{itemize}
\item[(i)] $|F^*v|=|v|\circ F$, $m_Y$-a.e. for all $v\in \mathcal{M}$,
\item[(ii)] $\left\{F^*v \mid v\in \mathcal{M} \right\} $ generates $F^*\mathcal{M}$ as a module. 
\end{itemize}  
Denote by $\mathcal{M}^*$ the dual module of $\mathcal{M}$. Then, we have the \textit{unique duality relation}, 
\[
F^*\mathcal{M}^*\times F^*\mathcal{M}\to L^{1}(Y,m_Y), 
\]
which is $L^{\infty}(Y)$-bilinear, continuous and satisfies 
\[
F^*w(F^*v)=w(v)\circ F, \  \text{$m_Y$-a.e. for all}\ v\in \mathcal{M}, w\in\mathcal{M}^*.
\]
For $\mathcal{M}=L^2(T^*X)$ (respectively $\mathcal{M}=L^2(TX)$) the pullback is denoted by $L^2(T^*X, F, m_Y)$ (respectively $L^2(TX, F, m_Y)$). A special instance of this construction occurs when $Y=C([0,1];X)$ equipped with the sup distance and a test plan $\pi$ as reference measure. The evaluation maps $e_t$ have bounded compression and there exists a unique element $\pi'_t\in L^2(TX, e_t, \pi)$ such that 
\[
\lim_{h\to 0} \frac{f\circ e_{t+h}-f\circ e_t}{h}= (e_t^*¨df)(\pi'_t)
\]
for all $f\in W^{1,2}(X)$,  where the limit is intended in the strong topology of $L^1(C([0,1];X)),\pi)$, that is, the space of integrable functions on $C([0,1];X))$ with respect to the test plan $\pi$ (see \cite[Theorem 2.3.18]{Gig2}). It follows from this result that for $\pi$-a.e. $\gamma$ and a.e. $t\in [0,1]$,  
\[
|\pi'_t|(\gamma)=|\dot{\gamma}_t|.
\]


\subsection{$\mathsf{CD}^*(K,N)$ and $\RCDst(K,N)$-spaces}

Here we briefly recall the synthetic notions of lower Ricci curvature bounds on metric measure spaces. 

A notion of metric measure spaces with Ricci curvature bounded below by $K\in \mathbb{R}$ and dimension bounded above by $N\in (1,\infty]$ was first considered in the setting of Optimal Transport Theory by Lott-Sturm-Villani \cite{LotVil, Stu1, Stu2}, resulting in the class of spaces with the \textit{curvature dimension condition} or briefly \textit{$\mathsf{CD}(K,N)$ spaces}.  It was then proved by Ohta that smooth compact Finsler manifolds are $\mathsf{CD}$ spaces \cite{Oht}. In contrast, a Finsler manifold can only arise as a limit of Riemannian manifolds with Ricci curvature uniformly bounded below if and only if it is Riemannian. Recall that a Finsler manifold is Riemannian if and only if the Cheeger energy is quadratic or, equivalently, if the heat flow is linear. 

To address the problem of isolating the class of Riemannian-like $\mathsf{CD}$-spaces, Gigli  \cite{Gig15} proposed to reinforce the definition of a $\mathsf{CD}(K,N)$ space $(X,d,m)$ with the functional-analytic condition of \textit{infinitesimal Hilbertianity}, that is, that the Sobolev space $W^{1,2}(X,d,m)$ is a Hilbert space (see Definition \ref{def-RCD}). 
This definition came out as a result of a program initiated by Gigli \cite{Gig10},  further developed by Gigli-Kuwada-Ohta \cite{GKO} and Ambrosio-Gigli-Savar\'e \cite{AGS14}, with the aim of investigating the heat flow on metric measure spaces and the introduction of $\mathsf{RCD}(K,\infty)$ spaces \cite{AGSrcd, AGMR}. The finite dimensional case, i.e. $\mathsf{RCD}(K,N)$ for $N \in (1, \infty)$  was then analyzed
independently in \cite{EKS} and \cite{AMS}).

At the emergence of $\mathsf{CD}(K,N)$ spaces, it was not clear whether this class exhibited a \textit{local-to-global} property, i.e. whether satisfying $\mathsf{CD}(K,N)$ for all subsets of a covering implies the condition on the full space. To address this issue, Bacher-Sturm introduced an apriori slightly weaker condition of Ricci curvature bounded below by $K$ with dimension at most $N$, namely the \textit{reduced curvature-dimension condition} or $\mathsf{CD}^*(K,N)$ \cite{BaSt}.

 To state the definitions and results in this section, we begin by recalling the so called \textit{distortion coefficients}. Given $K,N\in \mathbb R$ with $N\geq0$,  for $(t,\theta) \in[0,1] \times \mathbb R_{+}$ we define 
\begin{equation}
\label{E:sigma}
\sigma_{K,N}^{(t)}(\theta):= 
\begin{cases}
\infty, & \textrm{if}\ K\theta^{2} \geq N\pi^{2}, \crcr
\displaystyle  \frac{\sin(t\theta\sqrt{K/N})}{\sin(\theta\sqrt{K/N})} & \textrm{if}\ 0< K\theta^{2} <  N\pi^{2}, \crcr
t & \textrm{if}\ K \theta^{2}<0 \ \textrm{and}\ N=0, \ \textrm{or  if}\ K \theta^{2}=0,  \crcr
\displaystyle   \frac{\sinh(t\theta\sqrt{-K/N})}{\sinh(\theta\sqrt{-K/N})} & \textrm{if}\ K\theta^{2} \leq 0 \ \textrm{and}\ N>0.
\end{cases}
\end{equation}

For $N\geq 1, K \in \mathbb R$ and $(t,\theta) \in[0,1] \times \mathbb R_{+}$ we define
\begin{equation} \label{E:tau}
\tau_{K,N}^{(t)}(\theta): = t^{1/N} \sigma_{K,N-1}^{(t)}(\theta)^{(N-1)/N}.
\end{equation}

Let $\mathcal{P}_{2}(X,d,m)$ denote the family of probability measures with finite second moment, ${\rm Opt}(\mu_{0},\mu_{1})$ the set of optimal transports between $\mu_0$ and $\mu_1$ and $\mathrm{Geo}(X)$ the set of geodesics of $X$. 

\begin{defn}[$\mathsf{CD}$ condition]\label{D:CD}
A metric measure space $(X,d,m)$ is a $\mathsf{CD}(K,N)$ space if for each pair 
$\mu_{0}, \mu_{1} \in \mathcal{P}_{2}(X,d,m)$ there exists $\pi \in {\rm Opt}(\mu_{0},\mu_{1})$ such that
\begin{equation}\label{E:CD}
\rho_{t}^{-1/N} (\gamma_{t}) \geq 
 \tau_{K,N}^{(1-t)}(d( \gamma_{0}, \gamma_{1}))\rho_{0}^{-1/N}(\gamma_{0}) 
 + \tau_{K,N}^{(t)}(d(\gamma_{0},\gamma_{1}))\rho_{1}^{-1/N}(\gamma_{1}), \qquad \pi\text{-a.e.} \, \gamma \in \mathrm{Geo}(X),
\end{equation}
for all $t \in [0,1]$, where $\rho_t$ is such that $({e}_{t})_\sharp \, \pi = \rho_{t} m$.
\end{defn}

It is worth remembering here that for a Riemannian manifold $(M,g)$ of dimension $n$ and 
$h \in C^{2}(M)$ with $h > 0$, the metric measure space $(M,g,h \,d{\rm vol}_g)$ verifies condition $\mathsf{CD}(K,N)$ with $N\geq n$ if and only if  (see Theorem 1.7 of \cite{Stu2})
$$
{\rm Ric}_{g,h,N} \geq  K g, \qquad {\rm Ric}_{g,h,N} : =  {\rm Ric}_{g} - (N-n) \frac{\nabla_{g}^{2} h^{\frac{1}{N-n}}}{h^{\frac{1}{N-n}}}.  
$$
Here, we follow the convention that if $N = n$ the generalized Ricci tensor $Ric_{g,h,N}= Ric_{g}$ makes sense only if $h$ is constant. 

The reduced $\mathsf{CD}^{*}(K,N)$ condition requires the same inequality \eqref{E:CD} of $\mathsf{CD}(K,N)$ but with the coefficients $\tau_{K,N}^{(t)}(\sfd(\gamma_{0},\gamma_{1}))$ and $\tau_{K,N}^{(1-t)}(d(\gamma_{0},\gamma_{1}))$ replaced by $\sigma_{K,N}^{(t)}(\sfd(\gamma_{0},\gamma_{1}))$ and $\sigma_{K,N}^{(1-t)}(d(\gamma_{0},\gamma_{1}))$, respectively. Hence while the distortion coefficients of the $\mathsf{CD}(K,N)$ condition 
are formally obtained by imposing one direction with linear distortion and $N-1$ directions affected by curvature, 
the $\mathsf{CD}^{*}(K,N)$ condition imposes the same volume distortion in all the $N$ directions.

Now we will recall the \textit{generalized Bishop-Gromov comparison theorem} for $\mathsf{CD}^*(K,N)$-spaces with $K<0$. Let $B(x,R)$ be the metric ball around $x$ with radius $R$ and we denote its metric closure by $\overline{ B(x,R)}$. Note that the sharp version of this result is valid for $\mathsf{CD}^*(K,N)$ spaces as a consequence of \cite[Theorem 1.1]{CavStu} and \cite[Theorem 5.1]{Ohta}.

\begin{thm}[Generalized Bishop-Gromov volume growth inequality for $\mathsf{CD}^*(K,N)$]
\label{thm-BishopGrom}
Assume that the metric space $(X,d,m)$ satisfies the $\mathsf{CD}^*(K,N)$-condition for some $K<0$ and $N\in [1,\infty)$. Then  for all $r\leq R$,
\[
\frac{m ( \overline{B(x,r)})}{m (\overline{ B(x,R)})} \geq \frac{\int_{0}^{r}\sinh^{N-1} (\sqrt{-K/(N-1)}t) \  dt}{\int_{0}^{R}\sinh^{N-1}(\sqrt{-K/(N-1)}t) \  dt} .
\]

Furthermore, for the function $s_m(x,r) = \limsup_{\delta \to 0} \frac{1}{\delta} m ( \overline{B(x,r+ \delta)}  \setminus B(x,r))$ the following inequality holds
\[
\frac{s_m (x,r)}{s_m (x,R)}      \geq      \frac{\sinh^{N-1} (\sqrt{-K/(N-1)}r) }{\sinh^{N-1}(\sqrt{-K/(N-1)}R) }.
\]
\end{thm}

We now recall the definition of the reduced Riemannian curvature-dimension condition.

\begin{defn}[$\RCDst$ condition]\label{def-RCD}
A metric measure space $(X,d,m)$ is an $\RCDst(K,N)$ space if it is an infinitesimally Hilbertian $\mathsf{CD}^*(K,N)$ space.
\end{defn}

Cavalletti-Milman have shown the equivalence of the $\mathsf{CD}$ and $\mathsf{CD}^*$ conditions when the space is essentially non-branching and has finite measure \cite[Corollary 13.7]{CavMil}. In particular under the assumption of finite measure, $\mathsf{RCD}(K,N)$ is equivalent to $\RCDst(K,N)$. It is expected that $\mathsf{RCD}(K,N)$ is equivalent to $\RCDst(K,N)$ without any further assumptions. 

 Now we state the \textit{Laplacian comparison for distance functions} originally proved by Gigli for $\mathsf{CD}(K,N)$ spaces \cite[Corollary 5.15]{Gig15} with some extra assumption and shown to hold sharply on essentially nonbranching  $\mathsf{CD}^*(K,N)$ spaces (and more generally on $\mathsf{MCP}(K,N)$ spaces) in \cite{CavMon}. We will use this result in the following section. For simplicity we only state the result for $K<0$.  

\begin{thm}[Laplacian comparison for distance functions]
Let $K < 0$, $N \in (1, \infty)$, and $(X,d,m)$ be an $\mathsf{RCD}^*(K,N)$ space. Let $r: X \to \mathbb R$ be the function given by $r(x)=d(x,o)$, where $o \in X$.  Then
$r \in D(\mathbf{\Delta}, X \setminus \{o\})$ and
\begin{equation}\label{eq-lapComp}
\mathbf{\Delta} r|_{X\setminus \{o\}} \leq    \sqrt{-K (N-1)} \coth(\sqrt{-K /(N-1)}r)m.  
\end{equation}
\end{thm}

A useful tool for localization is a system of ``good'' cut-off functions. One such characterization is the following.
\begin{lem}[Lemma 3.1 of \cite{MN}]\label{lem:good-cut-off}
	Let $(X, d, m)$ be an $\mathsf{RCD}^{*}(K, N)$ space for some $K \in \mathbb{R}$ and $ N \in(1, \infty)$
	Then for every $x \in X, R>0, 0<r<R$ there exists a Lipschitz function $\rho=\rho^{r}: X \rightarrow \mathbb{R}$
	satisfying:
\begin{enumerate}
	\item $0 \leq \rho^{r} \leq 1$ on $X, \rho^{r} \equiv 1$ on $B(x,r)$ and $\op{supp} \rho^{r} \subset B(x,2r)$
	\item $r^{2}\left|\Delta \rho^{r}\right|+r\left|D \rho^{r}\right| \leq C(K, N, R)$.
\end{enumerate}
\end{lem}

Observe that for any compact set $K$ contained in an open set $U$ we can apply this lemma to find a cut-off function  $\rho_K \in W^{1,2}(X)$ which is Lipschitz, identically $1$ on $K$, identically $0$ on $X \backslash U$ and such that $\rho_K \in D(\mathbf{\Delta})$ with $\mathbf{\Delta} \rho_K \ll \mathrm{m}$ with bounded density. This version is formulated in Theorem 3.12 of \cite{GM}.

In order to introduce the notion of Hessian we define $\mathrm{Test}_{loc}(X)$ as the set of functions $f:X\to \mathbb{R}$ with the following property: For every bounded Borel set $B\subseteq X$, there exists a function $f_B\in \mathrm{Test}(X)$ such that $f_B=f$ $m$-a.e. in $B$. It is clear that $\mathrm{Test}_{bs}(X)\subset \mathrm{Test}(X)\subset \mathrm{Test}_{loc}(X)$. 

An important fact is that if $X$ satisfies $\RCDst(K,N)$ then $\mathrm{Test}_{bs}(X)$ is dense in $W^{1,2}(X)$. Furthermore, if $f\in \mathrm{Test}_{loc}(X)$ then $|\nabla f|^2\in W^{1,2}_{loc}(X)$ and by polarization, for every $f,g\in \mathrm{Test}_{loc}(X)$ we have that $\left\langle \nabla f,\nabla g\right\rangle \in W^{1,2}_{loc}(X)$ (see for example \cite[Proposition 3.1.3]{Gig2}). Moreover we have the following characterization.

\begin{lem} \label{lem:test:loc}
The set $\mathrm{Test}_{loc}(X)$ admits the description,
\begin{equation}\label{eq-test}
 \mathrm{Test}_{loc}(X)= \left\{f\in D_{loc}(\Delta)\cap L^{\infty}_{loc}(X,m) \mid |\nabla f|\in L^{\infty}_{loc}(X,m)\ \text{and} \ \Delta f\in W^{1,2}_{loc}(X)  \right\},
 \end{equation}
\end{lem}

\begin{proof}
Let $\mathrm{Test}_{loc}'(X)$ denote the right hand side of the above expression. For $f\in \mathrm{Test}_{loc}(X)$ from the definition and the discussion above we have $f\in D_{loc}(\Delta)\cap L^{\infty}_{loc}(X,m)$, $|\nabla f|\in L^{\infty}_{loc}(X,m)$ and $\Delta f\in W^{1,2}_{loc}(X)$. Since on each compact set $K$, $f$ agrees with $f_K$ when restricted to $K$ and these containments hold for $f_K\in \mathrm{Test}(X)$, we obtain $f\in \mathrm{Test}_{loc}'(X)$. 

Conversely, if $f\in \mathrm{Test}_{loc}'(X)$ then by Lemma \ref{lem:good-cut-off} for any compact set $K$ contained in an open set $U$ there exist a ``good''  cut-off function $\rho_K \in W^{1,2}(X)$ which is Lipschitz, identically $1$ on $K$, identically $0$ on $X \backslash U$ and such that $\rho_K \in D(\mathbf{\Delta})$ with $\Delta \rho_K \ll \mathrm{m}$ with bounded density. For any bounded Borel set $B$ we let $K=\overline{B}$ and define $f_B=\rho_K f$, where $\rho_K$ is the cut-off function for $K$ and any bounded open set $U\supset K$. By the Leibniz rule \cite[(3.9)]{Gig} we have $f_B\in \mathrm{Test}(X)$ and thus $f\in \mathrm{Test}_{loc}(X)$.
\end{proof}

For a function $u\in \mathrm{Test}_{loc}(X)$ we define the \textit{Hessian} of $u$ 
\[ 
\mathrm{Hess}[u]: \mathrm{Test}_{loc}(X)\times \mathrm{Test}_{loc}(X)\to L_{loc}^2(X,m),
\]
by the following expression
\begin{equation}\label{Hu}
\mathrm{Hess}[u](f,g):=\frac{1}{2}\left(  \left\langle \nabla f, \nabla\left\langle \nabla u, \nabla g\right\rangle \right\rangle + \left\langle \nabla g, \nabla\left\langle \nabla u, \nabla f\right\rangle \right\rangle - \left\langle \nabla u, \nabla\left\langle \nabla f, \nabla g\right\rangle \right\rangle \right).
\end{equation}
We note that this is a symmetric bilinear operator and it restricts to
\[ 
\mathrm{Hess}[u]: \mathrm{Test}_{bs}(X)\times \mathrm{Test}_{bs}(X)\to L^2(X,m).
\]

The space $W^{2,2}_{loc}(X)$ consists of the functions $f\in W^{1,2}_{loc}(X)$ such that for any $g_1$, $g_2$, $h\in \mathrm{Test}_{bs}(X)$,  there exists an $A\in L^2(T^*X)\otimes L^2(T^*X)$ such that

\[
2\int hA(\nabla g_1,\nabla g_2)\, \mathrm{d}m =\int -\left\langle \nabla f, \nabla g_1\right\rangle \mathrm{div}(h\nabla g_2)-\left\langle \nabla f, \nabla g_2\right\rangle \mathrm{div}(h\nabla g_1)-h\left\langle \nabla f, \nabla \left\langle \nabla g_1, \nabla g_2\right\rangle\right\rangle\, \mathrm{d}m.
\]
There is a unique such $A$ in $L^2(T^*X)\otimes L^2(T^*X)$ which is denoted by $\mathrm{Hess}(f)$ (see \cite[Section 1.5]{Gig2} for details). A very important result \cite[Theorem 3.3.8]{Gig2} states that $\mathrm{Test}(X)\subset W^{2,2}(X)$ and that for every $g_1,g_2\in \mathrm{Test}(X)$, 
\begin{equation}
\label{Hu-Hess}
\mathrm{Hess}[f](g_1,g_2)= \mathrm{Hess}(f)(\nabla g_1,\nabla g_2).
\end{equation}

It can be readily checked that $\mathrm{Test}_{loc}(X)\subset W^{2,2}_{loc}(X)$ as well, and that \eqref{Hu-Hess} is still valid for $f\in \mathrm{Test}_{loc}(X)$ and every $g_1,g_2\in \mathrm{Test}_{bs}(X)$.

The notion of \textit{divergence} of a vector field is defined as follows. Recall that $L^2_{loc}(TX)$ consists of those vector fields $V$ such that $|V|\in L^2_{loc}(X,m)$. We say that $V\in L^2_{loc}(TX)$ has a divergence in $L^2_{loc}$ and denote it by $V\in D_{loc}(\mathrm{div})$ if there exists $h\in L^2_{loc}(X,m)$ such that for every $f\in \mathrm{Test}_{bs}(X)$ it holds that 
\[
\int fh\,\mathrm{d}m =-\int df(V)\,\mathrm{d}m.
\]
In this case we write $\mathrm{div}V=h$.


\subsection{Bakry-\'Emery condition and Bochner's inequality}
\label{ss-bochner}
We begin this section by recalling the weak version of Bochner's inequality obtained by Ambrosio-Mondino-Savare \cite{AMS} and Erbar-Kuwada-Sturm \cite{EKS}.

\begin{thm}[Weak Bochner's inequality \cite{EKS, AMS}]
\label{thm-weak-bochner-inequality}
Let $(X,d,m)$ be an $\RCDst(K,N)$-space. Then, for all $f\in D(\Delta)$ with $\Delta f\in W^{1,2}(X,d,m)$ and all $g\in D(\Delta)\cap L^{\infty}(X,m)$ non-negative with $\Delta g\in L^{\infty}(X,m)$ we have
\begin{equation}\label{eq-Bochner}
\frac{1}{2}\dint\Delta g |\nabla f|^2\, \mathrm{d}m - \dint g\left\langle\nabla(\Delta f),\nabla f\right\rangle\, \mathrm{d}m\geq K\dint g|\nabla f|^2\, \mathrm{d}m + \frac{1}{N}\dint g(\Delta f)^2\,\mathrm{d}m.
\end{equation}
\end{thm}

A remarkable property is the equivalence of the $\mathsf{RCD}^{*}(K,N)$ condition and the Bochner inequality under some conditions (namely the Sobolev to Lipschitz property---which we recall below---and a certain volume growth estimate). The infinite dimensional case was settled in \cite{AGSrcd}, while the (technically more involved) finite dimensional refinement was established in \cite{EKS} and \cite{AMS}. 

Let $f,g \in \mathrm{Test}_{loc}(X)$ and define the measure-valued map 
\[
\Gamma_2(f,g):= \frac{1}{2}\mathbf{\Delta}\left\langle \nabla f, \nabla g\right\rangle -\frac{1}{2}\left(\left\langle \nabla f, \nabla \Delta g\right\rangle + \left\langle \nabla g, \nabla \Delta f\right\rangle \right)m.
\]
Let $\Gamma_2(f):=\Gamma_2(f,f)$. It was shown by Ambrosio-Mondino-Savar\'e  \cite{AMS} and Erbar-Kuwada-Sturm \cite{EKS} that the following non-smooth Bakry-\'Emery condition is satisfied on an $\RCDst(K,N)$-space:  For every $f\in \mathrm{Test}(X)$ 
\begin{equation}
\label{eq-bakry-emery-condition}
\Gamma_2(f)\geq \left(K|\nabla f|^2 + \frac{1}{N}(\Delta f)^2 \right)m.
\end{equation}

It follows immediately from the definition of local test functions that \eqref{eq-bakry-emery-condition} is satisfied as well for every $f\in \mathrm{Test}_{loc}(X)$. 
 
 Now we state a fundamental technical tool (see \cite{Sav}) which is useful when ``changing variables''. For simplicity, we state a weaker version than that in \cite{Sav}, suitable for our purposes. This result follows from the fact that $\mathrm{Test}_{loc}(X)$ is an algebra and from the Chain and Leibniz Rules for differentiation. 
 
\begin{prop}[\cite{Sav}]
\label{prop:change_of_variables}
Let $n\in \mathbb{N}$ and $\Psi:\mathbb{R}^n\to \mathbb{R}$ be a polynomial with no constant term. Let us fix  $f_1,\ldots, f_n \in \mathrm{Test}_{loc}(X)$ and denote $\Psi(f):=\Psi(f_1,\ldots,f_n):X\to\RR$ and $\Psi_{ij}:=\partial_{ij}\Psi$. Then, $\Psi(f)$ is in $\mathrm{Test}_{loc}(X)$ and the following formulae  hold true

\begin{itemize}
\setlength\itemsep{1em}
\item[(i)] 
$\begin{aligned}[t]
\left| \nabla \Psi(f) \right|^2m = \sum_{i,j}^{n} \Psi_i(f)\Psi_j(f)\left\langle \nabla f_i, \nabla f_j\right\rangle m,
\end{aligned}$

\item[(ii)]
$\begin{aligned}[t]
\mathbf{\Delta}(\Psi(f)) = \sum_i\Psi_i(f)\mathbf{\Delta}(f_i)+\sum_{i,j}^{n}\Psi_{ij}(f)\left\langle \nabla f_i, \nabla f_j\right\rangle m,
\end{aligned}$

\item[(iii)] 
$\begin{aligned}[t]
\Gamma_2(\Psi(f))& = \sum_{i,j}^{n}\Psi_i(f)\Psi_j(f)\Gamma_2(f_i,f_j) + 2\sum_{i,j,k}^{n}\Psi_i(f)\Psi_{jk}(f) \mathrm{Hess}[f_i](f_j,f_k)m\\
 &+ \sum_{i,j,k,h}^{n}\Psi_{ik}(f)\Psi_{jh}(f)\left\langle \nabla f_i, \nabla f_j\right\rangle\left\langle \nabla f_k, \nabla f_h\right\rangle m.
\end{aligned}$
\end{itemize} 
\end{prop}


\subsection{Isomorphisms of metric measure spaces}

This is an account of several results in \cite{Gig}.
We consider metric measure spaces $(X,d,m)$ such that $(X,d)$ is complete and separable and $m$ is a non-negative Radon measure on $X$. We begin by recalling the definition of isomorphism of metric measure spaces.

\begin{defn}[Isomorphisms between metric measure spaces]
We say that two metric measure spaces $(X_1,d_1,m_1)$ and $(X_2,d_2,m_2)$ are  \emph{isomorphic} provided there exists an isometry $T:(\supp(m_1),d_1)\to(\supp(m_2),d_2)$ such that $T_\sharp m_1= m_2$. Any such $T$ is called an isomorphism.
\end{defn}

The following property will allow us to study isomorphisms between metric measure spaces in terms of isometries between their $W^{1,2}$ spaces, see Proposition \ref{prop:isom}. 

\begin{defn}[Sobolev to Lipschitz property]
\label{def:sobtolip}
Let $(X,d,m)$ be a metric measure space. We say that $(X,d,m)$ has the \textit{Sobolev to Lipschitz property} if any $f\in W^{1,2}(X,d,m)$ with $|\nabla f|\leq 1$ $m$-a.e. admits a 1-Lipschitz representative, that is,  a 1-Lipschitz map $g:X\to\mathbb R$ such that $f=g$ $m$-a.e..
\end{defn}

Gigli showed (using a result of Rajala \cite{Raj}) that, for finite $N$,  $\mathsf{CD}(K,N)$-spaces have the Sobolev to Lipschitz property. Furthermore, Ambrosio-Gigli-Savar\'e showed that $\mathsf{RCD}(K,\infty)$-spaces also have the Sobolev to Lipschitz property for $N\in(1,\infty)$ (see the paragraph after \cite[Definition 4.9]{Gig}). As $\mathsf{CD}^*(K,N)$ spaces are $\mathsf{CD}(K^*,N)$ spaces for a suitable value of $K^*$ (see \cite{Cav} and \cite{CavStu}), $\RCDst(K,N)$ with $N\in(1,\infty)$ spaces also satisfy the Sobolev to Lipschitz property.
 
\begin{lem}[Contractions by local duality \protect{\cite[Lemma 4.19]{Gig}}]
\label{le:contrdual}
Let $(X_1,d_1,m_1)$ and  $(X_2,d_2,m_2)$ be two metric measure spaces with the Sobolev to Lipschitz property where $m_2$ gives finite mass to bounded sets, and   $T:X_1\to X_2$ a  Borel map such that $T_\sharp m_1\leq C m_2$ for some $C>0$. Then the following are equivalent
\begin{itemize}
\item[i)] $T$ is $m_1$-a.e. equivalent to a 1-Lipschitz map  from $(\supp(m_1),d_1)$ to $(\supp(m_2),d_2)$
\item[ii)] For any $f\in W^{1,2}(X_2,d_2,m_2)$ we have  $f\circ T\in  W^{1,2}(X_1,d_1,m_1)$, and moreover,
\[
|\nabla (f\circ T)| \leq | \nabla f| \circ T,\qquad m_1-a.e..
\]
\end{itemize} 
\end{lem}

\begin{prop}[Isomorphisms via duality with  Sobolev norms \protect{\cite[Proposition 4.20]{Gig}}]\label{prop:isom}
Let $(X_1,d_1,m_1)$ and $(X_2,d_2,m_2)$ be two metric measure spaces with the Sobolev to Lipschitz property and $T:X_1\to X_2$ a Borel map. Assume that both $m_1$ and $m_2$ give finite mass to bounded sets. Then the following are equivalent.
\begin{itemize}
\item[i)] Up to a modification on a $m_1$-negligible set, $T$ is an isomorphism of the metric measure spaces
\item[ii)] The following two are true.
\begin{itemize}
\item[ii-a)] There exist a Borel $m_1$-negligible set $\mathcal N\subset X_1$ and a Borel map $S:X_2\to X_1$ such that $S(T(x))=x$, $\forall x\in X_1\setminus \mathcal N$.
\item[ii-b)]The right composition with $T$ produces an isometry of $W^{1,2}(X_2,d_2,m_2)$ in $W^{1,2}(X_1,d_1,m_1)$, i.e. $f\in W^{1,2}(X_2,d_2,m_2)$ if and only if $f\circ T\in W^{1,2}(X_1,d_1,m_1)$ and in this case $\|f\|_{W^{1,2}(X_2)}=\|f\circ T\|_{W^{1,2}(X_1)}$.
\end{itemize}
\end{itemize}
\end{prop}


\subsection{Warped product of metric measure spaces}\label{ssec-warp}

Here we review the main definitions and results concerning the warped products of metric measure spaces following Gigli-Han \cite{GigHan}. 

Let $(X, d_X, m_X)$ and $(Y, d_Y, m_Y)$ be two complete and separable metric measure spaces and $w_d,w_m:Y\to \mathbb [0, \infty)$ two continuous functions such that $\{w_d=0\}\subset\{w_m=0\}$.  
The $l_w$-length of an absolutely continuous curve $\gamma=(\gamma^Y,\gamma^X)$ in $Y \times X$ is defined by
\[
l_w[\gamma]=\int_0^1 \sqrt{|\dot{\gamma}^Y_t|^2+w^2_d(\gamma^Y_t)|\dot{\gamma}^X_t|^2} \, \mathrm{d} t.
\]
The function $d_w: (Y \times X )^2 \to \mathbb R$ given by
\[
d_w(p,q)=\inf\{l_w[\gamma]:\gamma ~\text{is an absolutely continuous curve from}~ p ~\text{to}~ q \}
\]
is a pseudometric. Hence, it induces an equivalence relation on $Y \times X$.  By taking the quotient and then its completion we obtain a metric space denoted by $Y\times_wX$ and an induced distance denoted also by $d_w$.  If $w_d(y)>0$ there is no abuse in denoting the elements of $Y\times_wX$ by $(y,x)$ with $y\in Y$ and $x\in X$, because points in the completion not coming from points in $Y\times X$ will be negligible with respect to the measure of $Y\times_wX$. The same holds for the set  of  elements $(y,x)$ that satisfy $w_d(y)=0$. 

The measure $m_w$ on $Y\times_wX$ is defined as
\begin{equation}
\label{eq:warpmeas}
\int f(x)g(y)\, \mathrm{d} m_{w}(y,x)= \int \left(\int f(x)w_m(y)\,\mathrm{d}m_X(x) \right) g(y)\,\mathrm{d}m_Y(y),
\end{equation}
for any Borel non-negative functions $f: X \to \mathbb R$ and $g: Y \to \mathbb R$. 

The warped product of $(X, d_X, m_X)$ and $(Y, d_Y, m_Y)$ via the functions $w_d$ and $w_m$, called warping functions, is the metric measure space denoted by $(Y\times_wX,d_w,m_w)$. By definition $(Y\times_wX,d_w,m_w)$ is complete, separable and is a length space.

\begin{defn}[Almost everywhere locally doubling space]\label{def:aeLoc}
Let $(X,d,m)$ be a metric measure space. We say that it is an almost everywhere locally doubling space 
provided there exists a Borel set $B$ with $m$-negligible complement such that for every $x \in B$ there exists an open
 set $U$ containing $x$ and constants $C,R > 0$ for which 
 \[
 m (B(y,2r)) \leq C m(B(y,r))
 \]
 for $r \in (0,R)$ and $y \in U$.
\end{defn}

\begin{defn}[Measured-length space]\label{def:measL}
Let $(X,d,m)$ be a metric measure space. We say that it is  measured-length if there exists a Borel set $A \subset X$ with $m$-negligible complement that satisfies the following. For all  $x_0,x_1 \in A$, there exist $\varepsilon > 0$ and a map $(0,\varepsilon]^2 \to \mathcal P(C([0,1],X))$, $(\varepsilon_0, \varepsilon_1) \mapsto \pi ^{\varepsilon_0, \varepsilon_1}$, such that 
\begin{itemize}
\item For any $\varphi \in C_b(C([0,1], X))$, the map $(0,\varepsilon]^2 \to \mathbb R$ given by 
\[
(\varepsilon_0, \varepsilon_1) \mapsto \int \varphi d \pi ^{\varepsilon_0, \varepsilon_1},
\]
is Borel. 
\item  For every $\varepsilon_0, \varepsilon_1 \in (0,\varepsilon]$  and $i=0,1$,
we have 
 \[
 (e_i)_\sharp \pi^{\varepsilon_0, \varepsilon_1} = \frac {1_{B (x_i, \varepsilon_i) } } {m(B (x_i,  \varepsilon_i))}m.
  \] 
\item We have 
 \[
  \limsup_{\varepsilon_0, \varepsilon_1 \downarrow 0} \int \int_0^1 |  \gamma^{\cdot}_t|^2 \,dt \,d \pi ^{\varepsilon_0, \varepsilon_1} (\gamma) \leq d^2(x_0,x_1).
  \]
\end{itemize}
\end{defn}

\begin{thm}[\protect{\cite[Theorem 3.22]{GigHan}}]\label{thm:StoLipGH}
Let $(X,d,m)$ be an a.e.\ locally doubling and measured-length space, $I\subset \mathbb R$ a closed, possibly unbounded, interval and $w_d,w_m:I\to\mathbb [0,\infty)$ a couple of warping functions.  Assume that $w_m$ is strictly positive in the interior of $I$. Then the warped product space $(X_w,d_w,m_w)$, where $X_w= I \times_w X$, is almost everywhere doubling and a measured-length space. Hence, it has the Sobolev to Lipschitz property.
\end{thm}

The following result may be shown from the equivalence of the Beppo-Levi space (\cite[Definitions 3.8, 3.9]{GigHan}) and the Sobolev space on warped products obtained by Gigli-Han. For simplicity, we will not restate here the precise definition of the Beppo-Levi space, rather only summarize their results in a manner suitable for our purposes (cf. \cite[Propositions 3.10, 3.13, 3.14]{GigHan}). Given $f:X_w\to\mathbb R$, let $f^{(t)}:X\to\mathbb R$ and $f^{(x)}:I\to\mathbb R$ denote the functions $f^{(t)}(x)=f(t,x)$ and $f^{(x)}(t)=f(t,x)$.

\begin{thm}[\cite{GigHan}] \label{def-warped-BL} 
Let $(X,d,m)$ be a metric measure space, $I\subset \mathbb R$ a closed, possibly unbounded, interval and $w_d,w_m:I\to\mathbb [0,\infty)$ warping functions. Suppose that $\{ w_m=0\} $ is finite and for some $C >0$, $w_m(t) \leq C \inf_{\{s:\, w_m(s)=0\}} |t-s|$ for all $t \in  I$, then the following two are equivalent:
\begin{itemize}
\item[1.] $f \in  W^{1,2}(X_w,d_w,m_w)$
\item[2.]
\begin{itemize}
	\item [(i)] For $m$-a.e.\ $x\in X$ we have $f^{(x)} \in W^{1,2}(\mathbb R, d_{Euc}, w_m\mathcal L^1)$,
	\item[(ii)] For ${w_m}\mathcal L^1$-a.e. $t\in \mathbb R$ we have $f^{(t)}\in W^{1,2}(X)$,
	\item[(iii)] For all $(t,x) \in X_w$, 
	\begin{equation}
	\label{eq:warpedgrad}
	|\nabla  f|^2_{X_w}(t,x)=w_d^{-2}(t)|\nabla  f^{(t)}|_X^2(x)+|\nabla f^{(x)}|_{L^2(\mathbb R,w_m\mathcal L^1)}.
	\end{equation}
\end{itemize}
\end{itemize}
\end{thm}

\begin{rmk}
In the statements of Theorems \ref{thm:StoLipGH} and \ref{def-warped-BL}, $m$ is assumed to be a finite measure. However as explained in the remark after Definition 2.9 of \cite{GigHan}, if $w_m$ never vanishes, as in our application, then the results still hold when $m$ is infinite.
\end{rmk}

\begin{cor}\label{cor:sezioniprod} 
With the same notation and assumptions of Theorem \ref{def-warped-BL} the following are true.
\begin{itemize}
\item[i)] Let $f\in S^2_{\rm loc}(X_w)$. Then for $m$-a.e. $x$,  $f^{(x)}\in S^2_{\rm loc}( \omega_m \mathcal L^1)$.  For $\omega_m \mathcal L^1$-a.e. $t$, $f^{(t)}\in S^2_{\rm loc}(X)$.  Furthermore, \eqref{eq:warpedgrad} holds in this setting.
\item[ii)] Let $f_1\in S^2_{\rm loc}(w_m \mathbb R)$ and define $f: X_w \to \mathbb R$ by $f(t,x)=f_1(t)$. Then $f\in S^2_{\rm loc}(X_w)$ and
\[
|\nabla f|_{X_w}(t,x)=|\nabla f_1|_{w_m\mathbb R}(t),\qquad m_w-\text{a.e.} \ (t,x).
\]
\item[iii)] Let $f_2\in S^2_{\rm loc}(X)$ and define $f:X_w \to \mathbb R$ by $f(t,x):=f_2(x)$. Then $f\in S^2_{\rm loc}(X_w)$ and
\[
|\nabla f|_{X_w}(t,x)= w_d^{-1}(t)  |\nabla f_2|_{X}(x),\qquad m_w-\text{a.e.} \ (t,x).
\]
\end{itemize}
\end{cor}

\begin{proof}
All the properties follow from the previous theorem with a truncation and cut-off argument based on the locality property of minimal weak upper gradients, see subsection \ref{ssec-calculus}.
\end{proof}

\begin{cor}
	\label{cor:warpHil}
With the same notation and assumptions of Theorem \ref{def-warped-BL}, 
if $(X,d,m)$ is infinitesimally Hilbertian then the metric measure space $(X_w,d_w,m_w)$ is  infinitesimally Hilbertian. 
\end{cor}

\begin{proof}
Let $f,g\in S^2_{loc}(X_w)$.  For simplicity, in this proof we will write $|\nabla f \cdot |_{w_m \mathbb R}$ to refer to the weak upper gradient of a Sobolev function $f$ in $S^{2}(\mathbb{R},d_{Euc},w_m\mathcal{L}^{1})$. By Theorem \ref{def-warped-BL} we get
\[
|\nabla(f+g)|^2_{X_w}+|\nabla (f-g)|^2_{X_w} = 
w_d^{-2}(|\nabla(f+g)^{(t)}|^2_{X}+|\nabla (f-g)^{(t)}|^2_{X}) +
(|\nabla(f+g)^{(x)}|^2_{w_m\mathbb R}+|\nabla (f-g)^{(x)}|^2_{w_m \mathbb R}).
\]

Now, by Corollary \ref{cor:sezioniprod} above we know that 
$f^{(t)},g^{(t)} \in S^2_{loc}(X)$ and $f^{(x)},g^{(x)} \in S^2_{loc}(w_m \mathcal L^1)$. 
As $(X,d,m)$ is infinitesimally Hilbertian, 
\[
|\nabla(f^{(t)}+g^{(t)})|^2_X+ |\nabla (f^{(t)}-g^{(t)})|^2_X=2\big(|\nabla f^{(t)}|_X^2+|\nabla g^{(t)}|_X^2\big), \qquad m-\text{a.e}.
\]
In a similar way, because $(\mathbb R, d_{Euc}, \omega_m \mathcal L^1)$ is infinitesimally Hilbertian we obtain
\[
|\nabla(f^{(x)}+g^{(x)})|^2_{w_m\mathbb R}+|\nabla (f^{(x)}-g^{(x)})|^2_{w_m\mathbb R}=2\big(|\nabla f^{(t)}|_{w_m \mathbb R}^2+|\nabla g^{(t)}|_{w_m\mathbb R}^2\big),\qquad w_m \mathcal L^1-\text{a.e.}.
\]

Putting the equations together and because the choices of  
$f,g\in S^2_{loc}(X_w)$ were arbitrary, we get the result.
\end{proof}

Now we define,
\begin{align*}
\mathcal G = &\Big\{g \in S^2_{\rm loc}(X_w) \ |\  g(x,t)=\tilde g(x)\textrm{ for some }  \tilde g\in S^2(X)\cap L^\infty(X) \Big\},\\
\mathcal H  = &\Big\{h \in S^2_{\rm loc}(X_w)\ |\  h(x,t)=\tilde h(t)\textrm{ for some }  \tilde h\in S^2( w_m\mathbb R)\cap L^\infty(\mathbb R) \Big\},\\
\mathcal A =  & \textrm{ algebra generated by }\mathcal G\cup\mathcal H \subset S^2_{\rm loc}(X_w).
\end{align*}

\begin{prop}\label{prop:approximation}
Let $(X,d,m)$ be a metric measure space and $w_d,w_m:\mathbb R \to\mathbb [0,\infty)$ warping functions. Suppose that $\{ w_m=0\} $ is finite and for some $C \in \mathbb R$, 
\[
w_m(t) \leq C \inf_{ \{s: w_m(s)=0  \}} |t-s|
\]
 for all $t \in  I$, then
the set $\mathcal A\cap W^{1,2}(X_w)$ is dense in $W^{1,2}(X_w)$.
\end{prop}

\begin{proof} 
Consider the algebra
	\begin{equation*}
\mathcal A_a^b = \textrm{ algebra generated by }(\mathcal G\cup\mathcal H \cap S^2_{\rm loc}(([a,b] \times _w X,d_w,m_w)).
\end{equation*}
By the Cartesian product case proved in \cite[Proposition 6.6]{Gig} (see also \cite[Proposition 3.35]{DePG}), $\mathcal A_a^b \cap W^{1,2}(X_w)$ is dense in $W^{1,2}([a,b] \times _w X,d_w,m_w)$ whenever $[a,b]\subset \mathbb R\setminus \{w_m=0\}$.

It follows that $\mathcal A\cap W^{1,2}(X_w)$ is dense in $BL_0(X_w)$ which is the closure in $BL(X_w)$ of the space of functions which vanish in a neighborhood of $\{w_m=0\}\cup\{\infty\}$. (See \cite{GigHan} for the definitions of the Beppo-Levi spaces  $BL_0(X_w)$ and  $BL(X_w)$.) However, under the hypotheses,  \cite[Proposition 3.14]{GigHan} shows that $BL_0(X_w)=BL(X_w)=W^{1,2}(X_w)$ which implies the statement.
\end{proof}

\subsection{Universal covers of $\RCDst$ spaces}

A metric space $(Y,d_Y)$ is a \textit{covering space} of $(X,d_X)$ if there exists a continuous map $p: Y\to X$ such that for every point $x\in X$ there exists a neighborhood $U_x\subset X$ with the property that $p^{-1}(U_x)$ is a disjoint union of open subsets of $Y$ each of which is mapped homeomorphically onto $U_x$ by $p$. 

A (connected) metric space $(\tilde X, d_{\tilde{X}})$ is a \textit{universal cover} of $X$, with covering map $\tilde{p}$,  if for any other covering space $Y$ of $X$ with covering map $p$ there exists a continuous map $f:\tilde{X}\to Y$ such that $p \circ f= \tilde{p}$. Whenever a universal cover exists, it is unique. (Note that we do not require $X$ to be semilocally simply connected, so $\tilde{X}$ need not be simply connected.)

In the presence of the $\RCD^*$ condition, the following theorem was obtained by Mondino-Wei \cite[Theorem 1.1]{MW}.
\begin{thm} \label{thm-MW}
Let $(X,d,m)$ be an $\RCD^*(K,N)$-space for some $K\in\mathbb{R}$, $N\in (1,\infty)$. Then $(X,d,m)$ admits a universal cover $(\tilde X, \tilde d, \tilde m)$, with $\tilde m$ given by the pullback measure via the covering map, which is itself an $\RCD^*(K,N)$-space. 
\end{thm}
 

\section{Construction of a Busemann function}\label{sec-Busemann}

In this section we first prove that the volume entropy of compact $\RCDst(-(N-1),N)$ spaces is bounded above by $N-1$. In the equality case, we construct a Busemann type function $u$ defined on the universal cover of the space.  Finally we  show the existence and main properties of the Regular Lagrangian Flow  of $\nabla u$. As our space is noncompact, we need to make use of good cut-off functions, and local uniqueness results for Regular Lagrangian Flows and the continuity equation. 

\subsection{Volume growth entropy estimate for $\RCDst$ spaces}

\begin{thm}\label{lemEntK*}
Let $(X,d,m)$  be an $\RCDst(K,N)$-space with $N\in (1,\infty)$ and $K<0$. Then 
\[
h(X)\leq \sqrt{-K (N - 1)}. 
\]
\end{thm}

\begin{proof}
By the work of Mondino-Wei \cite{MW} (see Theorem~\ref{thm-MW}), the universal cover space $\tilde{X}$ is also an $\RCDst(K, N)$ space. In particular, it is a $\mathsf{CD}^*(K, N)$ space. Let $R>0$ and let us fix $r_0$ such that $0<r_0<R$. By  Theorem \ref{thm-BishopGrom},
\[
\tilde{m} ( B_{\tilde{X}} (x,R)) \int_{0}^{r_{0}\sqrt{-K/ (N - 1)} }\sinh^{ N - 1} t \, \mathrm{d}t \leq  \tilde{m} ( B_{\tilde{X}} (x,r_0)) \int_{0}^{R\sqrt{-K/ (N - 1)}}\sinh^{ N - 1 } t \, \mathrm{d}t .
\]
Taking logarithms, dividing by $R$ and taking the limsup on both sides of the previous inequality  we get
\[
h({X}) \leq \lim_{R\to \infty} \frac{1}{R} \ln \left( \int_{0}^{R\sqrt{-K/  (N - 1 )}} \sinh^{ N - 1} t \, \mathrm{d}t \right).
\] 
To conclude,  we use L'H\^{o}pital's rule. 
\end{proof}

The next corollary follows directly by taking $K=-(N-1)$ in the previous theorem.

\begin{cor}\label{thmEntK*}
Let $(X,d,m)$  be an $\RCDst(-(N-1),N)$-space with $N\in (1,\infty)$. Then $h(X)\leq  N -1$.
\end{cor}

We remark that the previous volume entropy growth estimate holds in the more general setting of spaces which satisfy the \textit{measure contraction property} introduced by Ohta \cite{Ohta} and Sturm \cite{Stu2}. Indeed, a Bishop-Gromov type inequality was obtained in \cite[Theorem 5.1]{Ohta} and the proofs of Theorem \ref{lemEntK*} and Corollary \ref{thmEntK*} can be carried out in this setting analogously.


\subsection{Construction of a Busemann function}
\label{sec-busemann-function}
In this section we will prove the following result on the existence of a Busemann-type function on the universal cover of a compact $\RCDst(K,N)$ space with maximal volume entropy. We will follow the strategy developed by Liu \cite{Liu2011}, with the necessary adaptations (cf. \cite[Theorem 1.7]{Jiang}). More precisely, we will prove: 

\begin{thm}\label{cor-u}
Let $(X,d,m)$ be a compact $\RCDst(K,N)$ space with $K<0$ and $N \in (1, \infty)$, and let $(\tilde X, \tilde d, \tilde m)$ be its universal cover. If $h(X)= \sqrt{-K (N-1)}$, then there exists a function $u: \tilde X \to \mathbb R$ with $u\in D_{loc}(\Delta)$ , that satisfies $| \nabla u| = 1$ $\tilde{m}$-a.e. and $\Delta u =  \sqrt{-K (N-1)}$ $\tilde{m}$-a.e. . 
\end{thm}

The theorem follows from the following technical lemma.  

\begin{lem}\label{lem-u}
Let $(X,d,m)$ be a compact $\RCDst(K,N)$ space with $K<0$, $N \in (1, \infty)$, and $(\tilde X, \tilde d, \tilde m)$ its universal cover.  If $h(X)= \sqrt{-K (N-1)}$, then for any $y_0 \in \tilde X$ and $R > 50 \, \mathrm{diam}(X)$  there exists $u_R: B(y_0, R) \to \mathbb R$ Lipschitz with  $| \nabla u_R| = 1$ $\tilde{m}$-a.e. and  $\Delta u_R =  \sqrt{-K (N-1)}$ $\tilde{m}$-a.e..
\end{lem}

To prove the previous lemma we need the following propositions.  Set $Q:= \sqrt{-K(N-1)}$. Let us recall the definition of the function $s_{\tilde{m}}$ appearing in Theorem \ref{thm-BishopGrom}:
\[
s_{\tilde{m}}(x,r) = \limsup_{\delta \to 0} \frac{1}{\delta} \tilde{m}\left( \overline{B(x,r+ \delta)}  \setminus B(x,r)\right).
\]

\begin{prop}\label{prop1}
 For any $o \in \tilde X$ we have
 \[
\limsup_{r \to \infty}\frac{s_{\tilde m} (o,r + 50R)}{s_{\tilde m} (o, r-50R)} = \exp({100 Q R}).
\]
In particular, there is a sequence of positive numbers  $r_i$ with $\lim_{i \to \infty}r_i=\infty$, such that $\frac{s_{\tilde m} (o,r_i + 50R)}{s_{\tilde m} (o, r_i-50R)} $  converges to $\exp({100 Q R})$.
\end{prop}

\begin{proof}
Since $h(X)= Q > 0$,  $\tilde X$ has infinite diameter.  Recall that by Mondino-Wei \cite{MW}, $(\tilde X, \tilde d, \tilde m)$ is an $\RCDst(K,N)$ space. By Theorem  \ref{thm-BishopGrom},
\[
\frac{s_{\tilde m} (o,r + 50R)}{s_{\tilde m} (o, r-50R)}   \leq      \frac{\sinh^{N-1} (   Q  (r+ 50R)) }{\sinh^{N-1}( Q (r-  50R)) }.
\]
Notice that
\[
\lim_{r \to \infty} \frac{\sinh^{N-1} (Q  (r + 50R)) }{\sinh^{N-1}( Q (r -  50R)) } = \exp({100QR}). 
\]
We will show that 
\[
\limsup_{r \to \infty} \frac{s_{\tilde m} (o,r + 50R)}{s_{\tilde m} (o, r-50R)} = \exp({100 Q R}).
\]
By contradiction, suppose that there exist $r_0 > 100R$ and $\varepsilon > 0$ such that for any $r \geq r_0$,
\[
\frac{s_{\tilde m} (o,r + 50R)}{s_{\tilde m} (o, r-50R)}   \leq   (1- \varepsilon) \exp({100  Q R}).
\]
Therefore, for any $r > r_0$ big enough  we have that 
\[
s_{\tilde m} (o,r) \leq   (1- \varepsilon) \exp({100 Q R}) s_{\tilde m} (o, r-100R).
\]
Iterating this inequality $\lfloor \tfrac{r-r_0}{100R}\rfloor$ times, where $\lfloor \tfrac{r-r_0}{100R} \rfloor$ is the largest integer smaller than or equal to $\tfrac{r-r_0}{100R}$,  we get
\begin{equation*}
s_{\tilde m} (o,r) \leq 
 \left(  (1-\varepsilon)  \exp({100   Q   R})\right)^{\lfloor \tfrac{r-r_0}{100R} \rfloor}  s_{\tilde m} (o, r-\lfloor \tfrac{r-r_0}{100R} \rfloor100R).
\end{equation*}

Now, $r-\lfloor \tfrac{r-r_0}{100R} \rfloor100R= r_0 + t$ for some $t\in [0,100R)$. 
Hence, by Theorem \ref{thm-BishopGrom} and as the hyperbolic sine is an increasing function:
\begin{eqnarray*}
s_{\tilde m} (o, r- \lfloor \tfrac{r-r_0}{100R} \rfloor 100R) &\leq & s_{\tilde m} (o, r_0)\frac{\sinh^{N-1} 
(Q(r- \lfloor \tfrac{r-r_0}{100R} \rfloor 100R)) }{\sinh^{N-1}(Q r_0)}\\
 &\leq & s_{\tilde m} (o, r_0)\frac{\sinh^{N-1} ( Q  (r_0 + 100R))}{\sinh^{N-1}(Q r_0)} .
\end{eqnarray*}
Thus, for $r \geq r_0$
\[
s_{\tilde m} (o,r) \leq c(N,K,r_0,R) \left(  (1-\varepsilon)  \exp({100  Q R})\right)^{\frac{r-r_0}{100R}},
\]
where we used that $\lfloor \frac{r-r_0}{100R} \rfloor \leq \frac{r-r_0}{100R}$.
Integrating $s_{\tilde m}(o, \cdot)$ from $r_0$ to $r$  and using the previous inequality, we get an upper bound of $\tilde m (B(o,r) \setminus B(o,r_0))$. Using this bound, we obtain 
\[
h(X)=\limsup_{r \to \infty} \frac{1}{r} \ln \tilde m (B(o,r)) < Q.
\]
This contradicts $h(X)= Q$, and concludes the proof.
\end{proof}

For the following proposition let us recall that any distance function $r(x):= \tilde d(o,x)$ on $\tilde{X}$ has a well-defined measure valued Laplacian on $\tilde{X}\setminus \{o\}$. Moreover it is a signed Radon measure  and an exact formula is presented in \cite[Corollary 4.19, Theorem 1.1]{CavMon}.  Denote $A(o,r_1,r_2):=\left\{ x\in \tilde{X} \mid r_1 \le  \tilde d(o,x) < r_2 \right\}$.  
Then, we have the following divergence formula. See \cite[Lemma 2.11]{Jiang} for Alexandrov space case. 
\begin{prop}\label{prop2}For a.e. $o \in \tilde{X}$ and for all but countably many  $t \in (0, \infty)$, 
\[
\int_{B(o,t)}  \mathbf{\Delta}r = s_{\tilde m} (o, t).
\]
In particular, for all but countably many $t_2 \ge t_1 > 0$, and a.e. $o \in \tilde{X}$,
\begin{equation}
\int_{A(o,t_1,t_2)}  \mathbf{\Delta}r = s_{\tilde m} (o, t_2) -s_{\tilde m} (o, t_1).  \label{eq-lap-annulus}
\end{equation}
\end{prop}

\begin{proof}
Fix  $0<\epsilon_0 <\frac t2$, define $\psi_{\epsilon_0}$  as
	\[   \psi_{\epsilon_0}(x)=\left\{
	\begin{array}{ll}
	0 & \text{if} \ x\in B(o,\epsilon_0)\\
	\frac{r(x)}{\epsilon_0} -1 & \text{if} \ x\in A(o,\epsilon_0, 2\epsilon_0)\\
	1 &  \text{otherwise}.
	\end{array} 
	\right. \]
Let $\{\delta_i\}_{i\in\mathbb{N}}$ be a decreasing sequence such that $\delta_i \to 0$. For each $\delta_i$ define a function $f_{\delta_i}:\tilde{X}\to\mathbb{R}$ by
\[   f_{\delta_i}(x):=\left\{
\begin{array}{ll}
       \psi_{\epsilon_0} (x)  & \text{if} \ x\in B(o,t)\\
     1- \frac{1}{\delta_i} (r(x)-t) & \text{if} \ x\in A(o,t,t+\delta_i)\\
      0 &  \text{otherwise}.
\end{array} 
\right. \]
 We observe that $f_{\delta_i}\in W^{1,2}(\tilde{X},\tilde{d},\tilde{m})$ for all $i\in\mathbb{N}$. Then, 
\[
\int_{B(o,t+ \delta_1) } f_{\delta_i}  \mathbf{\Delta} r =  \int_{B(o,t) }   \psi_{\epsilon_0}\,  \mathbf{\Delta}r + \int_{A(o,t, t+\delta_1)}  f_{\delta_i} \mathbf{\Delta} r.
\]
By the definition of $\mathbf{\Delta} r$ and $f_{\delta_i}$ we now have that,
\begin{eqnarray*}
\int_{B(o,t+ \delta_1)} f_{\delta_i}  \mathbf{\Delta} r &=& -\int_{B(o,t+ \delta_1)} \left\langle \nabla f_{\delta_i}, \nabla r\right\rangle\, \mathrm{d} \tilde m \\ 
&=&  \frac{1}{\delta_i}  \int_{A(o,t, t+ \delta_i)} \left\langle \nabla r, \nabla r\right\rangle\, \mathrm{d} \tilde m - \frac{1}{\epsilon_0} \int_{A(o,\epsilon_0, 2\epsilon_0)} \left\langle \nabla r, \nabla r\right\rangle\, \mathrm{d} \tilde m \\
&=&  \frac{1}{\delta_i}  \tilde m(A(o,t, t+ \delta_i)) - \frac{1}{\epsilon_0}\tilde m(A(o,\epsilon_0, 2\epsilon_0)).
\end{eqnarray*}
Hence, 
\[
\int_{B(o,t)} \psi_{\epsilon_0}  \mathbf{\Delta} r + \int_{A(o,t, t+\delta_1)}  f_{\delta_i} \mathbf{\Delta} r =  \frac{1}{\delta_i} \tilde m(A(o,t, t+ \delta_i)) -  \frac{1}{\epsilon_0}\tilde m(A(o,\epsilon_0, 2\epsilon_0)).
\]
Now choose $\delta_i$ to be a specific sequence achieving the $\limsup$ in the definition of $s_{\tilde{m}}$. Taking the limit when $i \to \infty$, we get: 
\[
\int_{B(o,t)}  \psi_{\epsilon_0} \, \mathbf{\Delta} r + \lim_{i \to \infty}  \int_{A(o,t, t+\delta_1)} f_{\delta_i} \mathbf{\Delta} r = s_{\tilde m} (o, t) - \frac{1}{\epsilon_0}\tilde m(A(o,\epsilon_0, 2\epsilon_0)). 
\]
Notice that $0 \leq  f_{\delta_i} \leq 1$ and 
\[ | \int_{A(o,t, t+\delta_1)} f_{\delta_i}\,  \mathbf{\Delta} r | \le \int_{A(o,t, t+\delta_i)} |\mathbf{\Delta} r |.  \]
Since  $|\mathbf{\Delta} r |$ is a Radon measure, we have $\lim_{i \to \infty} \int_{A(o,t, t+\delta_i)} |\mathbf{\Delta} r | = 0$ for all but countably many $t \in (0, \infty)$. 
Therefore
\[
\int_{B(o,t)} \psi_{\epsilon_0} \, \mathbf{\Delta}r = s_{\tilde m} (o, t) - \frac{1}{\epsilon_0}\tilde m(A(o,\epsilon_0, 2\epsilon_0)). 
\]
This is true for all $o \in \tilde{X}$. 
Now for $\RCDst (K,N)$ space with $1<N<\infty$, for a.e. $o\in \tilde{X}$, $\limsup_{\epsilon_0 \to 0} \frac{1}{\epsilon_0}\tilde m(A(o,\epsilon_0, 2\epsilon_0) =0$ (see \cite[Remark 5.4]{CavMon}). Letting $\epsilon_0 \to 0$ for those $o$ above gives the result. 
\end{proof}

\begin{rmk}
The final part of the above proof shows that a.e. $t$ $s_{\tilde{m}}(x,t)$ is actually a limit,
\[
s_{\tilde{m}}(x,r) = \lim_{\delta \to 0} \frac{1}{\delta} \tilde{m}\left( \overline{B(x,r+ \delta)}  \setminus B(x,r)\right).
\]
\end{rmk}

Let $A\subset \tilde{X}$. In the following proposition, we will use the notation $\strokedint_{A} \mathbf{\Delta}r := \frac{\int_{A}\mathbf{\Delta}r}{\tilde{m}(A)}$.

\begin{prop}\label{prop3}
Set $A_i= \{ y \in \tilde X \, | \,  r_i - 50R \leq \tilde  d(o,y) \leq  r_i + 50 R \}$. Then,
 $$\strokedint_{A_i} \mathbf{\Delta} r \geq  Q - \Psi(i)$$ where $\lim_{i \to \infty} \Psi (i)=0$. 
\end{prop}

\begin{proof}
We now prove that $\strokedint_{A_i} \mathbf{\Delta} r \geq  Q - \Psi(i)$. By \eqref{eq-lap-annulus} and the definition of $A_i$, 
\[
\int_{A_i} \mathbf{\Delta}r = s_{\tilde m} (o, r_i + 50 R) - s_{\tilde m} (o, r_i - 50 R).
\]
By Proposition~\ref{prop1}, as $i$ goes to infinity,
\[
 \frac{s_{\tilde m} (o,r_i + 50R)}{s_{\tilde m} (o, r_i-50R)}\ \longrightarrow \ \exp({100 Q R}),
\] 
and therefore, there exist $\Psi(i) > 0$ such that $\lim_{i \to \infty} \Psi(i)=0$ and 
 \[
 \frac{s_{\tilde m} (o,r_i + 50R)}{s_{\tilde m} (o, r_i-50R)} +   \Psi(i) \geq \exp({100 Q  R}).
\]
Thus, 
\begin{align*}
\strokedint_{A_i} \mathbf{\Delta} r &= \frac {s_{\tilde m} (o, r_i + 50 R)}{\tilde m (A_i)}-\frac {s_{\tilde m} (o, r_i - 50 R)}{\tilde m (A_i)}\\
&\geq \frac {s_{\tilde m} (o, r_i - 50 R)}{\tilde m (A_i)} ( \exp({100 Q R}) - 1)-\frac {s_{\tilde m} (o, r_i - 50 R))} {\tilde m (A_i)}\Psi(i).
\end{align*}
Hence we only need to show that 
\[
\lim_{i \to \infty} \frac {s_{\tilde m} (o, r_i - 50 R)}{\tilde m (A_i)} = \frac{ Q }{\exp({100 Q R}) - 1  }.
\]
This would imply the existence of  $\Psi(i) > 0$ that satisfies the claim. 

\medskip

By Theorem \ref{thm-BishopGrom} we have that for $t\in [r_i-50R,r_i+50R]$,
\begin{eqnarray*}
\frac{\tilde{m}(A_i)}{s_{\tilde{m}}(o,r_i-50R)}&=&\int_{r_i-50R}^{r_i+50R}\frac{s_{\tilde{m}}(o,t)}{s_{\tilde{m}}(o,r_i-50R)}dt \\
&\leq & \int_{r_i-50R}^{r_i+50R} \frac{\sinh^{N-1}\left( Q  t \right)}{\sinh^{N-1}\left(Q (r_i-50R) \right)}dt\\
&=& \frac{\int_{r_i-50R}^{r_i+50R}\sinh^{N-1}\left(Q t \right)dt}{\sinh^{N-1}\left(Q (r_i-50R) \right)}.
\end{eqnarray*}
Using L'H\^{o}pital's rule we conclude,
\begin{eqnarray*}
\lim_{i\to \infty}\frac{\tilde{m}(A_i)}{s_{\tilde{m}}(o,r_i-50R)} &\leq & \lim_{i\to \infty} \frac{\int_{r_i-50R}^{r_i+50R}\sinh^{N-1}\left(    Q t \right)dt}{\sinh^{N-1}\left( Q (r_i-50R) \right)}\\
&=& \lim_{i\to \infty} \frac{-\sinh^{N-1}(Q(r_i-50R))+\sinh^{N-1}(Q(r_i+50R))}{(N-1)Q\sinh^{N-2}( Q (r_i-50R))\cosh( Q (r_i-50R))}\\
&=&  \frac{-1 + \exp({100 Q R})}{Q }. 
\end{eqnarray*}
\end{proof}

Recall that $A_i= \{ y \in \tilde X \, | \,  r_i - 50R \leq  \tilde d(o,y) \leq  r_i + 50 R \}$. Fix a $y_0 \in A_i$.  Let $\pi: \tilde X \to X$  be 
the universal covering map, and set 
\[
A_i(y_0)= \{ y \in \tilde X\, |\,  \pi(y)=\pi(y_0), \,B(y,R) \subset A_i\}.
\]

\begin{prop}\label{prop4} 
For every $i\in\mathbb{N}$, there exists $y_i \in A_i(y_0)$ such that 
\[
\strokedint_{B(y_i,R)} \mathbf{\Delta} r \geq   Q - \Psi(i).
\]
\end{prop}

\begin{proof}
Let $E_i$ be the maximal set of $A_i(y_0)$ such that $B(y_1, R) \cap B(y_2, R) = \emptyset$ for distinct points $y_1, y_2$ in $E_i$. Set $F_i= \bigcup_{y \in E_i } B(y,R)$. Using Proposition \ref{prop3} we will show that 
\[
\strokedint_{F_i} \mathbf{\Delta} r \geq   Q  - \Psi(i).
\]
As $F_i= \bigcup_{y \in E_i } B(y,R)$ is the union of mutually disjoint balls it will follow then that there is a point $y_i \in E_i$ such that 
\[
\strokedint_{B(y_i,R)} \mathbf{\Delta} r \geq   Q  - \Psi(i).
\]

To achieve this goal, first we estimate a lower bound for $\frac{\tilde m (F_i)}{\tilde m (A_i)}$. Let  $G_i= \bigcup_{y \in E_i} B(y, 5R)$. The cardinality of $E_i$ is finite, all of its elements are preimages of the same point under the covering map $\pi$, and $\tilde m$ is locally equal to $m$, from which we obtain $\tilde m (F_i)= \sum_{y \in E_i } \tilde m (B(y,R)) = \mathrm{card}(E_i)  \tilde m (B(y', R)))$ and 
\[
\tilde m (G_i) \leq \sum_{y \in E_i} \tilde m (B(y, 5R))) = \mathrm{card}(E_i) \tilde m (B(y', 5R)))
\]
for $y' \in E_i$.  Thus, 
\[
\frac{\tilde m (F_i)}{\tilde m (G_i)}   \geq \frac{ \mathrm{card}(E_i) \tilde m (B(y',R))}{\mathrm{card}(E_i)\tilde m (B(y',5R))}  \geq   \frac{ v_{K, N}(R) }{v_{K,N}(5R) }, 
\] 
by applying Theorem \ref{thm-BishopGrom} with $v_{K, N}(r)= \int_{0}^{r}\sinh^{N-1} (Q t) \, \mathrm{d}t$.

Now we will find a bound for $\tilde m (A_i)$. We will prove that 
\[
A(o, r_i -10R, r_i + 10R)=\{ y \in \tilde X\, | \, r_i -10R <  \tilde d(o,y) < r_i + 10R\} \subset G_i.
\]
  Let $z \in A(o, r_i -10R, r_i + 10R)$, we will show $z \in G_i$.  As $z \in \tilde X$ there exists a point $y \in \pi^{-1}(\pi(y_0))$ such that $ \tilde d(z,y) \leq \mathrm{diam}(X)$. Then, by the triangle inequality,
\[
r_i - 10R - \mathrm{diam}(X) \leq \tilde d(o, y) \leq r_i + 10R + \mathrm{diam}(X).
\]
The previous inequality implies $y \in A_i(y_0)$. From the definition of $E_i$ there exists a point $y' \in E_i$ such that
$\tilde d(y, y') \leq R$. By the triangle inequality, $\tilde d(z,y') \leq \mathrm{diam}(X) + R$. Recalling that $R > 50 \, \mathrm{diam} (\tilde X)$ we deduce that $\tilde d(z,y') \leq 5R$. Hence, $z \in G_i$.  This proves  $A(o, r_i -10R, r_i + 10R) \subset G_i$. 

From the previous paragraph, $\tilde m (G_i) \geq m (A(o, r_i -10R, r_i + 10R))$. Recall that 
$A_i=A(o, r_i -50R, r_i + 50R)$. Hence, by the generalized Bishop-Gromov volume comparison for annular regions we obtain:
\[
\frac{\tilde m (G_i)}{\tilde m (A_i)} \geq \frac{\tilde m (A(o, r_i -10R, r_i + 10R))}{\tilde m (A_i)} \geq  \frac{\int_{r_i-10R}^{r_i + 10R}\sinh^{N-1} (Q t) \, \mathrm{d}t}{\int_{r_i-50R}^{r_i + 50R}\sinh^{N-1}(Q t) \, \mathrm{d}t} 
\]
As 
\[
\lim_{i \to \infty}  \frac{\int_{r_i-10R}^{r_i + 10R}\sinh^{N-1} ( Q t) \, \mathrm{d}t}{\int_{r_i-50R}^{r_i + 50R}\sinh^{N-1}( Q t) \, \mathrm{d}t}  \geq \frac{\exp({-60 Q R})}{5},
\]
we can write 
\[
\frac{\tilde m (G_i)}{\tilde m (A_i)} \geq c(K, N, R).
\]
Therefore, 
\[
\frac{\tilde m (F_i)}{\tilde m (A_i)} =  \frac{\tilde m (F_i)}{\tilde m (G_i)}  \frac{\tilde m (G_i)}{\tilde m (A_i)} \geq \frac{v_{K,N}(R)}{v_{K,N}(5R)} c(K,N,R).
\]
The Laplacian comparison theorem for $\RCDst(K,N)$-spaces \eqref{eq-lapComp} then yields 
\[
\mathbf{\Delta} r|_{\tilde{X}\setminus \{o\}} \leq Q \coth(Qr) \tilde m .
\]
Observe that $\mathbf{\Delta} r \leq \left( Q  + \delta(i,K,N)\right)\tilde{m} $ on $A_i$, because $\lim_{r \to \infty}\coth(r)=1$ and $\coth(r)\geq 1$, here $\lim_{i \to \infty} \delta(i,K, N)=0$. Therefore $\left(Q + \delta(i,K,N)\right)\tilde{m} - \mathbf{\Delta} r$ is a non-negative measure. As $F_i \subset A_i$ we compute
\[
0 \leq \int_{F_i}\left[  \left( Q + \delta(i,K,N)\right)\tilde{m} - \mathbf{\Delta} r \right] \leq \int_{A_i} \left[ \left( Q + \delta(i,K,N)\right)\tilde{m} - \mathbf{\Delta} r\right].
\]
Changing sign in the above equation and taking the average integral we find, 
\begin{eqnarray*}
\strokedint_{F_i} \left[ \mathbf{\Delta} r - \left( Q + \delta(i,K,N)\right)\tilde{m} \right]  & \geq &
 \frac{\tilde m(A_i)}{\tilde m (F_i)} \strokedint_{A_i}\left[ \mathbf{\Delta} r - \left( Q + \delta(i,K,N)\right)\tilde{m} \right] \\
& \geq & \frac{\tilde m(A_i)}{\tilde m (F_i)} ( Q - \varepsilon (i,K,N,R) - Q - \delta(i,K,N))\\
&\geq & - \frac{\varepsilon(i,K,N,R)+ \delta(i,K,N)}{C(K,N,R)}.
\end{eqnarray*}
From the first to the second line above we used $ \strokedint_{A_i} \mathbf{\Delta} r \geq Q  - \varepsilon (i,K,N,R)$, and from the second to the third, $\frac{\tilde m (F_i)}{\tilde m (A_i)}  \geq  C(K,N,R)$. Thus, 
\[
\strokedint_{F_i} \mathbf{\Delta} r \geq   Q  + \delta(i,K,N) - \frac{\varepsilon(i,K,N,R)+ \delta(i,K,N)}{C(K,N,R)}.
\]
\end{proof}

We are now ready to prove Lemma \ref{lem-u}, in essentially the same way as the corresponding part of \cite[Theorem 1.7]{Jiang}.

\begin{proof}[Proof of Lemma \ref{lem-u}]
Let $y_i \in \tilde X$ be as in Proposition \ref{prop4}. Then there exists a deck transformation (measure-preserving metric isometry) $\varphi_i : \tilde X \to \tilde X$ such that $\varphi_i(y_0)=y_i$.  Define $u_i : B(y_0,R) \to \mathbb R$ by $u_i(y)= r( \varphi_i(y))- \tilde d(o,y_i)$. As $B(y_0,R)$ is precompact and the $u_i$ are $1$-Lipschitz, by the Arzel\`a-Ascoli Theorem there is a subsequence of $u_i$  that converges to a  $1$-Lipschitz function $u_R$. 
To show that in fact $|\nabla u_R|=1$ $\tilde{m}$-a.e., note that the set 
$$\{ x \in B(y_0,R) \ |\   | \nabla u_R|(x) \, \ \text{and}\, \  |\nabla u_i|(x) \ \forall i \in \mathbb N \ \text{are well defined} \}$$
has full $\tilde m$ measure.  For any $x$ in this set and $i \in \mathbb N$, let $\gamma_i$ be a geodesic from $\varphi_i(x)$ to $o$. 
Then, $$u_i( (\varphi_i^{-1} \circ \gamma_i)_t) - u_i( (\varphi_i^{-1} \circ \gamma_i)_0)= r((\gamma_i)_t) - r((\gamma_i)_0)=\tilde d(o,(\gamma_i)_t) - \tilde d(o,(\gamma_i)_0)=-t .$$  Now, $\varphi_i^{-1} \circ \gamma_i$ subconverges to a geodesic $\alpha$. Thus, in the limit, we get  the previous inequality for $u_R$,  $u_R( \alpha_t) - u_R( \alpha_0)=-t $. From this we conclude that $ |\nabla u_R|(x)=1$.  Thus, $|\nabla u_R|=1$ $\tilde{m}$-a.e.

Moreover, the sequence $u_i$ is uniformly bounded in $W^{1,2}(B(y_0,R))$, so $u_R\in W^{1,2}(B(y_0,R))$
and
\[
\int_{\tilde{X}} \psi \mathbf{\Delta}u_R  = \lim_{i\to\infty} \int_{\tilde{X}} \psi \circ \varphi_i \mathbf{\Delta} u_i.
\]
Here $\psi$ is a compactly supported Lipschitz function on $B(y_0,R)$.

The Laplacian comparison \eqref{eq-lapComp} implies 
\begin{equation} 
\mathbf{\Delta} u_i(y)= \mathbf{\Delta} r(\varphi_i(y)) \leq Q + \Psi(i), \ y\in B(y_0,R).
\end{equation}
On the other hand,  Proposition \ref{prop4} gives, 
\[
\strokedint_{B(y_0,R)} \mathbf{\Delta} u_i = \strokedint_{B(y_i,R)} \mathbf{\Delta} r \geq Q  -\Psi(i).
 \]
 It follows then that 
 \[  \strokedint_{B(y_0,R)} \left| \mathbf{\Delta}  u_i - Q {\rm d}\tilde{m} \right| \leq  \Psi(i).  \]
  From these observations we obtain:
\begin{eqnarray*}
\int_{\tilde{X}} \psi \mathbf{\Delta}u_R  & = & \lim_{i\to \infty} \int_{B(y_0,R)}\psi\circ \varphi_i \mathbf{\Delta} u_i\\
& = & \int_{B(y_0,R)} \psi\circ \varphi\, Q\,\mathrm{d}\tilde{m}\\
& = & \int_{B(y_0,R)} \psi\, Q\,\mathrm{d}\tilde{m}.
\end{eqnarray*}
Whence, $u_R\in D(\mathbf{\Delta}, B(y_0,R))$ and $\mathbf{\Delta}u_R = Q\, \tilde{m}$. 
\end{proof}

Now Theorem~\ref{cor-u} is proved as follows. 
\begin{proof}[Proof of Theorem~\ref{cor-u}]
Take a sequence of radii $R_i\uparrow \infty$ and the corresponding sequence of functions $u_{R_i}$. Then, up to passing to a subsequence, the $u_{R_i}$ converge to a $1$-Lipschitz function $u:\tilde{X}\to \mathbb{R}$. Lemma~\ref{lem-u} immediately gives that $u\in D_{loc}(\mathbf{\Delta})$ and that $\mathbf{\Delta} u= Q\, \tilde{m}$, i.e. $\Delta u = Q$ $\tilde{m}$-a.e.  Finally, as we have seen in the proof of Lemma \ref{lem-u}, $|\nabla u_{R_i}|=1$ $\tilde{m}$-a.e.\ in the following stronger sense: For $\tilde{m}$-a.e.\ $x$, there is a geodesic $\gamma_i(t)$ through $x$ such that $u_{R_i}(\gamma_i(t))=u_{R_i}(\gamma_i(0)) -t$. Repeating the argument above in the proof of Lemma \ref{lem-u} gives $| \nabla u| = 1$ $\tilde{m}$-a.e.
\end{proof}


\subsection{The Hessian of $u$}
\label{ssec-hessian-u}
Throughout this section we maintain the assumption that $(X,d,m)$ is an $\RCDst(K,N)$ space with $K<0$ and $N\in (1,\infty)$. Let us recall that we denote the universal cover of $X$ by $(\tilde{X}, \tilde{d}, \tilde{m})$ and that $\tilde{X}$ is an $\RCDst(K,N)$-space \cite{MW}. In this section we will compute the Hessian of the Busemann-type function $u:\tilde{X}\to\mathbb{R}$ constructed in Section \ref{sec-busemann-function}. Throughout this section we reserve the notation $u$ for this function. The strategy and computations follow along the lines of \cite[Theorem 3.7]{Ket}, which in turn draws from \cite{Stu}, originally formulated in the language of $\Gamma$-Calculus.

Fix a point $x\in \tilde{X}$, and let $t\in \RR$. Let $v\in \mathrm{Test}_{loc}(\tilde{X})$,  $f,g \in \mathrm{Test}_{bs}(\tilde X)$, and consider the function
\[
\Psi(v,f,g)=\frac{1}{2}v^2+(1-v(x))v+ t\left( fg-f(x)g-g(x)f \right).
\]
Observe that $\Psi$ is a smooth function with $\Psi(0,0,0)=0$. The partial derivatives of $\Psi$ at $x$ are given by


\begin{table}[h]
\centering
\begin{tabular}{lll}
$\Psi_1|_x  = (v + (1-v(x)))|_x = 1$ & $\Psi_{11}|_x  =  1$               & $\Psi_{22}|_x  =  0$              \\
$\Psi_2|_x =  t(g-g(x))|_x=0$        & $\Psi_{12}|_x = \Psi_{21}|_x =  0$ & $\Psi_{23}|_x  =  \Psi_{32}|_x=t$ \\
$\Psi_3|_x  =  t(f-f(x))|_x=0$       & $\Psi_{13}|_x = \Psi_{31}|_x =  0$ & $\Psi_{33}|_x  =  0.$             
\end{tabular}
\end{table}

\vspace*{.05in}

We now turn our attention to the measure valued functional $\Gamma_2$ (see Section \ref{ss-bochner} for the definition) and we let $\gamma_2 \tilde{m}$ be its absolutely continuous part. Proposition \ref{prop:change_of_variables} guarantees that $\Psi(u,f,g)\in \mathrm{Test}_{loc}(\tilde{X})$ for every $f,g\in \mathrm{Test}_{bs}(\tilde{X})$. Therefore, following the same strategy as in \cite[Theorem 3.7]{Ket}, by (\ref{eq-bakry-emery-condition}) and Proposition \ref{prop:change_of_variables} we have that for every $x\in \tilde{X}$,

\begin{align*}
0  & \leq \gamma_2(\Psi(u,f,g)) -  K\left|\nabla \Psi(u,f,g) \right|^2 + \frac{1}{N}(\Delta \Psi(u,f,g))^2\\
 & =  \gamma_2(u) + 4t\mathrm{Hess}[u](f,g) + |\nabla u|^4 + 4t\left\langle \nabla u, \nabla f\right\rangle \left\langle \nabla u, \nabla g\right\rangle + 2t^2 |\nabla f|^2 |\nabla g|^2\\
& + 2t^2(\left\langle \nabla f, \nabla g\right\rangle )^2 - K|\nabla u|^2 - \frac{(\Delta u)^2}{N} - \frac{|\nabla u|^4}{N} -\frac{4t^2}{N}(\left\langle \nabla f, \nabla g\right\rangle)^2\\
&  - \frac{4t\Delta u}{N}\left\langle \nabla f, \nabla g\right\rangle -	 \frac{2\Delta u}{N}|\nabla u|^2 -\frac{4t}{N}|\nabla u|^2\left\langle \nabla f,\nabla g\right\rangle. 
\end{align*}
Grouping terms we obtain, 
\begin{align}
\label{eq-poli}
0 & \leq \gamma_2(u) - K|\nabla u|^2 - \frac{(\Delta u)^2}{N} + 			   \frac{N-1}{N}|\nabla u|^2 -\frac{2\Delta u}{N}|\nabla u|^2\\
& + 4t\left( \mathrm{Hess}[u](f,g) + \left\langle \nabla u, \nabla f\right\rangle \left\langle \nabla u, \nabla g\right\rangle - \left(\frac{\Delta u + |\nabla u|^2}{N}\right)\left\langle \nabla f,\nabla g\right\rangle  \right) \nonumber \\   
& + 2t^2\left( |\nabla f|^2|\nabla g|^2 + \frac{N-2}{N}					   (\left\langle \nabla f, \nabla g\right\rangle)^2 \right). \nonumber
\end{align}
The last term of  the previous inequality (\ref{eq-poli}), namely $|\nabla f|^2 |\nabla g|^2 + \frac{N-2}{N} (\left\langle \nabla f, \nabla g\right\rangle)^2$, is non-negative.  Hence, the discriminant of the right hand side of (\ref{eq-poli}) as a polynomial in $t$ is $\leq 0$.  That is, 
\begin{align*}
2\Bigg( \mathrm{Hess}[u](f,g) &  \left. + \left\langle \nabla u, \nabla f\right\rangle\left\langle \nabla u,\nabla g\right\rangle -  \left(\frac{\Delta u + |\nabla u|^2}{N}\right)\left\langle \nabla f,\nabla g\right\rangle\right)^2\\
 &\leq  \left(\gamma_2(u) - K |\nabla u|^2- \frac{(\Delta u)^2}{N}   + \frac{N-1}{N} |\nabla u|^2 -\frac{2\Delta u}{N}|\nabla u|^2\right)\left({|\nabla f|^2 |\nabla g|^2 + \frac{N-2}{N}(\left\langle \nabla f, \nabla g\right\rangle)^2}\right).
\end{align*}

Once this analysis has been performed we can explicitly compute the Hessian of $u$ as follows.
 
\begin{cor}
\label{cor-hessian-identity}
 Let $u\in \mathrm{Test}_{loc}(\tilde{X})$ with $|\nabla u|^2=1$ $\tilde m$-a.e. and $\Delta u = N-1$ $\tilde{m}$-a.e. Then for all functions $f,g\in \mathrm{Test}_{loc}(\tilde{X})$,
\begin{equation}\label{eq-Hessu}
\mathrm{Hess}[u](f,g)=\left\langle \nabla f, \nabla g\right\rangle-\left\langle \nabla u, \nabla f\right\rangle \left\langle \nabla u, \nabla g\right\rangle.
\end{equation}
\end{cor}

\begin{proof} We first consider $f,g\in \mathrm{Test}_{bs}(\tilde{X})$. 
Set $\Phi=|\nabla f|^2 |\nabla g|^2 + \frac{N-2}{N}(\left\langle \nabla f, \nabla g\right\rangle)^2$, and note that $\Phi$ has bounded support. Note also that $\Gamma_2(u)= 0$ and therefore $\gamma_2(u)=0$. Plugging this in our previous analysis and using that  $|\nabla u|^2=1$,  $\Delta u = N-1$ and $K=-(N-1)$ we have that 
\[
\dint_{\tilde{X}} {2\left( \mathrm{Hess}[u](f,g) + \left\langle \nabla u,\nabla f\right\rangle \left\langle \nabla u, \nabla g\right\rangle - 					   \left(\frac{\Delta u + |\nabla u|^2}{N}\right)\left\langle \nabla f, \nabla g\right\rangle \right)^2}\, \mathrm{d} \tilde{m}
\]
is less than or equal to 
\[
\dint_{\tilde{X}} \left(-K |\nabla u|^2 - \frac{(\Delta u)^2}{N} + \frac{N-1}{N} - \frac{2\Delta u}{N}\right)\Phi \, \mathrm{d}\tilde{m} =\dint_{\tilde{X}} \left( N-1 - \frac{(N-1)^2}{N} + \frac{N-1}{N}-\frac{2(N-1)}{N} \right)\Phi \, \mathrm{d}\tilde{m} = 0.
\]
Therefore,
\[
\mathrm{Hess}[u](f,g) + \left\langle \nabla u,\nabla f\right\rangle \left\langle \nabla u, \nabla g\right\rangle -  \left(\frac{\Delta u + |\nabla u|^2}{N}\right)\left\langle \nabla f, \nabla g\right\rangle =0
\]
and substituting the values of $|\nabla u|^2$ and $\Delta u$ in the previous equation we obtain that $\frac{\Delta u + |\nabla u|^2}{N}=1$, so that $\mathrm{Hess}[u](f,g)=\left\langle \nabla f, \nabla g\right\rangle-\left\langle \nabla u, \nabla f\right\rangle \left\langle \nabla u, \nabla g\right\rangle$.

For general $f,g\in \mathrm{Test}_{bs}(\tilde{X})$, we note that the result is local and the general case follows from a truncation argument using Lemma \ref{lem:test:loc}.
\end{proof}

\begin{rmk}
Note that the right hand side of \eqref{eq-Hessu} depends only on the $W^{1,2}$-norm of $f, g$. Thus we can use \eqref{eq-Hessu} to extend the definition of $\mathrm{Hess}\, u$ to $W^{1,2}(\tilde{X})$, which is what we adopt from now on.	
\end{rmk}


\subsection{Regular Lagrangian Flow of $\nabla u$}

In this section we will show the existence of a \textit{Regular Lagrangian Flow} of the gradient of the Busemann-type function $u:\tilde{X}\to\mathbb{R}$ constructed in the previous section, via the work developed by Ambrosio-Trevisan \cite{AT}. To do so, we make use of the reformulation of the results of Ambrosio-Trevisan obtained by Gigli-Rigoni \cite{GigRig}, which utilizes the language of Differential Calculus developed by Gigli, \cite{Gig2}. Let us recall the definition of a Regular Lagrangian Flow, following \cite{GigRig}.

\begin{defn}
\label{def-regular-lagrangian-flow}
Let $(X,d,m)$ be an $\mathsf{RCD}(K,N)$ space and $(V_t)\in L^2\left([0,1],L^2_{loc}(TX)\right)$. We say that 
\[
F^{(V_t)}: [0, 1] \times X \rightarrow X
\]
is a \textit{Regular Lagrangian Flow} for $(V_t)$ provided that:
\begin{itemize}
\item[i)] There exists $C > 0$ such that
\begin{equation}
(F^{(V_t)}_s)_{\sharp} m \leq C m, \ \ \forall s\in[0, 1]. 
\end{equation}
\item[ii)] For $m$-a.e. $x\in X$ the curve $[0, 1]\ni s\mapsto F^{(V_t)}_s(x)\in X$ is continuous and such that \[ F^{(V_t)}_0(x) = x.\]
\item[iii)] For every $f\in W^{1,2}(X)$ we have that for $m$-a.e. $x\in X$ the function $s\mapsto f(F^{(V_t)}_s(x))$ belongs to $W^{1,1}(0, 1)$ and satisfies
\begin{equation}
\label{eq-item-iii-rlf}
\frac{d}{ds} f(F^{(V_t)}_s (x)) = \mathrm{d}f(V_s)(F^{(V_t)}_s (x)), \ \  m \times \mathcal{L}^1|_{[0,1]}-a.e.(x, s).
\end{equation}
\end{itemize}
\end{defn}

With this definition in hand, we now recall the main existence and uniqueness result for Regular Lagrangian Flows in \cite{AT}, as expressed in \cite[Theorem 2.8]{GigRig}. Recall that the space of Sobolev vector fields $W^{1,2}_{C,loc}(TX)$ is the space of $V\in L^2_{loc}(TX)$ for which there is $T$ in the tensor product of $L^2(TX)$ with itself such that 
\[\int hT(\nabla g, \nabla \tilde{g})\, \mathrm{d}m = \int  \left\langle V,\nabla\tilde{g}\right\rangle \mathrm{div}(h \nabla g) + h \mathrm{Hess}(\tilde{g})(V,\nabla g)\, \mathrm{d}m
\]
for every $h, g, \tilde{g} \in \mathrm{Test}(X)$ with bounded support. In this case $T$ is the covariant derivative of $V$ and we denote it by $\nabla V$.

\begin{thm}[\cite{AT,GigRig}]
\label{thm-existence-and-uniqueness-rlf}
Let $(V_t)\in L^2\left([0,1], W^{1,2}_{C,loc}(TX)\right)\cap L^{\infty}\left([0,1],L^{\infty}(TX)\right)$ be such that $V_t\in D_{loc}(\mathrm{div})$ for a.e. $t\in [0,1]$, with 
\begin{equation}
\label{eq-rlf-condition}
\int_0^1||\  |\nabla V_t|\  ||_{L^2(X)} + ||\mathrm{div}(V_t)||_{L^2(X)} + ||(\mathrm{div}(V_t))^{-}||_{L^{\infty}(X)} \, \mathrm{d} t <\infty.
\end{equation} 
Then a Regular Lagrangian Flow $F^{(V_t)}$ for $(V_t)$ exists and is unique, in the sense that if $\tilde{F}^{(V_t)}$ is another flow, then $F^{(V_t)}_s(x)=\tilde{F}^{(V_t)}_s(x)$ for $m$-a.e. $x\in X$ and every $s\in [0,1]$.
\end{thm}

In the previous definition, a Regular Lagrangian Flow is associated to a time-dependent family of vector fields $(V_t)$. When dealing with a single vector field $V$ the hypotheses of the previous theorem simplify. To apply Theorem \ref{thm-existence-and-uniqueness-rlf}, it is enough that $V\in W^{1,2}_{C,loc}(T \tilde X)\cap L^{\infty}(T \tilde X)$ and that 
\begin{equation}
\label{eq-rlf-condition-no-time}
||\  |\nabla V|\  ||_{L^2(\tilde{X})} + ||\mathrm{div}(V)||_{L^2(\tilde{X})} + ||(\mathrm{div}(V))^{-}||_{L^{\infty}(\tilde{X})} <\infty.
\end{equation}

%

To proceed, let $B:= B(x_0,r)\subset \tilde{X}$ be a ball of some finite fixed radius $r>0$ centered at a fixed point $x_0$. Let $\rho=\rho^r$ be the good cut-off function such that $\rho =1$ on $B$ and vanishes outside the ball of radius $2r$, as in Lemma \ref{lem:good-cut-off}. We now consider the vector field $\rho \nabla u$ and show that it admits a Regular Lagrangian Flow by Theorem \ref{thm-existence-and-uniqueness-rlf}.

Let us first note that, indeed $\rho\nabla u\in W^{1,2}_{C,loc}(T \tilde X)\cap L^{\infty}(T \tilde X)$: From the proof of Lemma \ref{lem-u}, we have that $u\in \mathrm{Test}_{loc}(\tilde{X})$ and therefore, $\nabla u \in W^{1,2}_{C,loc}(T \tilde X)\cap L^{\infty}(T \tilde X)$. Moreover, since $\rho$ is Lipschitz and bounded, the Leibniz rule in \cite[Proposition 2.18]{Gig0} (note that, while the proposition requires the vector field to be in $W_C^{1,2}(\tilde{X})$, its proof makes sense almost verbatim for vector fields in $W_{C,loc}^{1,2}(\tilde{X})$) can be applied, yielding the claim. 

We now show that $\rho\nabla u$ satisfies \eqref{eq-rlf-condition-no-time}. Firstly, \cite[Proposition 2.18]{Gig0} can be applied to obtain that
\[ 
\nabla (\rho \nabla u) = \nabla \rho \otimes \nabla u + \rho\, \mathrm{Hess}(u), \ \ \  \ \ \   \mathrm{div} (\rho \nabla u) = \rho \Delta u + \langle \nabla \rho, \nabla u \rangle. 
\]
We use the first of these formulae in the following way: A bound for the Hilbert-Schmidt norm of $\mathrm{Hess}(u)$ is readily obtained from Bochner's inequality, Lemma \ref{lem:good-cut-off} guarantees that $|\nabla \rho|\leq C(K,N,R)$ for some fixed $R>r$ and  we showed before that $|\nabla u|= 1$. Then,  
\[
|\nabla (\rho \nabla u)| \leq |\nabla \rho \otimes \nabla u| + |\rho\, \mathrm{Hess}(u)| \leq C(K,N,R)+N-1<\infty.
\]
It follows that $\nabla (\rho \nabla u) \in L^2(\tilde{X})$. 

A similar reasoning coupled with the formula for $\mathrm{div} (\rho \nabla u)$ gives us that 
\[
|\mathrm{div} (\rho \nabla u)| \leq |\rho \Delta u| + |\langle \nabla \rho, \nabla u \rangle| \leq (N-1)+ C(K,N,R) <\infty.  
\]
Therefore, $\mathrm{div} (\rho \nabla u) \in L^2(\tilde{X})$. From this it is also immediate that $\mathrm{div} (\rho \nabla u)^{-} \in L^{\infty}(\tilde{X})$.

Hence Theorem \ref{thm-existence-and-uniqueness-rlf} applies and a Regular Lagrangian Flow $F^r:[0,1]\times \tilde{X}\rightarrow \tilde{X}$ for $\rho \nabla u$ exists and is unique in the sense of Definition \ref{def-regular-lagrangian-flow}.

The uniqueness of a Regular Lagrangian Flow is tied to the uniqueness of solutions of the \textit{continuity equation} (see for example \cite[Definition 2.9]{GigRig}). Given a metric measure space $(X,d,m)$, recall that two Borel maps $t\mapsto \mu_t\in \mathcal{P}(X)$ and $t\mapsto V_t\in L^2(TX)$ are said to solve the continuity equation 
\begin{equation}
\label{eq-continuity-equation}
\frac{d}{dt}\mu_t + \mathrm{div}(V_t\mu_t)=0
\end{equation}
provided that the following conditions are satisfied:
\begin{itemize}
\item[(i)] There exists $C>0$ such that $\mu_t\leq Cm$ for every $t\in [0,1]$,
\item[(ii)] $\int_0^1\int |V_t|^2\,\mathrm{d}\mu_t\,\mathrm{d}t<\infty$,
\item[(iii)] for any $f\in W^{1,2}(X)$ the map $t\mapsto \int f\, \mathrm{d}\mu_t$ is absolutely continuous and 
\[
\frac{d}{dt}\int f\,\mathrm{d}\mu_t = \int \mathrm{d}f(V_t)\, \mathrm{d}\mu_t \quad \text{a.e.}\  t.
\]
\end{itemize}
We will often refer to the continuity equation, \eqref{eq-continuity-equation}, as the continuity equation associated to $(V_t)$. 
The following result concerning the uniqueness of solutions of the continuity equation in connection with the uniqueness of Regular Lagrangian Flows was obtained in \cite{AT}. We recall the formulation of \cite[Theorem 2.10]{GigRig}. 

\begin{thm}
\label{thm-existence-and-uniqueness-cont-equ}
Let $(V_t)$ be as in Theorem \ref{thm-existence-and-uniqueness-rlf} and $\overline{\mu}\in \mathcal{P}(X)$ be such that  $\overline{\mu}\leq Cm$ for some $C>0$. Then there exists a unique $(\mu_t)$ such that the pair $(\mu_t, X_t)$ solves the continuity equation, \eqref{eq-continuity-equation}, and for which $\mu_0=\overline{\mu}$. Moreover, such $(\mu_t)$ is given by $\mu_s = (F_s^{(X_t)})_{\sharp}\overline{\mu}$ for all $s\in [0,1]$. 
\end{thm}

The following lemma and its proof regarding the local uniqueness of solutions of the continuity equation and that of Regular Lagrangian Flows were indicated to us by Ambrosio and Gigli. 

%

\begin{lem} \label{lem-local:unique}
 Let $(X,d,m)$ be an $\mathsf{RCD}(K,N)$ space and $(V_t), (W_t)\in L^2\left([0,1], L^2_{loc}(TX)\right)$ be such that $V_t=W_t$, $\tilde{m}$-a.e. on some open set $\Omega \subset X$ for all $t\in [0,1]$. Let $\overline{\mu}\in\mathcal{P}(X)$ be concentrated on some Borel set in $\Omega$ and assume that solutions $\mu_t$ and $\nu_t$ for the continuity equations associated to $(V_t)$ and $(W_t)$, respectively, with initial data $\overline{\mu}$ exist, are unique and are concentrated on Borel  subsets contained in $\Omega$ for all $t\in [0,1]$. Then  $\mu_t=\nu_t$ for all $t\in [0,1]$. Moreover, if $(V_t)\in L^{\infty}\left([0,1];L^{\infty}(TX) \right)$ and $(V_t)$ and $(W_t)$ admit Regular Lagrangian Flows $F^{(V_t)}$ and $F^{(W_t)}$ respectively, then for a.e. $x\in \Omega$ there exists $s_x\in (0,1]$ such that $F_s^{(V_t)}(x)=F_s^{(W_t)}(x)$ for all $s\leq s_x$.  
\end{lem}

\begin{proof}
Let $f\in W^{1,2}(X,d,m)$. Then, by the definition of solution of the continuity equation, the function $t\mapsto \int f\, \mathrm{d}\mu_t$ is absolutely continuous and satisfies 
\[
\frac{d}{dt}\int f\,\mathrm{d}\mu_t = \int \mathrm{d}f(V_t)\, \mathrm{d}\mu_t \quad \text{a.e.}\  t.
\]
Since $\mu_t$ is concentrated in $\Omega$ for all $t$, we have that 
\[
\frac{d}{dt}\int f\,\mathrm{d}\mu_t = \int \mathrm{d}f(W_t)\, \mathrm{d}\mu_t \quad \text{a.e.}\  t.
\]
Hence $\mu_t$ is a solution to the continuity equation associated to $(W_t)$. It then follows by uniqueness that $\mu_t=\nu_t$ for a.e.  $t\in [0,1]$ and by continuity that $\mu_t=\nu_t$ for all $t\in[0,1]$. 

To address the part about Regular Lagrangian Flows, let us recall that by \cite[Theorem 7.6]{AT} there exists $\eta\in \mathcal{P}(C([0,1];X)$ satisfying the following two conditions:
\begin{itemize}
\item[(i)] $\eta$ is concentrated on curves $\gamma\in C([0,1];X)$ satisfying that for all $f\in W^{1,2}(X)$, the function $t\mapsto f\circ \gamma(t)$ belongs to $W^{1,1}(0,1)$ and
\[
\frac{d}{dt} f\circ\gamma(t) = \mathrm{d}f(V_t)(\gamma(t)), \ \  \text{a.e.}\ t,
\]
\item[(ii)] $\mu_t= (e_t)_{\sharp}\eta$ for all $t\in [0,1]$. 
\end{itemize}
Note that the hypotheses of \cite[Theorem 7.6]{AT} are satisfied by the fact that $X$ is an $\mathsf{RCD}(K,N)$ space coupled with \cite[Lemma 9.2]{AT}. 

We now note that there exists a Borel map $T:X\to C([0,1];X)$ such that $\eta= T_{\sharp}\overline{\mu}$. This follows from \cite[Theorem 8.4]{AT} and the arguments in \cite[Theorem 18]{AC}, the idea being that, if such a map does not exist, then one can build two different solutions to the continuity equation associated to $(V_t)$, contradicting the uniqueness assumption. 
Once we have that $\eta= T_{\sharp}\overline{\mu}$, we obtain that $\mu_t= (e_t\circ T)_{\sharp}\overline{\mu}$. In particular, items (i) and (ii) stated above give us that $e_t\circ T$ is a Regular Lagrangian Flow for $(V_t)$. It follows that $F_s^{(V_t)}(x)= e_s\circ T(x)$ for $\overline{\mu}$-a.e. $x$ and all $s\in [0,1]$ as otherwise $(F_s^{(V_t)})_{\sharp}\overline{\mu}$ and $(e_s\circ T)_{\sharp}\overline{\mu}$ would be different solutions to the continuity equation associated to $(V_t)$. Therefore, $F^{(V_t)}$ is the unique Regular Lagrangian Flow associated to $(V_t)$. Similarly we conclude that $F_s^{(W_t)}$ is the unique Regular Lagrangian Flow associated to $(W_t)$.

We will now show that $F_s^{(V_t)}(x)= F_s^{(W_t)}(x)$ for $\overline{\mu}$-a.e. $x\in \Omega$ for all times $s\leq s_x$ for some $s_x\in (0,1]$. Recall that, Theorem 7.4 and Lemma 9.2 in \cite{AT} imply that, in fact, the curves $t\mapsto F^{(V_t)}_t(x)$ for each $x\in X$ are not only continuous but absolutely continuous and that 
\begin{equation}
\label{eq-metric-speed-norm-of-vector-field-rlf-cut-off-vectors}
| \dot{F}_s^{(V_t)}(x)| = |V_t|(F_s^{(V_t)}(x)) \quad \text{for a.e.}\ s\in[0,1].
\end{equation} 

Let $S$ be such that $|V_t|\leq S$ for all $t\in [0,1]$. By integrating the previous equation and using the definition of absolute continuity we obtain that 
\[
d\left(F^{(V_t)}_s(x), F^{(V_t)}_t(x)\right)\leq S |s-t| 
\]
for all $t<s$, with $t,s\in [0,1]$.

Using the previous inequality we claim the following: For each $x\in \Omega$ there exists a ball $B(x,r)$ centered at $x$ and some $\tilde{r}:=\tilde{r}(x,r)\in(0,1]$ such that $F^{(V_t)}_s(B(x,r))\subset \Omega$ for all $s\leq \tilde{r}$. We call such a ball well contained in $\Omega$ (with respect to $(V_t)$). Indeed, since $\Omega$ is open, there exists a ball $B(x,r_0)\subseteq \Omega$. Now for each $r<r_0$ and $y\in B(x,r)$ we have that 
\[
d(x,F^{(V_t)}_s(y))\leq d(x,y) + d(y,F^{(V_t)}_s(y)) \leq r + Ss.  
\]
Therefore, for every $\tilde{r}\in (0,1]$ with $0<\tilde{r}<\frac{r_0-r}{S}$, we have that 
\[
r+Ss \leq r+S\tilde{r} \leq r+(r_0-r)= r_0
\] 
so that $F^{(V_t)}_s(B(x,r))\subseteq B(x,r_0)\subseteq \Omega$ for all $s\in [0,1]$ with $s\leq \tilde{r}$. It is clear that $\Omega$ can be covered with well-contained balls. 


We will now define a map $F:[0,1]\times X\to X$. Let $B$ be any well-contained ball in $\Omega$ and for each $x\in B$ denote $t_x\in (0,1]$ the time such that $F^{(V_t)}_{t_x}(x)\notin B$ but $F^{(V_t)}_t(x)\in B$ for all $t<t_x$. Then we define $F$ as
\[
F(s,x):=
\begin{cases}
F^{(V_t)}_s(x) & \textrm{if}\ \text{$x\in B$ and $s<s_x$}  \\
F^{(W_t)}_s(x) & \textrm{if}\ \text{$x\in B$ but $s\geq s_x$ or $x\notin B$.} 
\end{cases}
\]
Let us observe that $F$ is clearly Borel measurable. Moreover, it is immediate to check that $F$ satisfies the definition of a Regular Lagrangian Flow for $(W_t)$. Therefore, it follows by uniqueness of the Regular Lagrangian Flows for $(W_t)$ that $F_t(x) = F^{(W_t)}_t(x)$ for $\overline{\mu}$-a.e. $x$ and by the definition of $F$ that $F^{(V_t)}_t(x) = F^{(W_t)}_t(x)$ for $\overline{\mu}$-a.e. $x\in B$ and $t<t_x$. As $\Omega$ can be covered with well-contained balls this immediately implies that  $F^{(V_t)}_t(x) = F^{(W_t)}_t(x)$ for $\overline{\mu}$-a.e. $x\in \Omega$ and all times $t\leq t_x$ for some $t_x\in (0,1]$. 
\end{proof}

\begin{rmk}
In typical applications of Lemma \ref{lem-local:unique}, \eqref{eq-rlf-condition} is satisfied so the existence and global uniqueness are guaranteed. If further $\Omega=B(x, r)$ is a ball, then the above argument shows that the local uniqueness of the continuity equation and the Regular Lagrangian Flow hold for a smaller ball $B(x, \tilde{r})$ for time $t \leq t_0$ provided $\tilde{r} + St_0 \leq r$. In our application, we have $S=1$ and we can work with $r>1$. Hence the local uniqueness holds for $B(x, r-1)$ and all $t \in [0, 1]$.
\end{rmk}

We now apply this lemma to our situation. Let $x_0\in \tilde{X}$ be the fixed point we chose after Theorem \ref{thm-existence-and-uniqueness-rlf}, consider an increasing sequence of radii $r_i\uparrow \infty$ with $r_1>1$ and $B_i:=B(x_0,r_i)$. It is clear that $\tilde{X}=\bigcup_{i=1}^{\infty} B_i$. We have already showed that a unique Regular Lagrangian Flow $F^{r_i}:[0,1]\times \tilde{X}\to \tilde{X}$ exists for each vector field of the form $\rho_{B_i}\nabla u$, where $\rho_{B_i}$ is the good cut-off function which is $\rho_{B_i}=1$ on $B_i$ and zero outside $B(x_0,2r_i)$. 

Let us fix an increasing sequence $\tilde{r}_i\uparrow \infty$ with $\tilde{r}_i<r_i-1$ for all $i$. Observe that for each $i$, $|\rho_{B_i}\nabla u|\leq 1$ and, moreover, following the terminology used in the proof of the previous lemma,  that the ball $B(x_0,\tilde{r}_i)$ is well contained in $B_i$ with respect to $\rho_{B_i}\nabla u$ for all times $s\in [0,1]$. Therefore, the previous lemma implies that for every $i<j$, both flows $F^{r_i}$ and $F^{r_j}$ coincide for $\tilde{m}$-a.e. $x\in B(x_0,\tilde{r}_i)$ for all times $s\in [0,1]$.

 Observe now that for each ball $B=B(x_0,r_i)$ considered above, and $t \leq \tilde{r}_i<r_i-1$,  $\mu_t=e^{-(N-1)t}\tilde{m}$ is a solution to the continuity equation for $V_t= \rho_{B_i} \nabla u $ on $B(x_0,\tilde{r}_i)$ with initial data $\overline{\mu}=\tilde{m}$. Hence, it follows from the previous theorem that on $B(x,r_i)$ and for  $t \leq \tilde{r}_i$, 
\begin{equation}
\label{eq-change-measure-flow}
(F_t)_{\sharp}\tilde{m}=e^{-(N-1)t}\tilde{m}.
\end{equation} 
Here we have used the fact that one can drop the requirement that the measures are probability measures in the previous theorem as the continuity equation (iii) implies that the total measure is preserved. 

We can then define a map $F:[0,1]\times \tilde{X}\to \tilde{X}$ (well defined up to a $\tilde{m}$-zero measure set) as 
\[
F_t(x):= F^{r_i}_t(x) \quad \text{if $x\in B(x_0,\tilde{r}_i)$}.
\]
We claim that $F$ is a Regular Lagrangian Flow for $\nabla u$. Item (i) of Definition \ref{def-regular-lagrangian-flow} follows from having that $(F_t)_{\sharp}\tilde{m} = e^{-(N-1)t}\tilde{m}$ and item (ii) of the same definition is immediately verified. For item (iii) of Definition \ref{def-regular-lagrangian-flow}, we know that 
\[
\frac{d}{ds} f(F_s (x)) = \rho_{B_i}\mathrm{d}f(\nabla u)(F_s (x)), \ \  m \times \mathcal{L}^1|_{[0,1]}-a.e.(x, s).
\]
So it suffices to take the limit when $i\to \infty$. 

As we are dealing with a single vector field $\nabla u$ (that is, in the notation of Definition \ref{def-regular-lagrangian-flow}, $V_t$ is independent of the time variable $t$), $F$ can be extended uniquely to a Regular Lagrangian Flow $F:[0,\infty)\times \tilde{X}\rightarrow \tilde{X}$. Furthermore, observe that the proof of \cite[Lemma 3.18]{GigRig} can be applied verbatim to our case and therefore, $F$ can be extended uniquely (preserving the rate of change in measure expressed in \eqref{eq-change-measure-flow} to a Regular Lagrangian Flow $F:(-\infty,\infty)\times\tilde{X}\to\tilde{X}$. 

Notice that the uniqueness statement in \cite[Theorem 2.8]{GigRig} implies that for $(V_t)$ independent of $t$, $F$ satisfies the semigroup property $F_t\circ F_s = F_{t+s}$, $\tilde{m}$-a.e. and for all $t,s\in \mathbb{R}$ (cf. \cite[Equation 2.3.10]{GigRig}).


We summarize the previous discussion in the following proposition. 

\begin{prop}\label{prop-F_tPushMeas}
Let $u:\tilde{X}\to \mathbb{R}$ be the function constructed in Theorem \ref{cor-u}. Then, there exists a unique Regular Lagrangian Flow (in the sense of Definition \ref{def-regular-lagrangian-flow}) $F:\mathbb{R}\times \tilde{X}\to \tilde{X}$ for $\nabla u$. Moreover, $F$ satisfies the semigroup property $F_t\circ F_s = F_{t+s}$, $\tilde{m}$-a.e. for all $t,s \in \mathbb{R}$, and the following change of measure formula holds,
\[ (F_t)_{\sharp}\tilde{m}=e^{-(N-1)t}\tilde{m}. \] 
\end{prop}

We end this section by pointing out that the following lemma holds in our setting (cf. \cite[Theorem 2.3 (iv)]{Gig}).

\begin{lem} 
\label{lem-distance-flow-representatives}
Let $u:\tilde{X}\to \mathbb{R}$ be the function constructed in Theorem \ref{cor-u} and $F:(-\infty, \infty)\times \tilde{X}\to \tilde{X}$ be the Regular Lagrangian Flow associated to $\nabla u$. Then, for every $t,s\in (-\infty,\infty)$ and $x\in \tilde{X}$, 
\[
\tilde d(F_{s}(x),F_{t}(x))=|s-t|=|u(F_s(x))-u(F_t(x))|.
\]
In particular, $u(F_{-u(x)}(x))=0$ for all $x \in \tilde X$ and the trajectories of $F_t$ are geodesics.
\end{lem}

\begin{proof}
Following the approach of the proof of $(a)\Rightarrow (b)$ in \cite[Proposition 2.7]{GigRig}, from \eqref{eq-item-iii-rlf} we obtain that, for all $t<s$,
\[
u\circ F_s - u\circ F_t = \int_t^s du(\nabla u)\circ F_r \, \mathrm{d}r.
\] 
Inverting the roles of $t$ and $s$, and using that $|\nabla u|=1$ $\tilde{m}$-a.e., 
\[
\left|u\circ F_s - u\circ F_t \right| = \left| s-t\right|. 
\]
Furthermore, since $|\dot{F}_s^{(X_t)}(x)| = |X_s|(F_s^{(X_t)}(x))$ for a.e. $s\in[0,1]$, we have that $\tilde{d}(F_{s}(x),F_{t}(x))\leq |s-t|$ for all $t<s$. Moreover, 
\[\left|u\circ F_s(x) - u\circ F_t(x) \right|\leq  \tilde d(F_{s}(x),F_{t}(x))\]
 because $u$ is $1$-Lipschitz. Therefore $ \tilde d(F_{s}(x),F_{t}(x))= |s-t|$.
\end{proof}


\section{Cheeger energy along the flow}\label{sec-CheegerE}

Consider the map $f_t=f\circ F_t$, where $F:(-\infty,\infty)\times\tilde{X}\to \tilde{X}$ is the Regular Lagrangian Flow of the Busemann-type function $u:\tilde{X}\to\mathbb{R}$ obtained in the previous section and $f\in W^{1,2}(\tilde{X})$. In this section we focus on computing the $W^{1,2}(\tilde X)$ norm of $f_t$. Here we will first resolve the regularity and show that if $f\in W^{1,2}(\tilde{X})$, then $f_t\in W^{1,2}(\tilde{X})$ as well. For this purpose, we need to use the heat flow to regularize $f_t$ first and adopt the techniques developed by \cite{DePG} to our setting. 
In the process we obtain the derivative of the Cheeger energy along $F_t$. Finally we integrate and localize the result.

\subsection{Derivative of the Cheeger energy along the flow}

Let us consider the map $f_t=f\circ F_t$ where $F_t$ is the  Regular Lagrangian Flow of the Busemann-type function $u$ of Theorem \ref{cor-u} and $f\in W^{1,2}(\tilde{X})$. In this section we study the $W^{1,2}$ norm of $f_t$. For that reason we begin by proving a version of \cite[Equation 3.39]{Gig} in our setting. 

\begin{lem}
\label{lem-test-plans}
For any $f\in S^2(\tilde{X}, \tilde d,\tilde{m})$ and $t\geq 0$, 
\[
\left| f(F_t(x))-f(x)\right| \leq \dint_0^t|\nabla f|(F_s(x))\, \mathrm{d}s
\]
for $\tilde{m}$-a.e.\ $x\in \tilde{X}$. Furthermore, the result also holds for $t\leq 0$ by taking the integral from $t$ to $0$.
\end{lem}

\begin{proof}
Let us consider a probability measure $\bar{m}$ on $\tilde{X}$ satisfying $\bar{m}\leq \tilde{m}$ and $\tilde{m}<\!< \bar{m}$. We define the measure $\pi:= T_{\sharp}\bar{m}\in \mathcal{P}(C([0,1];\tilde{X}))$ where $T: \tilde{X}\to C([0,1],\tilde{X})$ is given by $T(x)_t= F_t(x)$.
Recall that $e_t:C([0,1];\tilde{X})\to \tilde{X}$ denotes the evaluation map at $t$.  Notice that for all $t\geq 0$,
\[
(e_t)_{\sharp}\pi = (F_t)_{\sharp}\bar{m}\leq (F_t)_{\sharp}\tilde{m} = e^{-(N-1)t}\tilde{m} \leq \tilde{m}. 
\]
So $\pi$ is a test plan (with compression constant $\leq 1$). Denote the set of trajectories of $F$ by $\Gamma_F$. Observe that for any set of curves $\Gamma\subset AC([0,1];\tilde{X})$, the point $x$ lies in $T^{-1}(\Gamma)$ if and only if there exists $\gamma\in \Gamma$ such that $\gamma(t)=F_t(x)$ for any $t\in [0,1]$. Hence, such a $\gamma$ is an element of $\Gamma_F$. It follows that $T^{-1}(\Gamma)=T^{-1}(\Gamma\cap \Gamma_F)$ and we find that $\pi$ concentrates on trajectories of $F$. By Lemma \ref{lem-distance-flow-representatives} the elements of $\Gamma_F$ are constant speed geodesics satisfying that $d(\gamma(1),\gamma(0))=1$. In particular $\pi$ is concentrated on $1$-Lipschitz curves. 

For any $\Gamma\subset C([0,1];\tilde{X})$
\[
(e_t)_{\sharp}\pi(\Gamma) = \bar{m}(T^{-1}(e_t^{-1}(\Gamma))=(F_t)_{\sharp}\bar{m}(\Gamma) .
\]  
By \cite[(3.7)]{Gig}, for $0\leq t\leq 1$ and $f\in S^2( \tilde X)$,  for $\pi$-a.e. $\gamma$,
\[
|f(\gamma(t))-f(\gamma(0))| \leq \dint_0^t|\nabla f|(\gamma(s))|\gamma'(s)|ds = \dint_0^t|\nabla f|(\gamma(s))ds.
\]

Therefore, using that for $\tilde{m}$-a.e.\ $x\in \tilde{X}$ the flow $F$ is defined, and therefore for almost every $x$ there is a trajectory of $F$ passing through it, for every $0\leq t \leq 1$,
\begin{equation}
\label{eq-t-less-1}
|f(F_t(x))-f(x)| \leq \dint_0^t|\nabla f|(F_s(x))\, \mathrm{d}s.
\end{equation}
 An iteration of this argument will yield the result for any $t\in \mathbb{R}$. Let  $1\leq t \leq 2$, then by \eqref{eq-t-less-1}, 
\[
\dint_{\tilde{X}}|f(F_{t-1}(x))-f(x)|\ \mathrm{d}\tilde{m} \leq \dint_{\tilde{X}}\!\!\!\dint_0^{t-1}|\nabla f|(F_s(x))\ \mathrm{d}s\ \!\mathrm{d}\tilde{m}.
\]
A direct computation yields that the left-hand side of the previous inequality equals 
\[
\dint_{\tilde{X}}|f(F_{t-1}(x))-f(x)|\ \mathrm{d}\tilde{m} = \dint_{\tilde{X}}|f(F_{t}(x))-f(F_1(x))|\ \mathrm{d}(F_{-1})_{\sharp}\tilde{m} = e^{(N-1)}\dint_{\tilde{X}}|f(F_{t}(x))-f(F_1(x))|\ \mathrm{d}\tilde{m}  
\]
Moreover, by \eqref{eq-t-less-1} the right hand side becomes
\[
\dint_{\tilde{X}}\!\!\!\dint_0^{t}|\nabla f|(F_s(x))\ \mathrm{d}s\ \!\mathrm{d}\tilde{m} - \dint_{\tilde{X}}\!\!\!\dint_0^{1}|\nabla f|(F_s(x))\ \mathrm{d}s\ \!\mathrm{d}\tilde{m} \leq \dint_{\tilde{X}}\!\!\!\dint_0^{t}|\nabla f|(F_s(x))\ \mathrm{d}s\ \!\mathrm{d}\tilde{m} - \dint_{\tilde{X}} |f(F_1(x) - f(x)|\ \mathrm{d}\tilde{m}.
\]
Combining the previous equations, using that $e^{-(N-1)}\leq 1$, and the triangle inequality we obtain:
\begin{eqnarray*}
\dint_{\tilde{X}}|f(F_{t}(x)) - f(x)| &\leq & \dint_{\tilde{X}}|f(F_{t}(x))-f(F_1(x))|\ \mathrm{d}\tilde{m} + \dint_{\tilde{X}} |f(F_1(x) - f(x)|\ \mathrm{d}\tilde{m}  \\ 
 & \leq & e^{-(N-1)}\dint_{\tilde{X}}|f(F_{t-1}(x))-f(x)|\ \mathrm{d}\tilde{m} + \dint_{\tilde{X}} |f(F_1(x) - f(x)|\ \mathrm{d}\tilde{m}\\
  & \leq & \dint_{\tilde{X}}\!\!\!\dint_0^{t-1}|\nabla f|(F_s(x))\ \mathrm{d}s\ \!\mathrm{d}\tilde{m} + \dint_{\tilde{X}} |f(F_1(x) - f(x)|\ \mathrm{d}\tilde{m}\\
  &\leq  & \dint_{\tilde{X}}\!\!\!\dint_0^{t}|\nabla f|(F_s(x))\ \mathrm{d}s\ \!\mathrm{d}\tilde{m}
\end{eqnarray*}
This is precisely the result we claim in the case that  $1\leq t\leq 2$. Iterating this process the inequality follows for any $t\geq 0$, and similarly for any $t\leq 0$. 
\end{proof}

This implies a version of \cite[(3.40)]{Gig} with appropriate modifications, as will be shown in the next lemma.

\begin{lem}
For any $f\in S^2(\tilde{X}, \tilde d,\tilde{m})$ and $t\in \mathbb{R}$,
\[
\dint_{ \tilde X}  |f(F_t(x))-f(x)|^2 \ \mathrm{d}\tilde{m}(x) \leq t\left(\frac{1-e^{-(N-1)t}}{N-1} \right)\dint_{ \tilde X}  |\nabla f|^2(x)\ \mathrm{d}\tilde{m}(x).
\]
\end{lem}

\begin{proof}
Taking squares, integrating the inequality of Lemma \ref{lem-test-plans}, and using H\"older's inequality we obtain: 
\begin{eqnarray*}
\dint_{ \tilde X}  |f(F_t(x))-f(x)|^2\ \mathrm{d}\tilde{m}(x) & \leq & \dint_{ \tilde X}  \left(\dint_0^t|\nabla f|(F_s(x))ds \right)^2 \ \mathrm{d}\tilde{m}(x)  \leq  t\dint_{ \tilde X} \!\!\!\dint_0^t |\nabla f|^2(F_s(x))\ \mathrm{d}s\ \mathrm{d}\tilde{m}(x)\\
& \leq & t\dint_0^t\!\!\!\dint_{ \tilde X}  |\nabla f|^2(F_s(x))\ \mathrm{d}\tilde{m}(x)\ \mathrm{d}s  = t\dint_0^t\!\!\!\dint_{\tilde X}  |\nabla f|^2(x)\ \mathrm{d}(F_s)_{\sharp}\tilde{m}(x)\ \mathrm{d}s \\
& =& t\dint_0^t \dint_{\tilde X}  e^{-(N-1)s} |\nabla f|^2(x)\ \mathrm{d}\tilde{m}(x)\ \mathrm{d}s\\ & = & t\left(\dint_0^te^{-(N-1)s}\ \mathrm{d}s \right)\left(\dint_{\tilde X}  |\nabla f|^2(x)\ \mathrm{d}\tilde{m}(x) \right)\\
& = &  t\left(\frac{1-e^{-(N-1)t}}{N-1} \right)\dint_{\tilde X}  |\nabla f|^2(x)\ \mathrm{d}\tilde{m}(x) 
\end{eqnarray*}
\end{proof}

In the following lemma we compute the $L^2$ norm of $f\circ F_t$ and investigate its regularity as a function of $t$.
\begin{lem}
\label{lem-f-composed-with-flow}
Let $f\in W^{1,2}(\tilde{X})$. Then $f\circ F_t\in L^2(\tilde{X}, \tilde{m})$ for every $t\in \mathbb{R}$. Moreover, the map $t\mapsto f\circ F_t$ is locally Lipschitz. 
\end{lem}

\begin{proof}
First we compute the $L^2$ norm of $f\circ F_t$:
\begin{equation}
\label{eq-L2-norm-ft}
\| f\circ F_t \|_{L^2}^2= \dint_{\tilde X} (f\circ F_t)^2\ \mathrm{d} \tilde m = \dint_{\tilde X}  f^2 e^{-(N-1)t} \ \mathrm{d} \tilde m= e^{-(N-1)t} \| f\|_{L^2}^2 
\end{equation}
Therefore, as $f\in W^{1,2}(\tilde X)$ and in particular $f\in L^2(\tilde X)$ it follows that $f\circ F_t\in L^2(\tilde X)$. Now we proceed with the second part of the lemma. Let $t<s\in \mathbb{R}$, by the previous lemma,
\begin{align*}
\dint_{\tilde X}  |f\circ F_s - f\circ F_t|^2\ \mathrm{d}\tilde{m} & =  \dint_{\tilde X}  e^{-(N-1)t} |f\circ F_{s-t} - f|^2\ \mathrm{d}\tilde{m}\\ 
& \leq   e^{-(N-1)t}(s-t)\left(\frac{1-e^{-(N-1)(s-t)}}{N-1} \right) \dint_{\tilde X}  |\nabla f|^2\ \mathrm{d}\tilde{m} \\
& =  (s-t)\left(\frac{e^{-(N-1)t}-e^{-(N-1)s}}{N-1} \right)\dint_{\tilde X}  |\nabla f|^2\ \mathrm{d}\tilde{m}\\
& \leq (s-t)^2 e^{-(N-1)t} \| \nabla f \|_{L^2}^2.
\end{align*}
Hence, $t\mapsto f\circ F_t$ is locally Lipschitz with Lipschitz constant dominated by $e^{-(N-1)t} \| \nabla f\|_{L^2}$ (which is well defined because $f\in W^{1,2}(X)$).
\end{proof}

Before we embark on the main estimate of this section, we state the following lemma, a version of \cite[(4.34)]{Gig}, (see also \cite[Lemma 3.11]{DePG}) which holds (with the same proof) in our setting. The proof requires (local) Lipschitz regularity of the function $t\mapsto f\circ F_t$ (for $f\in L^2(\tilde{X})$) and a bound on the change of measure along the flow $(F_t)_{\#}\tilde{m}\leq C(t)\tilde{m}$ which we have by the previous Lemma and Proposition \ref{prop-F_tPushMeas}. Once these results are at our disposal, the result is obtained essentially by an application of Proposition \ref{prop-weakUg} and the first differentiation formula \eqref{eq-first-differentiation-formula}. It improves our previous lemmas on the $t$-regularity and provides a derivative formula which we will need.
Let us point out however, the change of sign in \eqref{eq-derivative-along-flow} with respect to \cite[(4.34)]{Gig} and \cite[Lemma 3.11]{DePG}, due to the fact that our flow goes \textit{in the direction of $\nabla u$}, while the corresponding flow in the aforementioned references goes in the opposite direction.

\begin{lem}
\label{lem-derivative-along-flow}
Let $f\in W^{1,2}(\tilde{X})$. Then the map $t\mapsto f\circ F_t\in L^2(\tilde{X})$ is of class $C^1$ and its derivative is given by 
\begin{equation}
\label{eq-derivative-along-flow}
\frac{\mathrm{d}}{\mathrm{d}t}f\circ F_t = \left\langle \nabla f,\nabla u\right\rangle\circ F_t.
\end{equation}
\end{lem}

In the remaining part of this subsection we will provide an estimate on the energy of $f_t$, which will allow us to conclude that for every $f\in W^{1,2}(\tilde{X})$, $f_t\in W^{1,2}(\tilde{X})$ as well, for $t\leq 0$.
To regularize we make use of the \textit{heat flow} $h_t:L^2(\tilde{X})\to L^2(\tilde{X})$. Recall that $h_t$ is the unique family of maps such that for any $f\in L^2(\tilde{X})$ the curve $[0,\infty)\ni t\mapsto  h_t(f)\in L^2(\tilde{X})$ is continuous, locally absolutely continuous on $(0,\infty)$, satisfies that $h_0(f)=f$, $h_t(f)\in D(\Delta)$ for $t>0$ and solves
\[
\frac{\mathrm{d}}{\mathrm{d}t} h_t(f) =\Delta h_t(f), \quad \mathcal{L}^1-\text{a.e.}\ t > 0. 
\]
We refer the reader to \cite[Section 4.1.2]{Gig} for a thorough exposition of the main properties of the heat flow on infinitesimally Hilbertian metric measure spaces. 

\begin{lem}
\label{lem-energy-of-heat-composed-ft-is-lipschitz}
 For each $t\geq 0$, let $h_t:L^2(\tilde{X})\to L^2(\tilde{X})$ be the heat flow on $\tilde{X}$ and $\varepsilon>0$ be fixed. Then the map $t\mapsto h_{\varepsilon}(f\circ F_t) \in W^{1,2}(\tilde{X})$ is Lipschitz and, in particular, the map 
 \[
 t\mapsto \frac{1}{2}\dint_{\tilde X}  | \nabla h_{\varepsilon}(f\circ F_t)|^2 \,\mathrm{d}\tilde{m}
 \] is Lipschitz. 	 
\end{lem}

\begin{proof}
Using the equivalence of (i) and (v) in \cite[Theorem 7]{EKS} and the fact that $\mathsf{BL}(K,N)$ implies $\mathsf{BL}(K,\infty)$, \cite[Corollary 6.3]{AT} implies that the $L^2\!-\!\Gamma$ inequality holds true. Therefore, 
\[
\| | \nabla\left( h_{\varepsilon}(f\circ F_s) - h_{\varepsilon}(f\circ F_t)\right) | \|_{L^2}\leq C(\varepsilon) \|f\circ F_s - f\circ F_t \|_{L^2}.
\]
(See \cite[Definition 5.1]{AT} for the precise value of $C(\varepsilon)$). Moreover, by \cite[(3.1.2)]{Gig2}
\[
\| h_{\varepsilon}(f\circ F_s- f\circ F_t) \|_{L^2} \leq \|f\circ F_s- f\circ F_t \|_{L^2}.
\]
Combining the previous inequalities, we find: 
\begin{equation}\label{eq-LipHeatF_t}
\| h_{\varepsilon}(f\circ F_s) - h_{\varepsilon}(f\circ F_t) \|_{W^{1,2}}\leq C(\varepsilon) \|f\circ F_s - f\circ F_t \|_{L^{2}}. \qedhere
\end{equation}
\end{proof}


The following Euler-type equation for $u$ is also needed in our estimate.  For $f\in \mathrm{Test}(\tilde{X})$, it allows us to compute the difference between terms of the form $\left\langle \nabla u, \nabla \Delta f\right\rangle$ and $\Delta\left\langle \nabla u, \nabla f\right\rangle$.

\begin{prop}[Euler equation for $u$]
\label{prop-euler-equation-u}
Let $u:\tilde{X}\to \mathbb{R}$ be the Busemann-type function of Theorem \ref{cor-u} and $f\in \mathrm{Test}(\tilde{X})$. Then the following identity holds true $\tilde{m}$-a.e.
\[
\Delta \left\langle \nabla u, \nabla f\right\rangle =  \left\langle \nabla u, \nabla \Delta f\right\rangle + 2\Delta f -2(N-1)\left\langle \nabla u, \nabla f\right\rangle - 2\left\langle \nabla u,\nabla \left\langle \nabla f, \nabla u\right\rangle\right\rangle.
\]
\end{prop}

\begin{proof}
The proof uses the same strategy as \cite[Proposition 3.12]{DePG}. Let us consider the modified function $e^u=\exp\circ u:\tilde{X}\to \mathbb{R}$. The chain rule (see for example \cite[Theorem 1.12]{Gig0}) implies that $e^u\in W^{1,2}_{loc}(\tilde{X})$ and that $\nabla e^u = e^u\nabla u$. Moreover, we claim that $e^u\in D_{loc}(\Delta)$. Indeed, given $f\in \mathrm{Test}_{bs}(\tilde{X})$, we have that
\[
\dint_{\tilde{X}}\left\langle \nabla f, \nabla e^u\right\rangle \,\mathrm{d}\tilde{m} = \dint_{\tilde{X}} e^{u}\left\langle \nabla f,\nabla u\right\rangle \,\mathrm{d}\tilde{m} = \dint_{\tilde{X}} \left\langle \nabla (e^{u}f) - e^{u}f\nabla u, \nabla u \right\rangle \, \mathrm{d}\tilde{m} =-\dint_{\tilde{X}}Ne^u f\,\mathrm{d}\tilde{m}.
\]
where we used the Leibniz rule (see \cite[Equation 4.16]{Gig15}) and the fact that $\Delta u =N-1$. The previous identities also show that $\Delta e^u = Ne^{u}$. 

We now let $\varepsilon>0$ and apply Bochner's inequality (Theorem \ref{thm-weak-bochner-inequality}) to $e^{u} + \varepsilon f$ where $f\in \mathrm{Test}(\tilde{X})$ against non-negative test functions $g\in \mathrm{Test}_{bs}(\tilde{X})$ with $g\in L^{\infty}(\tilde{X},\tilde{m})$ and $\Delta g\in L^{\infty}(\tilde{X},\tilde{m})$ obtaining the following inequality, valid  $\tilde{m}$-a.e. 
\begin{align*}
\varepsilon \Delta \left\langle \nabla e^u, \nabla f \right\rangle + \varepsilon^2 \Delta |\nabla f|^2 \geq & \varepsilon\left( \left\langle \nabla e^u, \nabla \Delta f\right\rangle + \left\langle \nabla f, \nabla \Delta e^u \right\rangle +2\frac{\Delta e^u\Delta f}{N} -2(N-1)\left\langle \nabla e^u,\nabla f\right\rangle \right) \\ +& \varepsilon^2\left( \left\langle \nabla f, \nabla \Delta f\right\rangle + \frac{(\Delta f)^2}{N} -(N-1)|\nabla f|^2  \right).
\end{align*}

We now divide by $\varepsilon$ and take the limit when $\varepsilon\to 0$ to obtain 
\[
\Delta\left\langle \nabla e^u, \nabla f\right\rangle \geq \left\langle \nabla e^u, \nabla \Delta f\right\rangle + \left\langle \nabla f, \nabla \Delta e^u \right\rangle +2\frac{\Delta e^u\Delta f}{N} -2(N-1)\left\langle \nabla e^u,\nabla f\right\rangle
\]

By substituting the values of $\nabla e^u$ and $\Delta e^u$ the previous inequality becomes 
\[
\Delta \left\langle \nabla u, \nabla f\right\rangle  + 2\left\langle \nabla u, \nabla \left\langle \nabla u, \nabla f\right\rangle \right\rangle \geq \left\langle \nabla u, \nabla \Delta f\right\rangle + 2\Delta f - 2(N-1)\left\langle \nabla u, \nabla f\right\rangle 
\]

An analogous argument by using $\varepsilon<0$ yields the other inequality, valid $\tilde{m}$-a.e.,
\[
\Delta \left\langle \nabla u, \nabla f\right\rangle  + 2\left\langle \nabla u, \nabla \left\langle \nabla u, \nabla f\right\rangle \right\rangle \leq \left\langle \nabla u, \nabla \Delta f\right\rangle + 2\Delta f - 2(N-1)\left\langle \nabla u, \nabla f\right\rangle. \qedhere
\]
\end{proof}	

Before proceeding, let us remark that the term $\left\langle \nabla u,\nabla \left\langle \nabla f, \nabla u\right\rangle\right\rangle$ in the previous inequality makes sense since $f\in \mathrm{Test}(\tilde{X})$ and $u\in \mathrm{Test}_{loc}(\tilde{X})$, so that $\left\langle \nabla f, \nabla u\right\rangle\in W^{1,2}_{loc}(\tilde{X})$.

In what follows we adopt the convention that for $f \in L^2(\tilde{X})$, $\dint_{\tilde{X}}|\nabla f|^2\,\mathrm{d}\tilde{m}=\infty$ if $f\not\in W^{1,2}(\tilde{X})$. 

\begin{thm}
\label{thm-f_t-is-Sobolev}
 Let $f\in  W^{1,2}(\tilde{X})$ and let $\mathcal{E}(t):=\frac{1}{2}\dint_{\tilde{X}}|\nabla f_t|^2\,\mathrm{d}\tilde{m}$. Then, for all $t\leq 0$, 
\[
\mathcal{E}(f_t) \leq e^{-(N-1)t}\mathcal{E}(f).
\]
In particular, $f_t\in W^{1,2}(\tilde{X})$ for all $t\leq 0$.
Moreover,  we have 
\begin{align} \label{eq-derivative-of-cheeger-energy}
\frac{\mathrm{d}}{\mathrm{d}t} \mathcal E (t) = -(N-3)\mathcal E (t)   - \dint_{\tilde X}   \left\langle\nabla f_{t}, \nabla u\right\rangle^2\, \mathrm{d}\tilde{m}.
\end{align}
\end{thm}

The proof depends on several lemmas. Let us consider the function 
\[
G(t,s):= \dint_{\tilde{X}}\frac{|f_t|^2-|h_sf_t|^2}{4s}\,\mathrm{d}\tilde{m}.
\]
Note that $G(t,s)\uparrow \frac{1}{2}\dint_{\tilde{X}}|\nabla f_t|^2\,\mathrm{d}\tilde{m}$ as $s\downarrow 0$. Thus we are interested in an uniform bound on $G(t,s)$. 
Note also that by Lemmas \ref{lem-f-composed-with-flow} and \ref{lem-energy-of-heat-composed-ft-is-lipschitz}, for each $s>0$, the function $t\mapsto G(t,s)$ is locally Lipschitz. 

\begin{lem}
	\label{lem-derivative-of-G}
We have
\begin{align}
\label{eq-derivative-of-G-final}
\frac{\mathrm{d}}{\mathrm{d}t} G(t,s) = & -(N-1)G(t,s) - \frac{1}{s} \dint_0^s\! \dint_{\tilde{X}} h_{s-\tau} f_t \Delta h_{s+\tau} f_t \,\mathrm{d}\tilde{m}\,\mathrm{d}\tau \nonumber \\
& -\frac{1}{s} \dint_0^s\! \dint_{\tilde{X}} \left\langle \nabla h_{s+\tau}f_t,\nabla u\right\rangle \left\langle \nabla h_{s-\tau}f_t,\nabla u\right\rangle \, \mathrm{d}\tilde{m} \, \mathrm{d}\tau. 
\end{align}
\end{lem}

\begin{proof}
By \eqref{eq-L2-norm-ft},
\begin{equation}
\label{eq-first-part-G}
\frac{\mathrm{d}}{\mathrm{d}t}\frac{1}{4s} \dint_{\tilde{X}}|f_t|^2\,\mathrm{d}\tilde{m} = -\frac{(N-1)}{4s}\dint_{\tilde{X}}|f_t|^2\,\mathrm{d}\tilde{m}. 
\end{equation}
On the other hand, 
\begin{align}
\label{eq-second-part-G}
\frac{\mathrm{d}}{\mathrm{d}t}\frac{1}{4s} \dint_{\tilde{X}}|h_sf_t|^2\,\mathrm{d}\tilde{m} = & \lim_{\tau\to 0} \frac{1}{2s}\dint_{\tilde{X}} h_sf_t\left( \frac{h_sf_{t+\tau}-h_sf_t}{\tau} \right)\,\mathrm{d}\tilde{m}  \\
= &  \lim_{\tau\to 0} \frac{1}{2s}\dint_{\tilde{X}} h_{2s}f_t\left( \frac{f_{t+\tau}-f_t}{\tau}\right) \,\mathrm{d}\tilde{m} \nonumber \\
= & \lim_{\tau\to 0} \frac{1}{2s} \dint_{\tilde{X}} \frac{e^{-(N-1)\tau}h_{2s}f_t\circ F_{-\tau}-h_{2s}f_t}{\tau}f_t\,\mathrm{d}\tilde{m}\nonumber \\
= & \frac{-(N-1)}{2s} \dint_{\tilde{X}} |h_sf_t|^2\, \mathrm{d}\tilde{m} - \frac{1}{2s} \dint_{\tilde{X}} \left\langle \nabla h_{2s}f_t,\nabla u\right\rangle f_t \, \mathrm{d}\tilde{m} \nonumber  
\end{align}
where we used Lemma \ref{lem-derivative-along-flow}.

Therefore, putting together identities \eqref{eq-first-part-G} and \eqref{eq-second-part-G} we obtain that 
\begin{equation}
\label{eq-derivative-of-G}
\frac{\mathrm{d}}{\mathrm{d}t} G(t,s) = -(N-1)G(t,s) +\frac{(N-1)}{4s}\dint_{\tilde{X}}|h_sf_t|^2\, \mathrm{d}\tilde{m} + \frac{1}{2s}\dint_{\tilde{X}} \left\langle  \nabla h_{2s}f_t,\nabla u\right\rangle f_t \, \mathrm{d}\tilde{m}
\end{equation}

Now, let us note that the function $\tau \mapsto \dint_{\tilde{X}} \left\langle \nabla h_{s+\tau}f_t,\nabla u\right\rangle h_{s-\tau}f_t$ is of class $C^1$ on $[0,s]$ as a consequence of $\tau\mapsto h_{s\pm\tau}f_t$ being $C^1$. Then, the integral in the last summand of the previous identity can be written as   
\begin{equation}
\label{eq-derivative-of-G-2}
\dint_{\tilde{X}} \left\langle \nabla h_{2s}f_t,\nabla u\right\rangle f_t\,\mathrm{d}\tilde{m} = \dint_{\tilde{X}} \left\langle \nabla h_sf_t,\nabla u\right\rangle h_sf_t\,\mathrm{d}\tilde{m} + \dint_0^s\frac{\mathrm{d}}{\mathrm{d}\tau} \dint_{\tilde{X}} \left\langle \nabla h_{s+\tau}f_t,\nabla u\right\rangle h_{s-\tau}f_t\,\mathrm{d}\tilde{m}\,\mathrm{d}\tau. 
\end{equation}
In turn, the first summand of the right hand side of the previous expression can be computed as follows 
\begin{equation}
\label{eq-derivative-of-G-3}
\dint_{\tilde{X}} \left\langle \nabla h_sf_t,\nabla u\right\rangle h_sf_t\,\mathrm{d}\tilde{m} = \dint_{\tilde{X}} \left\langle \nabla \frac{|h_sf_t|^2}{2},\nabla u\right\rangle \,\mathrm{d}\tilde{m} = -\frac{(N-1)}{2}\dint_{\tilde{X}} |h_sf_t|^2\,\mathrm{d}\tilde{m}
\end{equation}
where we used that $\Delta u =N-1$. It follows from plugging the two previous computations \eqref{eq-derivative-of-G-2} and \eqref{eq-derivative-of-G-3} in formula \eqref{eq-derivative-of-G} that 
\begin{equation}
\label{eq-derivative-of-G-4}
\frac{\mathrm{d}}{\mathrm{d}t} G(t,s) = -(N-1)G(t,s) + \frac{1}{2s}\dint_0^s\frac{\mathrm{d}}{\mathrm{d}\tau} \dint_{\tilde{X}} \left\langle \nabla h_{s+\tau}f_t,\nabla u\right\rangle h_{s-\tau}f_t\,\mathrm{d}\tilde{m}\,\mathrm{d}\tau. 
\end{equation}


 Observe that 
\[
\frac{\mathrm{d}}{\mathrm{d}\tau} \dint_{\tilde{X}} \left\langle \nabla h_{s+\tau}f_t,\nabla u\right\rangle h_{s-\tau}f_t\,\mathrm{d}\tilde{m} = \dint_{\tilde{X}} \left\langle \nabla \Delta h_{s+\tau}f_t,\nabla u\right\rangle h_{s-\tau}f_t - \left\langle \nabla h_{s+\tau}f_t,\nabla u\right\rangle \Delta h_{s-\tau}f_t\,\mathrm{d}\tilde{m}.
\]
In the following computation we will assume that $h_{s+\tau}f_t\in \mathrm{Test}(\tilde{X})$. We can do so without loss of generality up to an easy approximation argument using the density of test functions in $W^{1,2}(\tilde{X})$ and the fact that $h_{s+\tau}f_t\in W^{1,2}(\tilde{X})$. Then, using Proposition \ref{prop-euler-equation-u}, we get that 
\begin{align*}
\frac{\mathrm{d}}{\mathrm{d}\tau} \dint_{\tilde{X}} \left\langle \nabla h_{s+\tau}f_t,\nabla u\right\rangle h_{s-\tau}f_t\,\mathrm{d}\tilde{m} = & \dint_{\tilde{X}} -2h_{s-\tau}f_t\Delta h_{s+\tau}f_t \, \mathrm{d}\tilde{m}\\
+ & \dint_{\tilde{X}} 2(N-1)\left\langle \nabla u, \nabla h_{s+\tau}f_t\right\rangle h_{s-\tau}f_t \, \mathrm{d}\tilde{m}\\
+ & \dint_{\tilde{X}} 2\left\langle \nabla u, \nabla \left\langle \nabla h_{s+\tau}f_t, \nabla u\right\rangle \right\rangle h_{s-\tau}f_t \, \mathrm{d}\tilde{m}.
\end{align*}

The last term can be expressed by an integration by parts, and using that $\Delta u = N-1$, as
\begin{align*}
\dint_{\tilde{X}} 2\left\langle \nabla u, \nabla \left\langle \nabla h_{s+\tau}f_t, \nabla u\right\rangle \right\rangle h_{s-\tau}f_t \, \mathrm{d}\tilde{m} = & -2(N-1) \dint_{\tilde{X}} h_{s-\tau}f_t\left\langle \nabla h_{s+\tau}f_t,\nabla u\right\rangle \, \mathrm{d}\tilde{m}\\
 & -2 \dint_{\tilde{X}}\left\langle \nabla h_{s+\tau}f_t,\nabla u\right\rangle \left\langle h_{s-\tau}f_t,\nabla u\right\rangle \, \mathrm{d}\tilde{m}.
\end{align*}
Hence we have that 
\begin{align}
\label{eq-reminder-term-derivative-of-G}
\frac{\mathrm{d}}{\mathrm{d}\tau} \dint_{\tilde{X}} \left\langle \nabla h_{s+\tau}f_t,\nabla u\right\rangle h_{s-\tau}f_t\,\mathrm{d}\tilde{m} = & -2 \dint_{\tilde{X}} h_{s-\tau}f_t \Delta h_{s+\tau}f_t \, \mathrm{d}\tilde{m}\\
& -2\dint_{\tilde{X}}\left\langle \nabla h_{s+\tau}f_t,\nabla u\right\rangle \left\langle \nabla h_{s-\tau}f_t,\nabla u\right\rangle \, \mathrm{d}\tilde{m}. \nonumber
\end{align}
Combining with \eqref{eq-derivative-of-G-4} yields
\begin{align*}
\frac{\mathrm{d}}{\mathrm{d}t} G(t,s) = & -(N-1)G(t,s) - \frac{1}{s} \dint_0^s\! \dint_{\tilde{X}} h_{s-\tau} f_t \Delta h_{s+\tau} f_t \,\mathrm{d}\tilde{m}\,\mathrm{d}\tau \nonumber \\
& -\frac{1}{s} \dint_0^s\! \dint_{\tilde{X}} \left\langle \nabla h_{s+\tau}f_t,\nabla u\right\rangle \left\langle \nabla h_{s-\tau}f_t,\nabla u\right\rangle \, \mathrm{d}\tilde{m} \, \mathrm{d}\tau. \nonumber
\end{align*}
\end{proof}

Our next lemma deals with the last summand of the previous identity. 

\begin{lem}
	\label{lem-last-summand-of-G}
	We have
	\begin{align}
	\label{eq-derivative-of-G-5}
\dint_{\tilde{X}}\left\langle \nabla h_{s+\tau}f_t,\nabla u\right\rangle \left\langle \nabla h_{s-\tau}f_t,\nabla u\right\rangle \, \mathrm{d}\tilde{m} = e^{-(N-1)t} \dint_{\tilde{X}}\left\langle \nabla h_{s+\tau}f,\nabla u\right\rangle \left\langle \nabla h_{s-\tau}f,\nabla u\right\rangle \, \mathrm{d}\tilde{m}.
	\end{align}
\end{lem}

\begin{proof}
To that end, let $\varepsilon>0$ and observe that 
\begin{align}
\label{eq-limS1S2together}
\frac{\left\langle \nabla h_{s+\tau}f_{t+\varepsilon},\nabla u\right\rangle\left\langle \nabla h_{s-\tau}f_{t+\varepsilon},\nabla u\right\rangle - \left\langle \nabla h_{s+\tau}f_{t},\nabla u\right\rangle\left\langle \nabla h_{s-\tau}f_{t},\nabla u\right\rangle}{\varepsilon} =S_1 + S_2,
\end{align}
where 
\[
S_1:= \left\langle \nabla h_{s-\tau}f_{t+\varepsilon},\nabla u\right\rangle \left\langle \frac{\nabla\left(h_{s+\tau}f_{t+\varepsilon}-h_{s+\tau}f_{t} \right)}{\varepsilon}, \nabla u\right\rangle
\]
and 
\[  
S_2:= \left\langle \nabla h_{s+\tau}f_{t},\nabla u\right\rangle \left\langle \frac{\nabla\left(h_{s-\tau}f_{t+\varepsilon}-h_{s+\tau}f_{t} \right)}{\varepsilon}, \nabla u\right\rangle. 
\]
Now we compute $\lim_{\varepsilon\to 0}S_1$ as follows. First observe that by \eqref{eq-LipHeatF_t}, 
\[
\lim_{\varepsilon\to 0}\left\langle \nabla h_{s-\tau}f_{t+\varepsilon},\nabla u\right\rangle = \left\langle \nabla h_{s-\tau}f_{t},\nabla u\right\rangle
\]
where the limit is intended in $L^2$. Therefore, 
\begin{equation}
\label{eq-limS1_1}
\lim_{\varepsilon\to 0}S_1= \lim_{\varepsilon\to 0}\dint_{\tilde{X}} \left\langle \nabla h_{s-\tau}f_t,\nabla u\right\rangle \left\langle \frac{\nabla (h_{s+\tau}f_{t+\varepsilon}-h_{s+\tau}f_t)}{\varepsilon},\nabla u\right\rangle\,\mathrm{d}\tilde{m}.
\end{equation}

Now, as before, we assume $h_{s_\tau}f_t\in \mathrm{Test}(\tilde{X})$. The following estimate holds true in the general case by an approximation argument. Observe that for every $f,g\in W^{1,2}(\tilde X)$ the following holds, by using $\Delta u=N-1$,  (cf. \cite[(4.35)]{Gig})
\begin{equation}\label{eq-intbyparts}
\dint_{\tilde{X}} f \left\langle \nabla g, \nabla u\right\rangle\, \mathrm{d}\tilde{m} = -(N-1) \dint_{\tilde{X}} fg\,\mathrm{d}\tilde m -\dint_{\tilde{X}} g \left\langle \nabla f, \nabla u\right\rangle\, \mathrm{d}\tilde{m}. 
\end{equation}
Therefore, we obtain from \eqref{eq-limS1_1} that, 
\begin{align*}
\lim_{\varepsilon\to 0}S_1 = & \lim_{\varepsilon\to 0} -(N-1)\dint_{\tilde{X}}\left\langle \nabla h_{s-\tau}f_t,\nabla u\right\rangle\left( \frac{h_{s+\tau}f_{t+\varepsilon}-h_{s+\tau}f_t}{\varepsilon}\right)\,\mathrm{d}\tilde{m}\\
& - \dint_{\tilde{X}} \left\langle \nabla \left\langle \nabla h_{s-\tau}f_t,\nabla u\right\rangle, \nabla u\right\rangle\left( \frac{h_{s+\tau}f_{t+\varepsilon}-h_{s+\tau}f_t}{\varepsilon}\right)   \, \mathrm{d}\tilde{m}.
\end{align*}
We claim the previous expression is equal to
\[
\dint_{\tilde{X}} -(N-1)\left\langle \nabla h_{s-\tau}f_t,\nabla u\right\rangle\left\langle \nabla h_{s+\tau}f_t,\nabla u\right\rangle - \left\langle \nabla \left\langle \nabla h_{s-\tau}f_t,\nabla u\right\rangle, \nabla u\right\rangle\left\langle \nabla h_{s+\tau}f_t,\nabla u\right\rangle  \, \mathrm{d}\tilde{m}
\]
Indeed, it follows from H\"older inequality and the fact that $(f_{t+\varepsilon}-f_t)/\varepsilon$ is $L^2$-weakly convergent as $\varepsilon\to 0$ that 
\[
\lim_{\varepsilon\to 0}  \left| \dint_{\tilde{X}}\left\langle \nabla h_{s-\tau}f_t,\nabla u\right\rangle\left( \frac{h_{s+\tau}f_{t+\varepsilon}-h_{s+\tau}f_t}{\varepsilon}\right)\,\mathrm{d}\tilde{m} - \dint_{\tilde{X}} \left\langle \nabla h_{s-\tau}f_t,\nabla u\right\rangle\left\langle \nabla h_{s+\tau}f_t,\nabla u\right\rangle\, \mathrm{d}\tilde{m} \right| = 0
\]
and by a similar reason it is also true that 
\[
  \left| \dint_{\tilde{X}} \left\langle \nabla \left\langle \nabla h_{s-\tau}f_t,\nabla u\right\rangle, \nabla u\right\rangle\left( \frac{h_{s+\tau}f_{t+\varepsilon}-h_{s+\tau}f_t}{\varepsilon}\right)   \, \mathrm{d}\tilde{m} - \dint_{\tilde{X}} \left\langle \nabla \left\langle \nabla h_{s-\tau}f_t,\nabla u\right\rangle, \nabla u\right\rangle\left\langle \nabla h_{s+\tau}f_t,\nabla u\right\rangle  \, \mathrm{d}\tilde{m} \right|
\]
goes to $0$ as $\varepsilon\to 0$. Whence, we have obtained that 
\begin{equation}
\label{eq-limS1}
\lim_{\varepsilon\to 0} S_1 = \dint_{\tilde{X}} -(N-1)\left\langle \nabla h_{s-\tau}f_t,\nabla u\right\rangle\left\langle \nabla h_{s+\tau}f_t,\nabla u\right\rangle - \left\langle \nabla \left\langle \nabla h_{s-\tau}f_t,\nabla u\right\rangle, \nabla u\right\rangle\left\langle \nabla h_{s+\tau}f_t,\nabla u\right\rangle  \, \mathrm{d}\tilde{m}.
\end{equation}
A completely analogous procedure yields that 
\begin{equation}
\label{eq-limS2}
\lim_{\varepsilon\to 0} S_2 = \dint_{\tilde{X}} -(N-1)\left\langle \nabla h_{s-\tau}f_t,\nabla u\right\rangle\left\langle \nabla h_{s+\tau}f_t,\nabla u\right\rangle - \left\langle \nabla \left\langle \nabla h_{s+\tau}f_t,\nabla u\right\rangle, \nabla u\right\rangle\left\langle \nabla h_{s-\tau}f_t,\nabla u\right\rangle  \, \mathrm{d}\tilde{m}.
\end{equation}

It follows by using \eqref{eq-intbyparts} again and using \eqref{eq-limS1} and \eqref{eq-limS2} when taking the limit when $\varepsilon\to 0$ in \eqref{eq-limS1S2together} that 
\[
\frac{\mathrm{d}}{\mathrm{d}t}\dint_{\tilde{X}}\left\langle \nabla h_{s+\tau}f_t,\nabla u\right\rangle \left\langle \nabla h_{s-\tau}f_t,\nabla u\right\rangle \, \mathrm{d}\tilde{m} = -(N-1) \dint_{\tilde{X}}\left\langle \nabla h_{s+\tau}f_t,\nabla u\right\rangle \left\langle \nabla h_{s-\tau}f_t,\nabla u\right\rangle \, \mathrm{d}\tilde{m},
\]
and therefore, 
\[
\dint_{\tilde{X}}\left\langle \nabla h_{s+\tau}f_t,\nabla u\right\rangle \left\langle \nabla h_{s-\tau}f_t,\nabla u\right\rangle \, \mathrm{d}\tilde{m} = e^{-(N-1)t} \dint_{\tilde{X}}\left\langle \nabla h_{s+\tau}f,\nabla u\right\rangle \left\langle \nabla h_{s-\tau}f,\nabla u\right\rangle \, \mathrm{d}\tilde{m}.
\]
This finishes the proof of the lemma.
\end{proof}

We are now ready to give the proof of Theorem \ref{thm-f_t-is-Sobolev}.

\begin{proof}
Lemma \ref{lem-last-summand-of-G} and the Bakry-\'Emery estimate (see \cite[Theorem 4]{EKS} ) allow us to estimate the last summand on the right hand side of \eqref{eq-derivative-of-G-5} as
\begin{align*}
\left| -\frac{1}{s}\dint_{\tilde{X}}\left\langle \nabla h_{s+\tau}f,\nabla u\right\rangle \left\langle \nabla h_{s-\tau}f,\nabla u\right\rangle \, \mathrm{d}\tilde{m} \right|  \leq & \frac{1}{s} \dint_0^s e^{-(N-1)t} \left\| \nabla h_{s-\tau}f\right\|_{L^2(\tilde{X})} \left\| \nabla h_{s+\tau}f\right\|_{L^2(\tilde{X})} \,\mathrm{d}\tau \\
 \leq & \frac{1}{s} \dint_0^s e^{-(N-1)t} e^{2(N-1)(s-\tau)}e^{2(N-1)(s+\tau)}\| \nabla f \|^2_{L^2(\tilde{X})} \,\mathrm{d}\tau \\
= & e^{(N-1)(-t+2s)} \left\| \nabla f\right\|^2_{L^2(\tilde{X})}.
\end{align*}

%

Therefore, using this estimate and \eqref{eq-derivative-of-G-4}, we have the differential inequality 
\begin{equation}
\label{eq-differential-inequality-G}
\frac{\mathrm{d}}{\mathrm{d}t} G(t,s) \geq -(N-1)G(t,s) - e^{(N-1)(-t+2s)}\|\nabla f\|^2_{L^2(\tilde{X})}
\end{equation}
where we have discarded the second term of the right hand side of \eqref{eq-derivative-of-G-4} since 
\[
- \frac{1}{s} \dint_0^s\! \dint_{\tilde{X}} h_{s-\tau} f_t \Delta h_{s+\tau} f_t \,\mathrm{d}\tilde{m}\,\mathrm{d}\tau = \frac{1}{s}\dint_0^s\dint_{\tilde{X}} |\nabla h_{s}f_t|^2 \,\mathrm{d}\tilde{m}\,\mathrm{d}\tau \geq 0.
\]
Hence, from \eqref{eq-differential-inequality-G} we have that
\[
\frac{\mathrm{d}}{\mathrm{d}t} e^{(N-1)t}G(t,s) \geq - e^{2(N-1)s}\|\nabla f\|^2_{L^2(\tilde{X})}, 
\]
and it follows from Gronwall's lemma that 
\[
e^{(N-1)t}G(t,s)\leq G(0,s)+ te^{2(N-1)s}\|\nabla f\|^2_{L^2(\tilde{X})}.
\]
Therefore, for every $t\leq 0$, we have obtained that 
\begin{equation}
\label{eq-G-uniform-bound}
G(t,s)\leq e^{-(N-1)t}G(0,s).
\end{equation}

Let us note that, in fact, $G(0,s)\leq \frac{1}{2}\dint_{\tilde{X}}|\nabla f|^2\,\mathrm{d}\tilde{m}$ for $s\leq 1$. Indeed, first note that
\[
G(0,s) = \frac{1}{4s} \dint_{\tilde{X}} f\left(f-h_sf \right) + h_sf\left( f-h_sf \right)\, \mathrm{d}\tilde{m}.
\]
Now, on one hand, 
\[
f-h_sf = -\dint_0^s \frac{\mathrm{d}}{\mathrm{d}\tau} \left(h_{\tau}f \right)\, \mathrm{d}\tau = \dint_0^s \Delta h_{\tau}f\, \mathrm{d}\tau
\]
from which we can conclude that 
\[
\frac{1}{4s} \dint_{\tilde{X}} f\left( f-h_sf\right) = \frac{1}{4s} \dint_0^s\dint_{\tilde{X}} f\Delta h_{\tau}f \, \mathrm{d}\tilde{m}\,\mathrm{d}\tau = -\frac{1}{4s} \dint_0^s \dint_{\tilde{X}} \left\langle \nabla f, \nabla h_{\tau}f\right\rangle  \, \mathrm{d}\tilde{m}\,\mathrm{d}\tau.
\]
Therefore, by the H\"older inequality and the Bakry-Emery estimate, we obtain that 
\[
\frac{1}{4s} \dint_{\tilde{X}} f\left( f-h_sf\right) \leq \frac{1}{4} e^{(N-1)}\dint_{\tilde{X}}|\nabla f|^2\,\mathrm{d}\tilde{m}.
\]
On the other hand, an analogous analysis yields the same bound for $\frac{1}{4s} \dint_{\tilde{X}}h_sf\left( f-h_sf \right)\, \mathrm{d}\tilde{m}.$

To conclude the proof we note that by the uniform bound \eqref{eq-G-uniform-bound}, and since $G(t,s)\uparrow \frac{1}{2}\dint_{\tilde{X}}|\nabla f_t|^2\,\mathrm{d}\tilde{m}$ as $s\downarrow 0$, then the energies of $f_t$ are uniformly bounded for $t\leq 0$. Then, passing to the limit as $s\downarrow 0$ in \eqref{eq-G-uniform-bound}, we have that 
\[
\mathcal{E}(f_t) \leq e^{-(N-1)t}\mathcal{E}(f)
\]
for all $t\leq 0$. In particular, $f_t\in W^{1,2}(\tilde{X})$ for all $t\leq 0$.

We can now pass to the limit $s\downarrow 0$ in \eqref{eq-derivative-of-G-final} to obtain
\begin{align*}
\frac{\mathrm{d}}{\mathrm{d}t} \mathcal E (t) = -(N-1)\mathcal E (t) +2 \mathcal E (t)   - \dint_{\tilde X}   \left\langle\nabla f_{t}, \nabla u\right\rangle^2\, \mathrm{d}\tilde{m}.
\end{align*}
This finishes our proof.
\qedhere
\end{proof}

In the following theorem we see how the Cheeger energy of $f_t$ behaves along each of the summands of $\left\langle\nabla f_t, \nabla f_t\right\rangle= \mathrm{Hess}(u)( \nabla f_t, \nabla f_t) + \left\langle \nabla f_t, \nabla u\right\rangle^2$.

\begin{thm}\label{thm:Energy}
Let $u:\tilde{X}\to \mathbb{R}$ be the function built in Section \ref{sec-busemann-function}. The following identities hold for any $f \in W^{1,2}(\tilde X)$:
\begin{align*}
\dint_{\tilde X} \mathrm{Hess}(u) (\nabla f_t, \nabla f_t )\, \mathrm{d}\tilde m &= e^{-(N+1)t} \dint_{\tilde X} \mathrm{Hess}(u) (\nabla f, \nabla f )\, \mathrm{d} \tilde m \\
\dint_{\tilde X} \left\langle \nabla  f_t, \nabla u\right\rangle^2\, \mathrm{d}\tilde m &= e^{-(N-1)t} \dint_{\tilde X} \left\langle \nabla f, \nabla u\right\rangle^2\, \mathrm{d} \tilde m.
\end{align*}
\end{thm}

\medskip

\begin{proof}
 We will first prove the second equality by studying its $t$-derivative as in the proof of Lemma \ref{lem-last-summand-of-G}.  We compute 
\begin{align*}
\frac{\langle\nabla f_{t+h}, \nabla u\rangle ^2 - \langle\nabla  f_{t}, \nabla u\rangle^2 }{h}  & =  
\frac{\langle\nabla f_{t+h}, \nabla u\rangle - \langle\nabla f_{t}, \nabla u\rangle} {h} (\langle\nabla  f_{t+h}, \nabla u\rangle + \langle\nabla f_{t}, \nabla u\rangle) \\
&= ( \langle\nabla  \frac{ ( f_{t+h} - f_{t}) } {h} , \nabla u  \rangle)(\langle\nabla f_{t+h}, \nabla u\rangle + \langle\nabla  f_{t}, \nabla u\rangle).
\end{align*}
Observe that 
\begin{eqnarray*}
\lim_{h\to 0}\dint_{\tilde X}\langle\nabla \frac{(f_{t+h}-f_t)}{h}, \nabla u \rangle\langle \nabla f_{t+h},\nabla u\rangle\, \mathrm{d}\tilde{m} &=& \lim_{h\to 0} -(N-1)\dint_{\tilde X}\left(\frac{f_{t+h}-f_t}{h} \right)\langle \nabla f_{t+h},\nabla u\rangle\, \mathrm{d}\tilde{m}\\
& & - \dint_{\tilde X} \left(\frac{f_{t+h}-f_t}{h} \right) \langle \nabla \langle \nabla f_{t+h},\nabla u\rangle, \nabla u\rangle\, \mathrm{d}\tilde{m}.
\end{eqnarray*}

Let us denote the first and second summands of the left hand side of the previous equation by $A_1$ and $A_2$ respectively. We claim that 
\[
A_1= -(N-1)\dint_{\tilde X}\langle \nabla f_t, \nabla u\rangle^2 \mathrm{d}\tilde{m}.
\]
 To prove this claim, notice that 

\begin{align*}
\dint_{\tilde X}\left( \frac{f_{t+h}-f_t}{h}\right)\langle \nabla f_{t+h},\nabla u\rangle - \langle \nabla f_t, \nabla u \rangle ^2\mathrm{d}\tilde{m} &= \dint_{\tilde X} \left( \frac{f_{t+h}-f_t}{h}\right)\left(\langle \nabla f_{t+h},\nabla u\rangle - \langle \nabla f_t, \nabla u \rangle\right)\mathrm{d}\tilde{m}\\
& + \dint_{\tilde X} \langle \nabla f_t, \nabla u\rangle\left( \left(\frac{f_{t+h}-f_t}{h}\right)-\langle \nabla f_t,\nabla u\rangle\right)\mathrm{d}\tilde{m}.
\end{align*}

H\"older's inequality implies
\[
\left| \dint_{\tilde X} \left(\frac{f_{t+h}-f_t}{h} \right)\left(\langle \nabla f_{t+h},\nabla u\rangle- \langle \nabla f_t,\nabla u\rangle  \right)\mathrm{d}\tilde{m} \right| \leq \left\| \frac{f_{t+h}-f_t}{h} \right\|_{L^2} \|\langle \nabla f_{t+h},\nabla u\rangle-\langle \nabla f_t,\nabla u\rangle  \|_{L^2}.
\] 
This last expression converges to $0$ as $h\to 0$, since $\| \frac{f_{t+h}-f_t}{h} \|_{L^2}$ is bounded because $\frac{f_{t+h}-f_t}{h}$ is weakly convergent in $L^2$ and 
\[
\|\langle \nabla f_{t+h},\nabla u\rangle-\langle \nabla f_t,\nabla u\rangle  \|_{L^2}\to 0.
\]
 Moreover, by \cite[(4.34)]{Gig}, 
 \[
 \dint_{\tilde X} \langle \nabla f_t, \nabla u\rangle\left( \left(\frac{f_{t+h}-f_t}{h}\right)-\langle \nabla f_t,\nabla u\rangle\right)\mathrm{d}\tilde{m}\to 0,
 \] as $h\to 0$, and therefore the claim is proved. 

A similar procedure to the computation of $A_1$ yields 
\[
 A_2= -\dint_{\tilde X} \langle \nabla f_t, \nabla u\rangle \langle \nabla \langle \nabla f_t,\nabla u\rangle, \nabla u\rangle\, \mathrm{d}\tilde{m}.
 \] 
Therefore
\begin{equation*}
\lim_{h\to 0} \dint_{\tilde X} \langle \nabla \left(\frac{f_{t+h}-f_t}{h} \right), \nabla u\rangle\langle \nabla f_{t+h},\nabla u\rangle\, \mathrm{d}\tilde{m} = -(N-1)\dint_{\tilde X} \langle \nabla f_t,\nabla u\rangle^2\mathrm{d}\tilde{m} - \dint_{\tilde X} \langle \nabla f_t ,\nabla u\rangle \langle \nabla \langle \nabla f_t, \nabla u\rangle, \nabla u\rangle\, \mathrm{d}\tilde{m}. 
\end{equation*}
Thus, combining our observations, and using \eqref{eq-intbyparts}, we obtain 
\begin{align*}
\lim_{h \to 0} \dint_{\tilde X}\frac{\langle\nabla f_{t+h}, \nabla u\rangle ^2 - \langle\nabla f_{t}, \nabla u\rangle^2} {h}\,\mathrm{d}\tilde{m} = &  \dint_{\tilde X} (\langle\nabla \langle \nabla f_{t}, \nabla u\rangle, \nabla u\rangle) 2 (\langle\nabla  f_{t}, \nabla u\rangle)\, \mathrm{d}\tilde{m}\\
& -2(N-1)\dint_{\tilde X} \langle\nabla f_{t}, \nabla u\rangle^2 \mathrm{d}\tilde{m}\\
= & -(N-1)\dint_{\tilde X} \langle\nabla f_{t}, \nabla u\rangle^2 \mathrm{d}\tilde{m}.
\end{align*}

%

In conclusion, we have found that 
\begin{align} \label{eq-derivative-of-u-comp}
\frac{d}{dt} \dint_{\tilde X} \langle\nabla f_{t}, \nabla u\rangle^2\, \mathrm{d}\tilde{m} = -(N-1)\dint_{\tilde X}   \langle\nabla f_{t}, \nabla u\rangle^2\, \mathrm{d}\tilde{m} .
\end{align} Hence, 
$$
\dint_{\tilde X} \langle\nabla f_{t}, \nabla u\rangle^2\, \mathrm{d}\tilde{m} =e^{ -(N-1)t } \dint_{\tilde X} \langle\nabla f , \nabla u\rangle^2\, \mathrm{d}\tilde{m}.
$$

Now we will obtain the first equality. Observe that $\left\langle\nabla f, \nabla f\right\rangle= \mathrm{Hess}[u]( \nabla f, \nabla f) + \left\langle \nabla f, \nabla u\right\rangle^2$ implies that
\[
\frac{\mathrm{d}}{\mathrm{d}t} \mathcal E (t) = \frac{\mathrm{d}}{\mathrm{d}t} \frac{1}{2}\int_{\tilde X} \mathrm{Hess}(u)( \nabla f_t, \nabla f_t)\, \mathrm{d}\tilde{m} + \frac{\mathrm{d}}{\mathrm{d}t} \frac{1}{2}\int_{\tilde X} \left\langle \nabla f_t, \nabla u\right\rangle^2\, \mathrm{d}\tilde{m}.
\] 
By \eqref{eq-derivative-of-u-comp}, this becomes
$$
\frac{\mathrm{d}}{\mathrm{d}t} \mathcal E (t) = \frac{\mathrm{d}}{\mathrm{d}t} \frac{1}{2}\dint_{\tilde X}  \mathrm{Hess}(u)( \nabla f_t, \nabla f_t)\, \mathrm{d}\tilde{m}  -\frac{(N-1)}{2}\dint_{\tilde X}   \left\langle\nabla f_{t}, \nabla u\right\rangle^2\, \mathrm{d}\tilde{m}.
$$
From Theorem \ref{thm-f_t-is-Sobolev}, for $t\leq 0$, 
$$
\frac{\mathrm{d}}{\mathrm{d}t} \mathcal E (t)=  \dint_{\tilde X} \! \mathrm{Hess}(u)(\nabla f_t, \nabla f_t)\, \mathrm{d}\tilde{m} -  \frac{(N-1)}{2}\dint_{\tilde X}   \left\langle\nabla f_t, \nabla f_t\right\rangle\, \mathrm{d}\tilde{m}.
$$
Using both expressions for $\frac{\mathrm{d}}{\mathrm{d}t} \mathcal E (t)$ and solving for $\frac{\mathrm{d}}{\mathrm{d}t} \dint_{\tilde X}  \mathrm{Hess}(u)( \nabla f_t, \nabla f_t)\, \mathrm{d}\tilde{m}$, we get
$$ 
\frac{\mathrm{d}}{\mathrm{d}t} \int_{\tilde X} \mathrm{Hess}(u)( \nabla f_t, \nabla f_t)\, \mathrm{d}\tilde{m}= -(N-3)\int_{\tilde X} \mathrm{Hess}(u)( \nabla f_t, \nabla f_t)\, \mathrm{d}\tilde{m}.
$$
We conclude that for $t\leq 0$, 
$$ \dint_{\tilde X} \! \mathrm{Hess}(u)( \nabla f_t, \nabla f_t)\, \mathrm{d}\tilde{m}= e^{-(N-3)t} \dint_{\tilde X} \! \mathrm{Hess}(u)( \nabla f, \nabla f)\, \mathrm{d}\tilde{m}.$$
Now we reverse the flow, i.e. use the equation $f_t\circ F_{-t}=f$,  to see that the above equation holds for all $t$.
\end{proof}

\begin{rmk}\label{rmk:Energy}
As ${F_t }_\sharp \tilde m = e^{-(N-1)t} \tilde m$, we can rewrite the equalities in the previous theorem in the following way:
\[
\dint_{\tilde X} \! \mathrm{Hess}(u) (\nabla (f\circ F_t), \nabla (f \circ F_t) )\, \mathrm{d} \tilde m= e^{2t} \dint_{\tilde X} \! \mathrm{Hess}(u) (\nabla f, \nabla f )\, \mathrm{d} {F_t }_\sharp \tilde m
\]
and 
\[
\dint_{\tilde X}  \left\langle \nabla ( f\circ F_t), \nabla u\right\rangle^2\, \mathrm{d} \tilde m = \dint_{\tilde X}  \left\langle\nabla f, \nabla u\right\rangle^2\, \mathrm{d} {F_t }_\sharp \tilde m.
\]
\end{rmk}


\subsection{Localization of the Cheeger energy along the flow}\label{sec-localization}

Theorem \ref{thm:Energy} provides the behavior of $\left\langle\nabla (f \circ F_t), \nabla (f \circ F_t)\right\rangle $ in an integral form, i.e., at the level of the Cheeger energy. In this subsection we \textit{localize} that result, that is, we obtain a pointwise expression for $\left\langle\nabla (f \circ F_t), \nabla (f \circ F_t)\right\rangle $.

\begin{thm}\label{thm-g}
Let $u:\tilde{X}\to\mathbb{R}$ be the function constructed in Section \ref{sec-busemann-function}, $F:(-\infty,\infty)\times\tilde{X}\to\tilde{X}$ our Regular Lagrangian Flow. Then for every $f \in W^{1,2}(\tilde X)$
 the following identity holds
\[
\left\langle\nabla (f \circ F_t), \nabla (f \circ F_t)\right\rangle =  e^{2t} \mathrm{Hess}(u)(\nabla f, \nabla f) \circ F_t +  \left\langle\nabla f, \nabla u\right\rangle^2 \circ F_t.
\]
\end{thm}

The proof of this theorem requires the following lemma.

\begin{lem}\label{lem-quad}
Let  $f,g \in W^{1,2}(\tilde X,\tilde{d}, {F_t }_\sharp \tilde m)$ 
then 
\begin{equation}
\label{eq:polar}
\begin{aligned}
\int_{\tilde X} \left\langle \nabla (f\circ F_t), \nabla (g\circ F_t) \right\rangle d \tilde m   =  e^{2t}  \int_{\tilde X} \left\langle \nabla f, \nabla g \right\rangle \,d {F_t }_\sharp \tilde m 
 + (1- e^{2t}) \int_{\tilde{X}} \left\langle\nabla f, \nabla u\right\rangle \left\langle \nabla g, \nabla u\right\rangle   d {F_t }_\sharp \tilde m
\end{aligned}
\end{equation}
\end{lem}

\begin{proof}
By  Corollary (\ref{cor-hessian-identity}) and Remark \ref{rmk:Energy} we may write 
\begin{equation*}\label{eq-chE}
\dint_{\tilde X}   \left\langle\nabla (f\circ F_t), \nabla (f \circ F_t)\right\rangle \, \mathrm{d} \tilde m = e^{2t} \dint_{\tilde X} \left\langle \nabla f, \nabla f\right\rangle  \mathrm{d} {F_t }_\sharp \tilde{m}  +  
\left(1- e^{2t}\right) \dint_{\tilde X} \left\langle\nabla f, \nabla u\right\rangle^2 \mathrm{d} {F_t }_\sharp \tilde{m}.
\end{equation*}
Now, by the definition of $\left\langle \nabla \cdot, \nabla \cdot\right\rangle$, 
\begin{equation*}
\left\langle \nabla(f\circ F_t), \nabla (g\circ F_t)\right\rangle   =  \lim_{\varepsilon > 0} \frac{ |\nabla( g\circ F_t + \varepsilon f \circ F_t  )|^2 - |\nabla (g\circ F_t)|^2 }{2\varepsilon}.
\end{equation*}
Putting together both equations we find,
\begin{equation*}
\begin{aligned}
 \dint_{\tilde X} \left\langle \nabla(f\circ F_t),\nabla (g\circ F_t)\right\rangle \mathrm{d} \tilde{m} & =  \lim_{\varepsilon > 0} \frac 1{2\varepsilon} \left[ e^{2t}    \dint_{\tilde X} \left\langle \nabla  (g + \varepsilon f )    , \nabla  (g + \varepsilon f )  \right\rangle   - \left\langle \nabla g, \nabla g\right\rangle\, \mathrm{d} {F_t}_\sharp \tilde m  \right. \\
& \left. + (1- e^{2t}   ) \dint_{\tilde X} \left\langle\nabla (g + \varepsilon f ), \nabla u\right\rangle^2 - \left\langle\nabla g, \nabla u\right\rangle^2 \mathrm{d} {F_t }_\sharp \tilde m \right].
 \end{aligned}
\end{equation*}
The result follows. 
\end{proof}

We can now provide the proof of Theorem \ref{thm-g}.

\begin{proof}[Proof of Theorem \ref{thm-g}] By a simple approximation argument using the density of $\mathrm{Test}_{bs}(\tilde{X})$ functions in $W^{1,2}(\tilde{X})$, it suffices to show for
$f, g\in \mathrm{Test}_{bs}(\tilde{X})$. 
Since $\mathrm{Test}_{bs}(\tilde{X})$ is an algebra, $f^2, fg\in \mathrm{Test}_{bs}(\tilde{X})$. Thus,
\begin{equation}
\label{eq28}
\int_{\tilde{X}} g |\nabla f|^2 \, \mathrm{d} {F_t}_\sharp \tilde{m} = \int_{\tilde{X}} \left\langle \nabla (fg), \nabla f \right\rangle 
 - \left\langle \nabla g, \nabla (\tfrac{f^2}{2})\right\rangle \mathrm{d} {F_t }_\sharp \tilde m 
\end{equation}
Now, applying equation \eqref{eq:polar} from the previous lemma to each of the summands on the right hand side of the previous identity we get that 
\begin{equation*}
\begin{aligned}
\int_{\tilde{X}} g |\nabla f|^2 \, \mathrm{d} {F_t}_\sharp \tilde{m} = & e^{-2t}\left( \int_{\tilde X} \left\langle \nabla  ((f\circ F_t)(g\circ F_t) ), \nabla f\circ F_t \right\rangle \mathrm{d} \tilde m - \int_{\tilde X} \left\langle  \nabla  g\circ F_t, \nabla \tfrac{(f\circ F_t)^2}{2}\right\rangle \mathrm{d} \tilde m\right)\\
&  - (e^{-2t}-1)\dint_{\tilde X} \left\langle  \nabla fg, \nabla u\right\rangle  \left\langle \nabla f, \nabla u\right\rangle - \left\langle\nabla g, \nabla u\right\rangle\left\langle \nabla (\tfrac{ f^2}{2}), \nabla u\right\rangle  \mathrm{d} {F_t}_\sharp \tilde m.
\end{aligned}
\end{equation*}

We now use again the Leibniz rule for $\nabla(fg)$ and $\nabla f^2$ in the following computation. We also use the following: Since $f\in W^{1,2}(\tilde{X})$, then by Theorem \ref{thm-f_t-is-Sobolev} $f\circ F_t\in W^{1,2}(\tilde{X})$.  Now equation \eqref{eq:polar} can be applied for $f\circ F_t$ and $g\circ F_t$. Therefore, the previous identity is written as
\begin{equation*}
\int_{\tilde{X}} g |\nabla f|^2 \, \mathrm{d} {F_t}_\sharp \tilde{m} =  e^{-2t}\int_{\tilde X} (g\circ F_t)| \nabla(f\circ F_t)|^2\, \mathrm{d} \tilde m - (e^{-2t}-1)\dint_{\tilde X} g\left\langle \nabla f,\nabla u\right\rangle\left\langle f, \nabla u\right\rangle \mathrm{d} {F_t}_\sharp \tilde m.
\end{equation*}

Rearranging terms and using Corollary \ref{cor-hessian-identity}, we obtain 
\begin{align*}
\int_{\tilde X} (g\circ F_t)| \nabla(f\circ F_t)|^2\, \mathrm{d} \tilde m & = e^{2t} \int_{\tilde X}   g\, \mathrm{Hess}(u)(\nabla f, \nabla f ) \, \mathrm{d} {F_t }_\sharp \tilde m +  \int_{\tilde{X}}  g \left\langle \nabla  f, \nabla u\right\rangle^2 \, \mathrm{d} {F_t }_\sharp \tilde m \\
& = e^{2t} \int_{\tilde X}  (g\circ F_t)\, \mathrm{Hess}(u)(\nabla f, \nabla f ) \circ F_t\, \mathrm{d} \tilde m +  \int_{\tilde{X}} (g\circ F_t) \left\langle \nabla  f, \nabla u\right\rangle^2 \circ F_t \, \mathrm{d}  \tilde m.
\end{align*}
Finally, as $g$ is arbitrary we conclude the validity of the formula. 
\end{proof}

\section{The quotient metric measure space $(X',d',m')$}

\subsection{Continuous representative of $F$}

Using our knowledge of the value of $|\nabla f_t|$ we can now improve the regularity of the flow and show that for fixed $t$, the function $F_t$ is Lipschitz.

\begin{thm}\label{thm:rapprcont}
The map $F: (-\infty,\infty) \times  \tilde X \to \tilde X$ admits a continuous representative with respect to the measure $\mathcal L^1\times \tilde m$. Still denoting such representative  by $F$, we have:
\begin{itemize}
\item[i)] The semigroup property holds, i.e., for every $t,s\in \mathbb R$ and $x \in \tilde X$ we have $F_t(F_s(x)) = F_{t+s}(x)$.
Moreover,
\begin{equation*}
\tilde d(F_t(x),F_{t+s}(x)) =  |s|.
\end{equation*}
\item[ii)] For every $t\in\mathbb R$, $F_t$ is a bi-Lipschitz map with $\Lip(F_t) \leq \max \{e^{t},1\}$.
\item[iii)] Given a curve $\gamma:[0,1] \to \tilde X$ let $\bar \gamma:=F_t\circ\gamma$. Then one of the curves is absolutely continuous if and only if the other is and their metric speeds are related by the following inequality
\begin{equation}
\label{eq:locspeeds}
\min\{1,e^{t}\} |\dot\gamma_s|   \leq  |\dot{\bar\gamma}_s| \leq \max\{1,e^{t}\} |\dot\gamma_s|\qquad\text{ for a.e.\ }s\in[0,1].
\end{equation}
\end{itemize}
\end{thm}

\begin{proof} 
Statements in $i)$ follows from Proposition \ref{prop-F_tPushMeas} and Lemma \ref{lem-distance-flow-representatives}. 
Now for each $t \in \mathbb R$ we will first obtain a $\max\{1,e^{t}\}$-Lipschitz representative of $F_t$. 
By Theorem (\ref{thm-g}) we know that for $f \in  W^{1,2}(\tilde X, \tilde d, \tilde m)$,
\[
\left\langle \nabla (f \circ F_t), \nabla (f \circ F_t)\right\rangle =  e^{2t} \mathrm{Hess}(u)(\nabla f, \nabla f) \circ F_t +  \left\langle \nabla f, \nabla u\right\rangle^2 \circ F_t.
\]
Therefore, 
\begin{align*}
\left\langle \nabla (f \circ F_t), \nabla (f \circ F_t)\right\rangle  & \leq \max\{e^{2t},1\} \left( \mathrm{Hess}(u)(\nabla f, \nabla f) \circ F_t +  \left\langle \nabla f, \nabla u\right\rangle^2 \circ F_t\right)\\
& =  \max\{e^{2t},1\} \left\langle \nabla f , \nabla f \right\rangle \circ F_t .
\end{align*}
Thus,  $|\nabla (f \circ F_t)|  \leq \max\{1,e^{t}\}$. Because $\tilde{X}$ has the Sobolev to Lipschitz property, $f \circ F_t$ has a $\max\{1,e^{t}\}$-Lipschitz representative.

 As in \cite[Lemma 4.19]{Gig},  the functions $f_{n,k}=\max\{0,\min\{ \tilde d(\cdot,x_n),k-\tilde d(\cdot,x_n)\}\}$,  with $\{ x_n\}$ dense in $\tilde X$ are $1$-Lipschitz  with bounded support and thus belong to $W^{1,2}(\tilde X, \tilde d, \tilde m)$ with $|\nabla f_{k,n}|\leq 1$ $\tilde m$-a.e. Let $\mathcal D=\{ f_{n, k} \} \subset W^{1,2}(\tilde X, \tilde d, \tilde m)$. Then $\mathcal D$ is a countable set of $1$-Lipschitz functions with compact support such that $\mathcal D$ is dense in the space of $1$-Lipschitz functions with compact support with respect to uniform convergence.
Thus, for all $y_0,y_1 \in \tilde X$,
\begin{equation}\label{eq-Dequality}
\tilde d(y_0, y_1)= \sup_{f \in \mathcal D} |f(y_0) - f(y_1)|.
\end{equation}

Since $\mathcal D$ is countable, then there is an $\tilde m$-negligible Borel set $\mathcal N'$ such that for every $f \in \mathcal D$, the restrictions $f \circ F_t: \tilde X \setminus \mathcal N' \to \mathbb R$ are $\max\{1,e^{t}\}$-Lipschitz. 
Therefore, by (\ref{eq-Dequality}) for $x_0,x_1 \in F_t^{-1}  ( \tilde X \setminus \mathcal N')$ we have
\[
\tilde d(F_t (x_0),F_t (x_1)) =\sup_{f \in \mathcal D} |f(F_t (x_0))- f(F_t (x_1))| \leq  \max\{1,e^{t}\} \tilde d(x_0,x_1).
\]
Hence, for each $(t, x), (s,y) \in \mathbb R \times \tilde X$ we obtain 
\begin{equation}
\label{eq-F_t-locally-Lipschitz}
\tilde d(F_t (x),F_s (y)) \leq \tilde d(F_t (x),F_t (y)) + \tilde d(F_t (y),F_s (y)) \leq  \max\{1,e^{t}\} \tilde d(x,y) + |s-t|.
\end{equation}
This proves that $F$ admits a continuous representative. 
Then the statements in $ii)$ follows. 

For $iii)$, let us assume that $\gamma$ is absolutely continuous. Then 
\[
\tilde d(\bar{\gamma}_{h}, \bar\gamma_{s}) = \tilde d(F_t(\gamma_{h}), F_t(\gamma_{s}) )\leq \max\{1, e^{t} \} \tilde d(\gamma_{h},\gamma_{s}) \leq \max\{1, e^{t} \}\int_h^s |\dot{\gamma}_r|\, \mathrm{d}r.
\]
Therefore, $|\dot{\bar\gamma}_s| \leq \max\{1,e^{t}\} |\dot\gamma_s| $ for a.e.-$s\in[0,1].$ The other inequality is proven in a similar way. 
\end{proof}

We continue this section by defining a quotient metric measure space $(X',d',m')$ induced by the flow $F$. We will show  that it is an infinitesimally Hilbertian space, and that it satisfies the Sobolev to Lipschitz property.  We now provide the definition of $X'$.

\begin{defn}\label{def-X'}
Let $X'= u^{-1}(0)$ and  define $d':X'\times X'\to \mathbb{R}$ by 
\[
d'(z,y)= \inf\{  L(\gamma) | \gamma \in AC([0,1], \tilde X),\, u \circ \gamma= 0,\, \gamma_0=z,\,\gamma_1=y \}.
\]
Here $L(\gamma)= \int_0^1 |\dot\gamma_r|dr$.
\end{defn}

\begin{lem}\label{lem-d'}
Let $X'$  be as in Definition \ref{def-X'}, then $d'$ is a well defined function and $(X', d')$ is a metric space. The inclusion map $\iota: (X',d') \to (\tilde{X}, \tilde{d})$ is 1-Lipschitz.
\end{lem}

\begin{proof}
First we will show that the set 
\[\{ \gamma \in AC([0,1], \tilde X),\, u \circ \gamma= 0,\, \gamma_0=z,\,\gamma_1=y \} \]
 is nonempty for any $z,y \in X'$.  As  $\tilde X$ is a geodesic space there exists an absolutely continuous $\gamma:[0,1] \to \tilde X$  such that $\gamma_0=z$ and $\gamma_1=y$. By  Theorem \ref{thm:rapprcont},  the curve $t \mapsto F_{-u(\gamma_t)} (\gamma_t)$ is contained in $u^{-1}(0)$. We only have to prove that it is absolutely continuous. To that end, let $M=\max\{\Lip (F_{-u(\gamma_s)})\, | \, 0 \leq s \leq 1\}$. This maximum $M$ is achieved because $u$, $F$, and $\gamma$ are continuous.  Using the triangle inequality, together with the fact that $F_{-u(\gamma_s)}$ is a Lipschitz function for all $s$ and that $u$  is a $1$-Lipschitz function, gives, for all $0\leq s \leq t \leq 1$, 
\begin{align}
\tilde d(F_{-u(\gamma_s)} (\gamma_s),F_{-u(\gamma_t)} (\gamma_t)) \leq & \,
\tilde d(F_{-u(\gamma_s)} (\gamma_s),F_{-u(\gamma_s)} (\gamma_t))  + \tilde d(F_{-u(\gamma_s)} (\gamma_t),F_{-u(\gamma_t)} (\gamma_t)) \\
\leq   & \, \Lip (F_{-u(\gamma_s)}) \tilde d(\gamma_s,\gamma_t) + |u(\gamma_t)-u(\gamma_s)| \nonumber \\
\leq   & \, (\Lip (F_{-u(\gamma_s)}) + 1) \tilde d(\gamma_s,\gamma_t) \nonumber \\
\leq   & \,(M+1) \int_s^t |\dot{\gamma}_r|\, \mathrm{d}r . \nonumber
\end{align}
Hence, $F_{-u(\gamma_t)} (\gamma_t)$ is absolutely continuous in $(\tilde X, \tilde d)$ and $d'$ is well defined.  

If $z,y \in u^{-1}(0)$ then, 
\begin{align}
\tilde d(z,y) \leq  &  \inf\{  L(\gamma) | \gamma \in AC([0,1], \tilde X),\, u \circ \gamma= 0,\, \gamma_0=z,\,\gamma_1=y \} \\
  = & d'(z,y). \nonumber
\end{align}
This shows that $\iota$ is a 1-Lipschitz map and that $d'$ is positive definite. Symmetry and the triangle inequality follow from the definition of $d'$. 
\end{proof}


\subsection{Metric speed of curves in the quotient space}

Let $\pi:\tilde{X}\to X'$ be given by $\pi(x)= F_{-u(x)}(x)$. By Lemma \ref{lem-distance-flow-representatives}, $\pi$ is well defined and from now on we call it the projection map. The aim of this subsection is to study $\pi$ and its effect  on the metric speed of curves in $\tilde{X}$.  The main results of this subsection are collected in the following proposition, which will be used in the next subsection  to relate a subspace of $W^{1,2}(\tilde X, \tilde d, \tilde m)$  with $W^{1,2}(X',d',m')$. 

\begin{prop}\label{prop-mst0}
Let $\boldsymbol{\uppi}$ be a test plan on $\tilde X$. Then, for  $\boldsymbol{\uppi}$-a.e. $\gamma$, the curve $\tilde \gamma = \pi \circ \gamma$ in $(X',d')$ is absolutely continuous and for a.e. $t\in [0,1]$,
\begin{enumerate}
\item $|\dot{\tilde\gamma}_t| \leq e^{-u(\gamma_t)}\ |\dot\gamma_t|$. 
\item The projection map $\tilde \pi: \tilde X \to X'$ is locally Lipschitz, i.e. for all $x_0 \in \tilde X$ and all $x,y \in B(x_0,r)$, 
\begin{equation*}
d'(\pi(x),\pi(y)) \leq e^{-u(x_0)+3r} \tilde d(x,y).
\end{equation*} 
\end{enumerate}
\end{prop}

To prove $(1)$ we will follow the strategy developed by De Philippis-Gigli  (Section 3.6.2 \cite{DePG}) and define a ``truncated'' and reparametrized flow $\hat{F}$ with the property that for large $s$  the maps $\hat{F}_s$ approximate the projection map $\pi:u^{-1}([-R,R])\to u^{-1}(0)$, for $0<R<1$.  

Let $0<\overline{R}<R<1$ and $\psi\in C^{\infty}(\mathbb{R})$ with support in $(-R, R)$ such that $\psi(z)=-\frac{1}{2}z^2$ for all $z\in [-\overline{R},\overline{R}]$. Define the function $\hat{u}=\psi\circ u:\tilde{X}\to \mathbb{R}$ and consider a reparametrization function $\mathrm{rep}_s(r)$ defined by the property that $\partial_s \mathrm{rep}_s(r)=\psi' (\mathrm{rep}_s(r)+r)$ and $\mathrm{rep}_0(r)=0$. We now define the  flow $\hat{F}:\mathbb{R}\times \tilde{X}\to \tilde{X}$ by $\hat{F}_s(x):= F_{\mathrm{rep}_s(u(x))}(x)$ and note that $\hat{F}_s(x)=F_{(e^{-s}-1)u(x)}(x)$ on $u^{-1}([-\overline{R},\overline{R}])$. It follows from these definitions that $\hat{F}$ is the Regular Lagrangian Flow associated to $\hat{u}$. Moreover, the following formulae hold for all $x\in u^{-1}([-\overline{R}, \overline{R}])$,
\begin{eqnarray}
\hat{u}& = & -\frac{1}{2}u^2,\\
\nabla \hat{u}& =& -u\nabla u,\\
\Delta \hat{u} &=& -u(N-1)-1,\\
\mathrm{Hess}(\hat{u})&=& -u\mathrm{Id} + (u-1)(\nabla u\otimes \nabla u).  
\end{eqnarray}

The previous formulae, imply $\hat{u}\in \mathrm{Test}(\tilde{X})$, in particular it has bounded gradient, Laplacian and Hessian. When $s\to\infty$ then $\mathrm{rep}_s(u(x))\to -u(x)$ for every $x\in u^{-1}([-\overline{R},\overline{R}])$, that is $\hat{F}_s$ converges uniformly to $\pi:=F_{-u(\cdot)}(\cdot)$, the projection map. We observe that $\hat{F}_s$ is the identity on $\tilde{X}\setminus u^{-1}([-R,R])$ and it sends $u^{-1}([-\overline{R},\overline{R}])$ to itself. 

In the following, for each $s\in \mathbb{R}$, we only concern ourselves with $\hat{F}_s|_{u^{-1}([-\overline{R},\overline{R}])}$, because this will be sufficient for our purposes. Observe that \cite[Lemma 3.30]{DePG}, \cite[Proposition 3.31]{DePG} hold in this setting because, as we will now see, $\hat{F}_s$ is of bounded deformation (i.e. Lipschitz with bounded compression) for any $s\in \mathbb{R}$.  We begin by showing that $\hat{F}_s$ is Lipschitz on $u^{-1}([-\overline{R},\overline{R}])$:
\begin{eqnarray*}
\tilde d(\hat{F}_s(x), \hat{F}_s(y))& = & \tilde d(F_{(e^{-s}-1)u(x)}(x), F_{(e^{-s}-1)u(y)}(y))\\
& \leq & \tilde d(F_{(e^{-s}-1)u(x)}(x), F_{(e^{-s}-1)u(x)}(y)) + \tilde d(F_{(e^{-s}-1)u(x)}(y), F_{(e^{-s}-1)u(y)}(y))\\
&\leq & \max\{1, e^{(e^{-s}-1)(u(x))} \}\tilde d(x,y) + |(e^{-s}-1)| \, |u(x) - u(y)|\\
&\leq & \max\{1, e^{(e^{-s}-1)(u(x))} \}\tilde d(x,y) + |(e^{-s}-1)| \tilde d(x,y)\\
&\leq & \left( \max\{1, e^{|(e^{-s}-1)|\overline{R})} \} + |(e^{-s}-1)|\right) \tilde d(x,y).
\end{eqnarray*}

This proves $\hat{F}_s|_{u^{-1}([-\overline{R},\overline{R}])}$ is Lipschitz with Lipschitz constant 
\[
\max\{1, e^{|(e^{-s}-1)|\overline{R})} \} + |(e^{-s}-1)|
\]
for any $s\in \mathbb{R}$. 

Let us proceed by showing that $\hat{F}_s|_{u^{-1}([-\overline{R},\overline{R}])}$ is of bounded compression, as: 
\[
(\hat{F}_s|_{u^{-1}(-\overline{R},\overline{R}])})_{\sharp}\tilde{m}= (F_{(e^{-s}-1)u(\cdot)})_{\sharp}\tilde{m} = e^{-(N-1)(e^{-s}-1)u(\cdot)}\tilde{m} \leq e^{(N-1)|(e^{-s}-1)|\overline{R}}\tilde{m}.
\]

We now recall  \cite[Lemma 3.30]{DePG} and \cite[Proposition 3.31]{DePG}.  

\begin{lem}[De Philippis-Gigli, Lemma 3.30]
\label{lem:Dephillipis_Gigli_3.30}
Let $\varphi\in W^{1,2}(\tilde{X})$. Then the map $s\mapsto \varphi\circ \hat{F}_s\in L^2(\tilde X)$ is $C^1$ and its derivative is given by
\begin{equation}
\frac{d}{ds} \varphi\circ\hat{F}_s =\left\langle\nabla \varphi, \nabla\hat{u}\right\rangle\circ \hat{F}_s.
\end{equation}
If $\varphi$ is further assumed to be in $\mathrm{Test}(\tilde{X})$, then the map $s\mapsto  \mathrm{d}(\varphi \circ\hat{F}_s)\in L^2(T\tilde{X})$ is also $C^1$ and its derivative is given by
\begin{equation}
\frac{d}{ds}\left(\mathrm{d}(\varphi\circ\hat{F}_s) \right)=\mathrm{d}(\left\langle\nabla\varphi, \nabla\hat{u}\right\rangle\circ\hat{F}_s).
\end{equation}
\end{lem}

\begin{prop}[De Philippis-Gigli, Proposition 3.31]
\label{DePhilippis_Gigli_Prop_3.31}
Let $v\in L^2(T\tilde{X})$ and set $v_s:=\mathrm{d}\hat{F}_s(v)$. Then the map $s\mapsto \frac{1}{2}|v_s|^{2}\circ\hat{F}\in L^1(\tilde{X})$ is $C^1$ on $\mathbb{R}$ and its derivative is given by the formula
\begin{equation}
\label{derivative_norm_squared}
\frac{d}{ds}\frac{1}{2}|v_s|^{2}\circ\hat{F}=\mathrm{Hess}[\hat{u}](v_s,v_s)\circ\hat{F}_s,
\end{equation}
the incremental ratios being convergent both in $L^1(\tilde{X})$ and $\tilde{m}$-a.e. If $v$ is also bounded, then the curve $s\mapsto \frac{1}{2}|v_s|^{2}\circ\hat{F}$ is $C^1$ also when seen with values in $L^2(\tilde{X})$, and in this case the incremental ratios in \eqref{derivative_norm_squared} also converge in $L^2(\tilde{X})$ to the right hand side.
\end{prop}

We will use the previous results to prove the following monotonicity formula. The proof is similar to that of \cite[Corollary 3.32]{DePG}.

\begin{cor}
\label{cor-metric-speed-1}
Let $v\in L^2(T\tilde{X})$ be concentrated on $B:=u^{-1}([-\overline{R}, \overline{R}])$ and set $v_s:=\mathrm{d}\hat{F}_s(v)$. Then for every $s_1, s_2\in \mathbb{R}$ such that $s_1\leq s_2$,
\begin{equation}
\label{eq:endecr}
\left(e^{-2u}|v_{s_2}|^2\right)\circ\hat{F}_{s_2}\leq \left(e^{-2u} |v_{s_1}|^2\right)\circ\hat{F}_{s_1},\qquad \tilde{m}\text{-a.e.}
\end{equation}
\end{cor}

\begin{proof}
We may assume that $v$ is bounded up to replacing it with $v_n:=\chi_{\{|v|\leq n\}}v$, using the fact that $|\mathrm{d}\hat{F}_s(v_n)|\circ\hat{F}_s=|\mathrm{d}\hat{F}_s(v)|\circ\hat{F}_s$ on $\{|v|\leq n\}$ and letting $n\to\infty$.

Now we observe that on the complement of $B$ both sides of \eqref{eq:endecr} are zero $\tilde{m}$-a.e. (as a consequence that $v$ is concentrated on $B$). So that we only need to prove
\[
\left(e^{-2u}|v_{s}|^2\circ\hat{F}_{s}\right)\chi_B\leq \left(e^{-2u} |v|^2\right)\chi_B,\qquad\tilde{m}\text{-a.e.}
\]

Observe that by Lemma \ref{lem:Dephillipis_Gigli_3.30} the derivative of $s\mapsto u\circ\hat{F}_s$ is 
\begin{equation}
\label{u_decreases_exponentialy}
\frac{d}{ds} u\circ\hat{F}_s = \left\langle \nabla u, \nabla \hat{u}\right\rangle \circ\hat{F}_s = -u\circ\hat{F}_s.
\end{equation}
Therefore, integrating with respect to $s$ we obtain $u\circ\hat{F}_s = e^{-s}u$. As we are assuming $v$ is bounded, by Proposition \ref{DePhilippis_Gigli_Prop_3.31}, the map $s\mapsto \frac{|v_s|^2}
{2}\circ\hat{F}_s\chi_{B}\in L^1(\tilde{X})$ is $C^1$ and then
\begin{eqnarray*}
\label{eq-monotonicity-vector-fields}
\frac{d}{ds}\left(e^{-2u\circ\hat{F}_s}\frac{|v_s|^2}{2}\circ\hat{F}_s\chi_{B} \right) &=& \left( \frac{d}{ds}\left(e^{-2u\circ\hat{F}_s}\right)\frac{|v_s|^2}{2}\circ\hat{F}_s + \frac{d}{ds}\left(\frac{|v_s|^2}{2}\circ\hat{F}_s\right)e^{-2u\circ\hat{F}_s} \right)\chi_B \\
&=&\left( \frac{|v_s|^2}{2}\circ\hat{F}_s\left\langle\nabla e^{-2u}, \nabla \hat{u} \right\rangle\circ\hat{F}_s +e^{-2u\circ\hat{F}_s}\mathrm{Hess}[\hat{u}]\left(v_s,v_s \right)\circ\hat{F}_s\right)\chi_B\\
&=& \left((u\circ\hat{F}_s) e^{-2u\circ\hat{F}_s}|v_s|^2 \circ\hat{F}_s + e^{-2u\circ\hat{F}_s}\mathrm{Hess}[\hat{u}]\left(v_s,v_s \right)\circ\hat{F}_s   \right)\chi_B\\
&=& e^{-2u\circ\hat{F}_s}\left((u\circ\hat{F}_s-1)\left\langle \nabla u, v_s\right\rangle^2\circ\hat{F}_s\right)\chi_B\\
&\leq & 0.
\end{eqnarray*}

Recall that $\bar R < R \leq 1$, from which follows that $u(x) \leq 1$ for all $x \in B$. This concludes the proof. 
\end{proof}

\begin{prop}
\label{prop-ms.prj}
Let $\boldsymbol{\uppi}$ be a test plan and $\gamma:[0,1]\to u^{-1}([-\overline{R},\overline{R}])$. Then $\mathrm{ms}_t(\tilde{\gamma})\leq e^{-u(\gamma_t)}\mathrm{ms}_t(\gamma)$ for a.e. $t\in [0,1]$, $\boldsymbol{\uppi}$-a.e. $\gamma$, where $\tilde{\gamma}:= \pi \circ \gamma$.
\end{prop}

The proof of the  proposition resembles that of \cite[Proposition 3.33]{Gig}, as follows.

\begin{proof}
Abusing the notation we will still denote by $\hat{F}_s$ the map $C([0,1],\tilde{X})\to C([0,1],\tilde{X})$ taking $\gamma \mapsto \hat{F}_s\circ \gamma$. Recall that for every $t\in [0,1]$ the differential of $\hat{F}_s$ induces a map, still denoted by $d\hat{F}_s$, from $L^2(T\tilde{X}, e_t,\boldsymbol{\uppi})$ to $L^2(T\tilde{X}, e_t,\boldsymbol{\uppi}_s)$. We claim that for any $s_1\leq s_2$ and any $V \in L^2(T\tilde{X}, e_t,\boldsymbol{\uppi})$, 
\begin{equation}
\label{eq-decreasing}
\left(e^{-2u\circ e_t}|d\hat{F}_{s_2}(V)|^2\right)\circ\hat{F}_{s_2}\leq \left(e^{-2u\circ e_t}|d\hat{F}_{s_1}(V)|^2\right)\circ\hat{F}_{s_1} \quad \quad \boldsymbol{\uppi}-\text{a.e.}
\end{equation} 

To prove the claim we first consider $V$ to be of the form $e^*_tv$ for some $v\in L^2(T\tilde{X})$. By Proposition \ref{cor-metric-speed-1}, for $s_1\leq s_2$, $\boldsymbol{\uppi}$-a.e. ,  
\begin{eqnarray*}
\left(e^{-2u\circ e_t}|d\hat{F}_{s_2}(e^*_tv)|^2\right)\circ\hat{F}_{s_2} &=&\left(e^{-2u\circ e_t} |e^*_td\hat{F}_{s_2}(v)|^2\right)\circ\hat{F}_{s_2}\\
& =& \left(e^{2u}|d\hat{F}_{s_2}(v)|^2\right)\circ e_t\circ \hat{F}_{s_2}\\
& = &  \left(e^{2u}|d\hat{F}_{s_2}(v)|^2\right)\circ \hat{F}_{s_2}\circ e_t \\
&\leq &  \left(e^{2u}|d\hat{F}_{s_1}(v)|^2\right)\circ \hat{F}_{s_1}\circ e_t\\
& = & \left(e^{-2u\circ e_t}|d\hat{F}_{s_1}(e^*_tv)|^2\right)\circ\hat{F}_{s_1}.
\end{eqnarray*}

Let $(A_i)_{i\in\mathbb{N}}$ be a Borel partition  of $C([0,1],\tilde{X})$. The locality property of $d\hat{F}_s: L^2(T\tilde{X}, e_t,\boldsymbol{\uppi})\to L^2(T\tilde{X}, e_t,\boldsymbol{\uppi}_s)$ implies that any combination of the form $\sum \chi_{A_i}e^{*}_tv_i$, with $v_i\in L^2(T\tilde{X})$, satisfies
\[
\left(e^{-2u\circ e_t}|d\hat{F}_{s_2}(\sum \chi_{A_i}e^{*}_tv_i)|^2\right)\circ\hat{F}_{s_2}\leq \left(e^{-2u\circ e_t}|d\hat{F}_{s_1}(\sum \chi_{A_i}e^{*}_tv_i)|^2\right)\circ\hat{F}_{s_1} \quad \quad \boldsymbol{\uppi}-\text{a.e.}
\]
As the elements of the form $\sum \chi_{A_i}e^{*}_tv_i$ are dense in $L^2(T\tilde{X}, e_t,\boldsymbol{\uppi})$ and $d\hat{F}_s$ is continuous when considered as a map $L^2(T\tilde{X}, e_t,\boldsymbol{\uppi})\to L^2(T\tilde{X}, e_t,\boldsymbol{\uppi}_s)$ the claim follows.

Let $(\boldsymbol{\uppi}_s)_t'\in L^2(T\tilde{X},e_t,\boldsymbol{\uppi}_s)$ be the speed at time $t$ of the test plan $\boldsymbol{\uppi}_s$. Applying \eqref{eq-decreasing} to $\boldsymbol{\uppi}_t'$ and using the Chain rule for speeds \cite[Proposition 3.28]{DePG} we obtain that for $s_1\leq s_2$ and a.e.\ $t\in[0,1]$,
\[
\left(e^{-2u\circ e_t}|(\boldsymbol{\uppi}_{s_2})'_t|^2\right)\circ\hat{F}_{s_2}\leq \left(e^{-2u\circ e_t}|(\boldsymbol{\uppi}_{s_1})'_t|^2\right)\circ\hat{F}_{s_1},\qquad \boldsymbol{\uppi}-a.e..
\]
Now we integrate with respect to $t$ and recall the link between point-wise norm and metric speed given in \cite[(3.58)]{DePG} to obtain,
\begin{equation}
\label{eq-monotonicity2}
\dint\!\!\!\dint_0^1 e^{-2u(\gamma_t)}|\dot{\gamma}_t|^2 \mathrm{d}t\ \mathrm{d}\boldsymbol{\uppi}_{s_2}(\gamma) \leq \dint\!\!\!\dint_0^1 e^{-2u(\gamma_t)}|\dot{\gamma}_t|^2  \mathrm{d}t\ \mathrm{d}\boldsymbol{\uppi}_{s_1}(\gamma).
\end{equation}

The lower semicontinuity of the corresponding functional follows analogously as in \cite[Proposition 3.33]{DePG}. Now let us consider the functions $\hat{F}_s$ as functions from $\overline{B}$ to itself, and recall that they converge uniformly to the projection map $\pi:\tilde{X}\to u^{-1}(0)$ as $s\to\infty$. Then the test plans $\boldsymbol{\uppi}_s$ weakly converge to $\pi_*\boldsymbol{\uppi}$ as $s\to \infty$ and therefore,  
\[
\dint\! \dint_0^1 |\dot{\gamma}_t|^2  \mathrm{d} \pi_*\,\boldsymbol{\uppi} \leq \liminf_{s\to\infty} \dint\!\dint_0^1 e^{-2u(\gamma_t)}|\dot{\gamma}_t|^2  \mathrm{d} \boldsymbol{\uppi}_s .
\]
From the last expression it follows that 
\[
\dint\! \dint_0^1 \mathrm{ms}^2_t(\pi\circ\gamma)  \mathrm{d}t\ \mathrm{d} \boldsymbol{\uppi} \leq \dint\!\! \dint_0^1  e^{-2u(\gamma_t)}\mathrm{ms}^2_t(\gamma) \mathrm{d}t\ \mathrm{d} \boldsymbol{\uppi}.
\]
Now, the argument to conclude the proof from this integral formulation follows exactly as the corresponding part of \cite[Proposition 3.33]{DePG}. 
\end{proof}

\begin{proof}[Proof of Proposition \ref{prop-mst0}]
We start by proving (1). By Proposition \ref{prop-ms.prj}, (1) holds for $\gamma \in u^{-1}[-\overline R, \overline R]$. Proceeding as in the proof of Proposition \ref{prop-ms.prj} it is possible to show that if $\boldsymbol{\uppi}$ is a test plan and $\gamma:[0,1]\to u^{-1}([c-\overline{R},c+\overline{R}])$. Then $\mathrm{ms}_t(pr_c{\gamma})\leq e^{-u(\gamma_t)+c}\mathrm{ms}_t(\gamma)$ for a.e. $t\in [0,1]$, $\boldsymbol{\uppi}$-a.e. $\gamma$, where $pr_c{\gamma}:= F_{-u(\gamma) +c} \circ \gamma$. Let $c=\overline R$ and 
\[
\gamma:[0,1]\to u^{-1}([c-\overline{R},c+\overline{R}])=u^{-1}([0,2\overline{R}]).
\]
It follows by  iii) in Theorem \ref{thm:rapprcont} and Proposition \ref{prop-ms.prj}  that for almost every $t \in [0,1]$,
\begin{equation*}
e^{-\overline R}\mathrm{ms}_t(pr_{\overline R}(\gamma)) \leq \mathrm{ms}_t(pr_0(pr_{\overline R}{\gamma}))   \leq   e^{-\overline R}\mathrm{ms}_t(pr_{\overline R}(\gamma)).
\end{equation*}
Note that $pr_0({\gamma})=pr_0(pr_{\overline R}{\gamma})$. Thus, for almost every $t \in [0,1]$,

\begin{equation*}
\mathrm{ms}_t(pr_0{\gamma}) = e^{-\overline R}\mathrm{ms}_t(pr_{\overline R}(\gamma))  \leq   e^{-\overline R} e^{-u(\gamma_t)+ \overline R}\mathrm{ms}_t(\gamma).
\end{equation*}
This shows that (1) is satisfied for curves on $u^{-1}(([0,2\overline{R}])$. Proceeding in the same way, (1) follows. 

Now we prove part (2). Let $x,y \in B(x_0,r)$ and $\gamma:[0,1]\to \tilde{X}$ be a minimal geodesic joining them.  As $u$ is $1$-Lipschitz : 
\[u(\gamma_t) \geq \max\{u(\gamma_0), u(\gamma_1)\} - \tilde d(\gamma_0,\gamma_1),\]
\[  u(\gamma_0) \geq -r + u(x_0),\]
\[  u(\gamma_1) \geq -r + u(x_0).\]
Thus,  $u(\gamma_t) \geq  - r + u(x_0)  - 2r= u(x_0) -3r$. From the previous paragraph  $|\dot{\tilde\gamma}_t| \leq e^{-u(\gamma_t)}\ |\dot\gamma_t|$. Therefore, 
$d'(\pi(x),\pi(y)) \leq L(\tilde \gamma) \leq e^{-u(x_0)+3r} \tilde d(x,y)$. 
\end{proof}


\subsection{Properties of the quotient metric measure space}

Here we show that $(X',d')$ is a complete, separable, and geodesic metric space.  Then we define a measure  $m'$ on $X'$  and study the relationship between the spaces $W^{1,2}(X',d',m')$  and $W^{1,2}(\tilde X, d, \tilde m)$. At the end of the subsection we show that  $(X',d',m')$ is an infinitesimally Hilbertian space that satisfies the Sobolev to Lipschitz property.

\begin{thm}\label{thm-piCont}
With the same notation and assumptions of Definition \ref{def-X'},  $(X',d')$ is a complete, separable and geodesic metric space. 
\end{thm}

\begin{proof}
By Proposition \ref{prop-mst0} the  map $\pi$ is continuous, we will show that $X'$ is separable. Since $\tilde X$ is separable there exists a countable dense subset $\{ x_j \} \subset \tilde X$. 
Consider an open set $U \subset X'$, then $\pi^{-1}(U)$ is open in $\tilde X$.
As $\tilde X$ is separable there exists $x_j \in \pi^{-1}(U)$, and then $\pi (x_j )\in U$. Thus, $\{ \pi (x_j) \}$ is a dense subset of $X'$. 

To prove that $(X',d')$ is complete let $\{ x_j \} \subset  X'$ be a Cauchy sequence. Then, because $\iota: X' \to \tilde X$ is $1$-Lipschitz, $\{ \iota(x_j) \}$ is a Cauchy sequence in $\tilde X$, and hence it has a convergent subsequence $\iota (x_{j_k}) \to x$. 
Given that $\pi$ is continuous, $x_{j_k}=\pi(\iota(x_{j_k}) )\to \pi (x)$. 

To prove that $(X',d')$ is a geodesic space recall that a complete, locally compact length space is geodesic. So it is enough to prove that $(X',d')$ is locally compact. This is very similar to the previous paragraph. 
Let $x \in X'$ and $r>0$. If  $\{ x_j \} \subset  \overline{B^{d'}(x,r)}$, then $\{ \iota(x_j) \} \subset  \overline {B^{\tilde d}( \iota(x), r)}$. 
Now, since $(\tilde X, \tilde d)$ is locally compact, there exists a convergent subsequence $\iota (x_{j_k}) \to y$. 
Because $\pi$ is continuous, $x_{j_k}=\pi(\iota(x_{j_k}) )\to \pi (y)$ and $d'(\pi(y),x)= \lim_{k \to \infty} d'(x_{j_k},x) \leq r$. This concludes the proof. 
\end{proof}

Given that $u: \tilde X \to \mathbb R$ and $\pi: \tilde X \to X'$ are continuous (see (2) in Proposition \ref{prop-mst0} where it is shown that $\pi$ is locally Lipschitz and recall that $u$ is Lipschitz),  we define a Borel measure on $X'$.

\begin{defn}\label{def:mmp} 
We define the measure $m'$ on $(X',d')$ by 
$$
m' (A) = ( \int_0^1 e^{(N-1)s} ds )^{-1} \tilde m ( \pi^{-1}(A) \cap u^{-1}[0,1])
$$
for any Borel set $A \subset X'$. 
\end{defn}

\begin{lem}\label{lem-Aab}
Given a Borel set $A \subset X'$, let $A_a^b=\{ x \in \tilde X | u(x) \in [a,b], \,  \pi(x) \in A \}$. Then, 
\begin{equation}\label{eq-Aab}
\tilde m (A_a^b) =   m'(A)  \int_{a}^{b} e^{(N-1)s} \, \mathrm{d}s.
\end{equation}
\end{lem}

\begin{proof}
The proof follows that of Proposition 5.28 \cite{Gig}. For completeness we give some details. Note that by the definition of $m'$, equation (\ref{eq-Aab}) holds for $a=0$ and $b=1$.  By Proposition \ref{prop-F_tPushMeas} and Theorem \ref{thm:rapprcont},  we know that ${F_a}_\sharp \tilde m= e^{-(N-1)a}\tilde m$ and $F^{-1}_a=F_{-a}$.  Thus, 
\begin{align*}
\tilde m (A_a^{a+1})  = & \, e^{(N-1)a} \tilde m (  {F_a}^{-1}(A_a^{a+1}))  = e^{(N-1)a}  \tilde m (A_0^{1})  \\
=    & \, m'(A)  \int_{0}^{1} e^{(N-1)a}  e^{(N-1)s} ds =       m'(A)  \int_{a}^{a+1}  e^{(N-1)s} \,\mathrm{d}s.
\end{align*}

To prove that equation (\ref{eq-Aab}) holds for $a=0$ and $b=1/2$, we use again Proposition \ref{prop-F_tPushMeas} and Theorem \ref{thm:rapprcont}. Thus, 
\[
\tilde m (A_0^{1})  =  \tilde m (A_0^{1/2})  + \tilde m (A_{1/2}^{1}) = (1 + e^{\tfrac 1 2(N-1)} )\tilde m (A_0^{1/2}).
\]
With some algebra we conclude 
\[
\tilde m (A_0^{1/2}) =  (1 + e^{\tfrac 1 2(N-1)} )^{-1} \tilde m (A_0^{1}) =         m'(A)  \int_{0}^{1/2}  e^{(N-1)s} \, \mathrm{d}s.
\]
Continuing in this way,  equation (\ref{eq-Aab}) holds for $a \in \mathbb R$ and $b=a+k/2^n$ with $k,n \in \mathbb N$. 
Then an approximation argument concludes the proof. 
\end{proof}

\begin{prop}\label{prop-DfDg1}
Let $h \in \Lip(\mathbb R)$ with compact support and identically $1$ on $[a,b]$. Let $f \in L^2(\tilde X)$ be of the form 
$f(x)= g(\pi(x))h(u(x))$ for some $g \in L^2(m')$. If $f \in W^{1,2}(\tilde X)$ then $g \in W^{1,2}(X')$
and for $\tilde m$-a.e. $x \in u^{-1}[a,b]$ we have
\begin{equation}
|\nabla g|_{X'} (\pi(x)) \leq e^{u(x)} |\nabla f|_{\tilde X} (x).
\end{equation}
\end{prop}

\begin{proof}
Let $\boldsymbol{\uppi}'$ be a test plan on $X'$.  Define
\[
T: X' \times [a',b'] \to \tilde X,\quad  \hat T: C([0,1], X') \times [a',b'] \to C([0,1], \tilde X),
\]
 and $\boldsymbol{\uppi} \in \mathcal P (C([0,1], \tilde X))$ given by $T(x,s)= F_s(\iota (x))$,  $\hat T (\gamma, s)_t= T(\gamma_t,s)$ and 
\[
\boldsymbol{\uppi} = \hat T_\sharp (\boldsymbol{\uppi}' \times  (b'-a')^{-1} \mathcal L^1_{[a',b']}),
\] 
with $[a',b'] \subset [a,b]$.

We claim that $\boldsymbol{\uppi}$ is a test plan on $\tilde X$. That is, $\boldsymbol{\uppi}$ has finite kinetic energy and bounded compression.  Finite kinetic energy for $\boldsymbol{\uppi}$ follows from the fact that $\boldsymbol{\uppi}'$ is a test plan and so it has finite kinetic energy, and that $\mathrm{ms}_t(\hat{T}(\gamma,s)) \leq  \Lip (F_s) |\dot\gamma_t|$ (where $\mathrm{ms}_t(\hat{T}(\gamma,s))$ denotes the metric speed of $\hat{T}(\gamma,s)$), by Theorem \ref{thm:rapprcont} and Theorem \ref{lem-d'}. Set $M=\max\{ \Lip (F_s) \,|\, s\in [a',b']\}$, then,
\begin{align*}
\frac{1}{2} \int \int_0^1 |\dot\gamma_t|^2 \, \mathrm{d} t\, \mathrm{d}\boldsymbol{\uppi} (\gamma) = &
\frac{1}{2} \int \int_0^1 \int_{a'}^{b'} (b'-a')^{-1} \mathrm{ms}_t(\hat{T}(\gamma,s))^2 \, \mathrm{d}s \,\mathrm{d}t\, \mathrm{d}\boldsymbol{\uppi'} (\gamma) \\
\leq& \, M \frac{1}{2} \int \int_0^1  |\dot\gamma _t|^2 \, \mathrm{d}t \,\mathrm{d}\boldsymbol{\uppi'} (\gamma)  < \infty.
\end{align*}

To show that $\boldsymbol{\uppi}$ has bounded compression it is enough to consider sets of the form 
\[A_{c}^{d}=\{ x \in \tilde X \, | \, u(x) \in [c,d], \,  \pi(x) \in A \},\]
 for some Borel set $A\subset X'$. Thus, using that $\boldsymbol{\uppi}'$ has bounded compression, and equation (\ref{eq-Aab}),
\begin{align*}
{e_t}_\sharp \boldsymbol{\uppi} (A_{c}^{d} )=&\, \boldsymbol{\uppi}' \times   (b'-a')^{-1} \mathcal L^1_{[a',b']} ( ( e_t \circ \hat T )^{-1}(A_{c}^{d}) )\\
=& \,\boldsymbol{\uppi}' (e_t^{-1}( A ) )      (b'-a')^{-1} \mathcal L^1_{[a',b]} ([c,d]) \\ 
\leq&\, C m'(A). 
\end{align*}

The definitions of $\boldsymbol{\uppi}$ and $f$ yield, 
\begin{equation}\label{eq-Df=Dg}
\int | f(\gamma_1) - f(\gamma_0)|\, \mathrm{d}\boldsymbol{\uppi} (\gamma)=  \int | g(\gamma_1) - g(\gamma_0)|\, \mathrm{d}\boldsymbol{\uppi}' (\gamma).
\end{equation}
Now, the definitions of $|\nabla f|_{\tilde X}$, and $\boldsymbol{\uppi} $, imply the following estimates: 
\begin{align}\label{eq-DfDg}
\int | f(\gamma_1) - f(\gamma_0)| \, \mathrm{d}\boldsymbol{\uppi}  \leq & \int \int^1_0 |\nabla f|_{\tilde X} (\gamma_t) |\dot\gamma|\,  \mathrm{d}t \, \mathrm{d}\boldsymbol{\uppi}\\
\leq & \int \int^1_0 \int^{b'}_{a} (b'-a')^{-1}|\nabla f|_{\tilde X} (\hat T(\gamma,s)_t) \mathrm{ms}_t(\hat T (\gamma,s))\,  \mathrm{d}s \, \mathrm{d}t \, \mathrm{d}\boldsymbol{\uppi}'  \nonumber \\
\leq & \int \int^1_0 (b'-a')^{-1} \int^b_a  \Lip (F_s)   |\nabla f|_{\tilde X} (\hat T(\gamma,s)_t) | \dot\gamma_t|\, \mathrm{d}s\, \mathrm{d}t\, \mathrm{d}\boldsymbol{\uppi}' \nonumber.
\end{align} 
In the previous inequalities we used  Theorem \ref{thm:rapprcont} to bound $\mathrm{ms}_t(\hat T(\gamma,s))  \leq  \Lip (F_s) |\dot\gamma_t|$.

Combining equality (\ref{eq-Df=Dg}), inequality (\ref{eq-DfDg}), and that $\boldsymbol{\uppi}'$ was chosen arbitrarily, we conclude that  $g \in W^{1,2}(X')$, and the right hand side above is an upper gradient for $g$, i.e. that for $m'$-a.e. $x'$,
\begin{equation*}
|\nabla g|_{X'}(x') \leq (b'-a')^{-1} \int^{b'}_{a'}  \Lip (F_s) |\nabla f|_{\tilde X} ( T(x',s)) \, \mathrm{d}s. 
\end{equation*}
This proves $g \in W^{1,2}(X')$, and for $0 \leq a < b$, gives the right estimate for $|\nabla g|_{X'}$ when we let $a'$ and $b'$ converge to $u(x)$,  as for $s \in [a,b]$ the inequality $\Lip (F_s)  \leq \max\{e^{s}, 1\}=e^s$ holds by Theorem \ref{thm:rapprcont} part (ii).

If $a < 0$, write $f=\tilde f \circ F_t$, here $t \geq -a$.  Then $\tilde f (x)= g(\pi(x))$ for $x \in u^{-1}[a+t,b+t]$ and $0 \leq a + t \leq b + t$. 
We note that $\left\langle\nabla  f , \nabla  u \right\rangle=0$. Then by the definition of Regular Lagrangian Flow, Definition \ref{def-regular-lagrangian-flow} (iii),  and Corollary 
\ref{cor-hessian-identity}, the equality $|\nabla f|^2_{\tilde X}=  {\rm Hess}[u] (\nabla f, \nabla f)$ holds $\tilde m$ a.e. in $u^{-1}[a,b]$. In combination with Theorem \ref{thm-g} we have thus found,
\[
\left\langle\nabla  f , \nabla  f \right\rangle =  e^{2t} \mathrm{Hess}(u)(\nabla \tilde f, \nabla \tilde f) \circ F_t +  \left\langle\nabla \tilde f, \nabla u\right\rangle^2 \circ F_t=  e^{2t} \left\langle\nabla  \tilde f , \nabla \tilde  f \right\rangle  \circ F_t.
\]
The previous equality holds for $0 \leq a \leq b$, and so we conclude that $m'$-a.e. $x'$,
\[|\nabla f | (x) = e^t |\nabla \tilde f | (F_t(x)) \geq e^t e^{-(u(x)+t)} |\nabla  g |_{X'} (\pi(x)).\]
\end{proof}

\begin{thm}\label{thm-DfDg2}
Assume $ h \in \Lip(\mathbb R)$ has compact support and is identically  $1$ on $[a,b]$. Let  $f \in L^2(\tilde X)$ be of the form 
$f(x)= g(\pi(x))h(u(x))$, for some $g \in L^2(X',m')$.  Then $g \in W^{1,2}(X',d',m')$ if and only if $f \in W^{1,2}(\tilde X, \tilde d, \tilde m)$,
and for $\tilde m$-a.e.\ $x \in u^{-1}[a,b]$ we have 
\begin{equation}
|\nabla f|_{\tilde X} (x) = e^{-u(x)} |\nabla g|_{X'} (\pi(x)).
\end{equation}
\end{thm}

\begin{proof}
By Proposition \ref{prop-DfDg1} it is enough to prove that if $g \in W^{1,2}(X',d',m')$ then $f \in W^{1,2}(\tilde X, d, \tilde m)$
and $|\nabla f|_{\tilde X} (x) \leq e^{-u(x)} |\nabla g|_{X'} (\pi(x))$ holds for $\tilde m$-a.e.\ $x \in u^{-1}[a,b]$. Let $G: \tilde X \to \mathbb R$ be given by 
\begin{equation}
G(x)= e^{-u(x)} |Dg|_{X'} (\pi(x)) h(u(x)) + g(\pi(x)) |h'| (u(x)).
\end{equation}
We will show that $G$ is a weak upper gradient of $f$. Notice that $G$ is in $L^2(\tilde m)$ and that $G(x)= e^{-u(x)} |\nabla g|_{X'} (\pi(x))$ for $x \in u^{-1}[a,b]$. 

For $x \in \supp (f)$ following the same arguments of the proof of Theorem 4.19 in [AGS114] (this is the property of weak gradient being a local object) it is sufficient to check the definition of weak upper gradients for $f$ using test plans $\boldsymbol{\uppi}$ such that  for each $t \in[0,1]$, 
\[
\gamma_t \subset A(x,r)= \{ y \in \tilde X \, | \,   u(y) \in [u(x)-r,u(x)+ r], \,\, d'(\pi(x),\pi(y)) \leq r \}
\]
 and $\gamma \in  \supp(\boldsymbol{\uppi})$.
Fix such $\boldsymbol{\uppi}$.  By  (2) in Proposition \ref{prop-mst0}, the map $\hat \pi: C([0,1],  \overline{A(x,r)}) \to C([0,1],  X')$ given by $\hat \pi (\gamma)= \pi \circ \gamma$ is Lipschitz.  Arguing as in the proof of Proposition \ref{prop-DfDg1}, we conclude that $\boldsymbol{\uppi}'= \hat \pi_\sharp \boldsymbol{\uppi}$ is a test plan on $X'$.

Since $g \in W^{1,2}(X')$ and the way $\boldsymbol{\uppi}'$ was defined, by Proposition \ref{prop-weakUg}, for $\boldsymbol{\uppi}$-a.e. $\gamma$ the map $t \mapsto g( \hat \pi(\gamma)_t)$ is equal a.e. on $[0,1]$ and $\{0,1\}$ to an absolutely continuous map $g_{\hat \pi(\gamma)}$ such that for a.e. $t \in [0,1]$
\begin{equation}
|g'_{\hat \pi(\gamma)}|(t) \leq |\nabla g|_{X'} ( \hat \pi(\gamma)_t )| \hat \pi\dot(\gamma)_t | \leq  e^{-u(\gamma_t)}  |\nabla g|_{X'} (\pi(\gamma_t) |\dot\gamma_t|.
\end{equation}
In the last inequality we used  (1) from Proposition \ref{prop-mst0}.

For any absolutely continuous curve $\gamma$ in $\tilde X$,  $h\circ u \circ \gamma$ is absolutely continuous with derivative $|(h\circ u \circ \gamma)'| \leq |h'| (u \circ \gamma) |\dot\gamma|$. Hence, for $\boldsymbol{\uppi}$-a.e. $\gamma$ the map $t \mapsto f(\gamma_t)=g(\pi(\gamma_t))h(u(\gamma_t))$ is equal a.e. on $[0,1]$ and $\{0,1\}$ to the absolutely continuous map $f_\gamma(t)= g_{\hat \pi(\gamma)}(t) h(u ( \gamma_t))$ such that for a.e. $t \in [0,1]$ satisfies
\begin{equation}
|f'_\gamma|(t) \leq  \big(  e^{-u(\gamma_t)}  |\nabla g|_{X'} (\pi(\gamma_t)) h(u ( \gamma_t)) +  g(\pi(\gamma_t)) |h'| (u \circ \gamma_t) \big) |\dot\gamma_t|.
\end{equation}
This proves that for $\tilde m$-a.e. $x \in u^{-1}[a,b]$
\begin{equation*}
|\nabla f|_{\tilde X} (x) \leq e^{-u(x)} |\nabla g|_{X'} (\pi(x)).
\end{equation*}
\end{proof}

\begin{prop}\label{prop-SobtoLip} 
Under the assumptions of Definition \ref{def-X'}  and Definition \ref{def:mmp}, the space  $(X',d',m')$ 
is infinitesimally Hilbertian, almost everywhere locally doubling and a measured-length space. Hence, it satisfies the Sobolev to Lipschitz property.
\end{prop}

For the definition of locally doubling and measured-length space see Definition \ref{def:measL} and Definition \ref{def:aeLoc} in Subsection \ref{ssec-warp}. 

\begin{proof}
By Theorem \ref{thm-DfDg2} and the infinitesimally Hilbertianity of $(\tilde X, \tilde d, \tilde m)$ it is easy to see that  $(X',d',m')$ is infinitesimally Hilbertian.  We now prove that $(X',d',m')$ is everywhere locally doubling. 

\bigskip 
 
Given $x' \in X'$ and $R > 0$, for $r < R/2$ define
\begin{equation}\label{eq-A(x,r)}
A(x',r) = \{ x \in \tilde X \, | \, u(x) \in [-r,r], \, d'(x', \pi(x)) < r \} \subset B(\iota (x'),2r).
\end{equation}

By (2) in Proposition \ref{eq-prmu}, there exists a Lipschitz constant $L >1$ for $\pi : \overline{B(\iota(x'), R)} \to X'$. Notice that $B(\iota (x'), r/L)  \subset B(\iota (x'),2r)$ because $L > 1$. Since $u$ is 1-Lipschitz and by the triangle inequality,  if $y \in B(\iota(x'), 2r)$ then $|u(y)| \leq u(\iota(x')) + \tilde d(y, \iota(x')) \leq 2r$. Thus,  $B(\iota(x'),2r) \subset \overline{B(\iota(x'),R)}$.  This shows $d'(\pi(y), x') \leq r$ for any $y \in B(\iota (x'),r/L)$. Since $u$ is $1$-Lipschitz it follows that 
$u(y) \leq r/L < r$ for any $y \in B(\iota (x'), r/L)$. Thus, 
\begin{equation}\label{eq-m2}
B(\iota (x'), r/L) \subset A(x',r).
\end{equation}

Equation (\ref{eq-Aab}) gives 
\begin{equation}\label{eq-m1}
\tilde m (A(x',r)) = m'(B'(x',r)) \int_{-r}^r e^{(N-1)s}\, \mathrm{d}s. 
\end{equation}
Let  $c(r)= \int_{-r}^r e^{(N-1)s} ds$.  Starting with equation  (\ref{eq-m1}), then using equation (\ref{eq-m2}), that $(\tilde X, \tilde d, \tilde m)$ is locally doubling with constant $C_{\tilde X}$ \cite{Villani09}, equation (\ref{eq-A(x,r)}) and equation  (\ref{eq-m1}) once more, we estimate 
\begin{align*}
m'(B'(x', r))  = & \,c^{-1}(r) \tilde  m (A(x,r)) \geq  c^{-1}(r) \tilde m (B(\iota (x'), r/L)) \\
\geq &\,C_{\tilde X}  c^{-1}(r) \tilde  m (B(\iota (x'), r/2L))  \geq  C_{\tilde X} c^{-1}(r) \tilde  m (A(x,r/4L)) \\
= &\, C_{\tilde X} c^{-1}(r) c(r/4L)  m' (B'(x', r/4L)). 
\end{align*}
That is,  $m'(B'(x', r))   \geq  C m' (B'(x', r/4L))$, for $C=  C_{\tilde X} c^{-1}(r) c(r/4L)$. Therefore $(X',d',m')$ is  almost everywhere locally doubling.

\bigskip
Now we show that $(X',d',m')$ is a measured-length space. Let $x_0, x_1 \in X'$, define $\varepsilon =1$ and 
take $\varepsilon_0, \varepsilon_1 \in (0,\varepsilon]$. Let $\tilde \gamma$ be a geodesic in $X'$ from $x_0$ to $x_1$, and
$x_i= \tilde \gamma_{i/n}$ for $i=0,1,...,n$, $n= \lfloor1+ 1/\sqrt{\varepsilon'} \rfloor$  and $\varepsilon'=\max\{ \varepsilon_0, \varepsilon_1\}$. 

Let $\varepsilon_i= \varepsilon_0 + \frac{i}{n}(\varepsilon_1-\varepsilon_0)$, and define
$\mu_i^{\varepsilon_0,\varepsilon_1}= (\tilde m (A(x_i,\varepsilon_i))^{-1} \tilde m|_{A(x_i,\varepsilon_i)}$. Here $A(x_i,\varepsilon_i)$ is defined by equation (\ref{eq-A(x,r)}). From equation (\ref{eq-Aab}), 
\begin{equation}\label{eq-prmu}
\pi_\sharp \mu_i^{\varepsilon_0,\varepsilon_1}= (m' (B'(x_i, \varepsilon_i)))^{-1} m'|_{B'(x_i,\varepsilon_i)}.
\end{equation}
Let $\boldsymbol{\uppi}_i^{\varepsilon_0,\varepsilon_1}$ be the only optimal geodesic plan from  
$\mu_i^{\varepsilon_0,\varepsilon_1}$ to $\mu_{i+1}^{\varepsilon_0,\varepsilon_1}$ ([GRS16]). By the triangle inequality and our choices of $x_i$ and $\varepsilon_i$, for $y_i \in A(x_i,\varepsilon_i)$ we have 
\[
\tilde d(y_i,y_{i+1}) \leq 2\varepsilon_i + d'(x_i,x_{i+1}) + 2\varepsilon_{i+1} \leq 4\varepsilon' + \frac{1}{n}d'(x_0,x_{1}).
\]
  It follows that 
\begin{equation}\label{eq-W22}
\int \int_0^1 |\dot\gamma_t|^2 \, \mathrm{d}t\, \mathrm{d} \boldsymbol{\uppi}_i^{\varepsilon_0,\varepsilon_1} (\gamma) = W^2_2( \mu_i^{\varepsilon_0,\varepsilon_1},  \mu_{i+1}^{\varepsilon_0,\varepsilon_1}) \leq  (4\varepsilon' + \frac{1}{n}d'(x_0,x_{1}))^2.
\end{equation}
From the definition of $\varepsilon$ and $\varepsilon'$ for  $\boldsymbol{\uppi}_i^{\varepsilon_0,\varepsilon_1}$ a.e. $\gamma$, $u (\gamma_t ) \subset  [-\varepsilon',\varepsilon'] \subset [-1, 1]$.

\bigskip

Gluing the plans $\boldsymbol{\uppi}_i^{\varepsilon_0,\varepsilon_1}$ we construct a plan $\boldsymbol{\uppi}^{\varepsilon_0,\varepsilon_1}$ that satisfies
\begin{enumerate}[(i)]
\item
\[
(\op{Restr}^{\frac {i+1}{n}}_{\frac{i}{n}})_\sharp \boldsymbol{\uppi}^{\varepsilon_0,\varepsilon_1} = \boldsymbol{\uppi}_i^{\varepsilon_0,\varepsilon_1}, \quad  i=0,1,...,n,
\]
where $\op{Restr}^{b}_{a}$ is the restriction operator to $[a,b]$.
\item
\begin{align}\label{eq-integral}
\int \int_0^1 |\dot{\gamma}_t|^2\, \mathrm{d}t \, \mathrm{d} \boldsymbol{\uppi}^{\varepsilon_0,\varepsilon_1} (\gamma) = &\, n\sum_{i=0}^{n-1} \int \int_0^1 |\dot{\gamma}_t|^2 \, \mathrm{d}t \, \mathrm{d} \boldsymbol{\uppi}_i^{\varepsilon_0,\varepsilon_1} (\gamma)  \nonumber \\
\leq & n^2   (4\varepsilon' + \frac{1}{n}d'(x_0,x_{1}))^2 \leq (8 \sqrt{\varepsilon'}+ d'(x_0,x_{1}))^2.
\end{align}
Note that $n=\lfloor1+ 1/\sqrt{\varepsilon'} \rfloor$ and $\varepsilon' < 1$ implies $4n\varepsilon' \leq 8 \sqrt{\varepsilon'}$. Then using (\ref{eq-W22}) and taking into account the rescaling factor we get the previous inequality. 
\item for
$\boldsymbol{\uppi}^{\varepsilon_0,\varepsilon_1}$ a.e. $\gamma$,
\begin{equation}\label{eq-uEst2}
u (\gamma_t ) \subset [-\varepsilon' , \varepsilon'] \subset [-1, 1].
\end{equation}
\end{enumerate}

\bigskip

Define:
\begin{equation}\label{eq-planMS}
\overline{\boldsymbol{\uppi}}^{\varepsilon_0,\varepsilon_1} := \pi_\sharp \boldsymbol{\uppi}^{\varepsilon_0,\varepsilon_1}
\end{equation}

From (\ref{eq-prmu})  we get 
\[
 {e_i}_\sharp \bar{\boldsymbol{\uppi}}^{\varepsilon_0, \varepsilon_1} = \frac 1 {m'(B'(x_i,\varepsilon_i) )}  m'|_{B' (x_i, \varepsilon_i)}\,\,\,\,i=0,1. 
 \]

By (1) in Proposition (\ref{eq-prmu}) we know that 
\begin{equation*}
\int \int_0^1 |\dot{\gamma}_t|^2 \, \mathrm{d}t \,\mathrm{d} \bar{\boldsymbol{\uppi}}^{\varepsilon_0,\varepsilon_1} (\gamma) \leq 
\int \int_0^1  e^{u(\gamma_t)}   |\dot{\gamma}_t|^2\, \mathrm{d}t\, \mathrm{d} \boldsymbol{\uppi}^{\varepsilon_0,\varepsilon_1} (\gamma). 
\end{equation*}
From equations (\ref{eq-integral}), (\ref{eq-uEst2}), and $\varepsilon'=\max\{ \varepsilon_0, \varepsilon_1\}$ it follows that 
\begin{equation*}
\limsup_{\varepsilon_0, \varepsilon_1 \downarrow 0}   
 \int \int_0^1 |\dot{\gamma}_t|^2 \,\mathrm{d}t \,\mathrm{d} \bar{\boldsymbol{\uppi}}^{\varepsilon_0,\varepsilon_1} (\gamma) \leq 
\limsup_{\varepsilon_0, \varepsilon_1 \downarrow 0}     
e^{\varepsilon ' }(8 \sqrt{\varepsilon '} + d'(x_0,x_{1}))^2= d'(x_0,x_{1})^2.
\end{equation*}
\end{proof}


\section{$(\tilde X, \tilde d, \tilde m)$ is isomorphic to $(X'_{\omega},d'_{\omega}, m'_{\omega})$}\label{sec-Isom}

Let $X'_w$ denote the warped product of $(X',d',m')$ with warping functions $w_{d'}, w_{m'}:\mathbb R \to \mathbb R $ given by  $w_{m'}(t)=e^{(N-1)t}$ and $w_{d'}(t)=e^{t}$. In Subsection \ref{ssec-measPres} we prove that there is a locally bi-Lipschitz map from $(\tilde X, \tilde d, \tilde m)$ to $(X'_{\omega},d'_{\omega}, m'_{\omega})$ that preserves the measures.  Then we show that the spaces are isomorphic by showing that their $W^{1,2}$ spaces are isomorphic.

\subsection{$\tilde {X}$ is measure preserving homeomorphic to a warped product }\label{ssec-measPres}

Here we prove that there is a locally bi-Lipschitz map from $(\tilde X, \tilde d, \tilde m)$ to $(X'_{\omega},d'_{\omega}, m'_{\omega})$ that preserves the measures.

Proceeding as in Proposition \ref{prop-mst0}, or directly using the definition of $d'_w$,  we obtain the following. 

\begin{prop}\label{prop-warpmet}
For all  $(x'_0,t_0) \in X'_w$ and $r>0$,    
\begin{equation*}
d'(x',y')  \leq e^{-t_0+3r} d'_w((x',t),(y',s)),
\end{equation*} 
for all  $(x',t), (y',s) \in B((x_0,t_0),r)$.
\end{prop}

We are now ready to construct the locally  bi-Lipschitz maps.

\begin{prop}\label{prop:homeo}
Let $T: X'_w \to \tilde X$ and  $S: \tilde X \to X'_w $ be defined by 
$$T(x',t)=F_t(\iota (x'))$$
and 
$$S(x)=(\pi (x), u(x)).$$
Then $T$ and $S$ are inverses of each other,  $S$ is $2$-Lipschitz and
$T$ is locally Lipschitz.
\end{prop}

\begin{proof} 
It is clear that $T \circ S={\rm Id}_{\tilde X}$ and $S\circ T={\rm Id}_{X'_{\omega}}$.  Let us prove that $T$ is locally Lipschitz. 
Let $(x'_0,t_0) \in X'_w$ and $r>0$.  Consider $(x'_1,t_1),(x'_2,t_2) \in B((x_0,t_0),r)$.   By the triangle inequality, Theorem \ref{thm:rapprcont}, and Proposition \ref{prop-warpmet}, we obtain
\begin{align*}
\tilde d(T(x'_1,t_1),T(x'_2,t_2))  = &\, \tilde d(F_{t_1}(\iota(x'_1)),F_{t_2}(\iota(x'_2)))\\
 \leq & \, \tilde d( F_{t_1}(\iota(x'_1)), F_{t_1}(\iota(x'_2)) ) + \tilde d(   F_{t_1}(\iota(x'_2)), F_{t_2}(\iota(x'_2)) )  \\
\leq &\, \Lip(F_{t_1})d'(x'_1,x'_2) + |t_1-t_2|\\
 \leq & \,\Lip(F_{t_1}) e^{-t_0+3r}  d'_w((x',t_1),(y',t_2)) + d'_w((x',t_1),(y',t_2)). 
\end{align*}
It follows that $T$ is locally Lipschitz.

Now we prove that $S$ is Lipschitz.  
Let $\gamma: [0,1] \to \tilde X$ be a geodesic from $T(x'_1,t_1)$ to $T(x'_2,t_2)$. 
 As  $u: X \to \mathbb R$ is 1-Lipschitz,  the curve $u \circ \gamma$ is absolutely continuous and $| \dot{u}(\gamma_t)| \leq |\dot{\gamma}_t|$. From Proposition \ref{prop-mst0} (1), we know that $ e^{u(\gamma_t)}  |\dot{\tilde \gamma}_t| \leq |\dot{\gamma_t}|$, 
here $\tilde \gamma = \pi \circ \gamma$. Thus, 
\begin{align*}
2 \tilde d(T(x'_1,t_1),T(x'_2,t_2)) &=  2 \int  |\dot{\gamma}_t|\, \mathrm{d}t \\
&\geq \int e^{u(\gamma_t)}  |\dot{\tilde \gamma}_t| +  | \dot{u}(\gamma_t)|\, \mathrm{d}t\\
&\geq  \int  \sqrt{e^{2u(\gamma_t)}  |\dot{\tilde \gamma}_t|^2 +  | \dot{u}(\gamma_t)|^2} \, \mathrm{d}t \\
&\geq\   d'_w\big((x'_1,t_1),(x'_2,t_2)\big).
\end{align*}
\end{proof}

Applying Lemma \ref{lem-Aab} we see that  $T$ and $S$ are measure preserving:

\begin{prop}[$T$ and $S$ are measure preserving]\label{prop:prodmeas}
Let $T: X'_w \to \supp (\tilde m)$ and $S: \supp (\tilde m) \to X'_w$  be given by $T(x',t)= F_t(\iota (x'))$ and $S(x)=(\pi (x), u(x))$.  Then $T_\sharp(m'_{\omega})=\tilde m$ and  $S_\sharp \tilde m=m'_\omega$.
\end{prop}

\begin{proof}
As $S$ and $T$ are inverses of each other, it is sufficient to  prove that $S_\sharp \tilde m=m'_w$. Given that both $m'_w$ and $S_\sharp \tilde m$ are Borel measures defined on $X'_w$ which has positive warping functions, it is enough to prove that for any Borel set $E\subset X'$ and any interval $I=[a,b]\subset \mathbb R$ the following equality holds 

\begin{equation*}\label{eq:rettangoli}
S_\sharp \tilde m(E\times I)= m'_w(E \times I).
\end{equation*}

Equation (\ref{eq-Aab}) implies,
\begin{equation*}
S_\sharp \tilde m(E \times I)= \tilde m (S^{-1} (E \times I))= \tilde m (E_a^b)= m'(E) \int_a^b e^{(N-1)s}\,\mathrm{d}s.
\end{equation*}
By the definition of $m'_w$, 
\begin{eqnarray*}
m'_w(E \times I)=  \int_I \left(\int_{X'} \chi_E(x)w_{m'}(t)\,\mathrm{d}m'(x) \right)\,\mathrm{d}t =  m'(E) \int_a^b w_{m'}(t)\,\mathrm{d} t   = m'(E) \int_a^b e^{(N-1)t}\,\mathrm{d}t.
\end{eqnarray*}
\end{proof}

Before we continue with our discussion we establish the following lemma which is needed here.

\begin{lem}\label{lem-Dist}
For any $x, y\in \tilde X$, 
	\begin{equation}\label{eq-Dist}
	\tilde d(x, y) \geq |u(x)-u(y)|.
	\end{equation}
\end{lem}

\begin{proof}
Without loss of generality we assume that $u(x) \geq u(y)$. For any $t\in \mathbb R$ sufficiently negative, by the triangle inequality and Theorem \ref{thm:rapprcont} we have 
\begin{align} \label{eq-dist} \tilde d(x, y) & \geq \tilde d(x, F_t(x)) - \tilde d(F_t(x), y) \nonumber \\
& \geq  -t - \tilde d(F_t(x), F_{t+u(x)-u(y)}(y)) - \tilde d(F_{t+u(x)-u(y)}(y), y) \nonumber \\
& = -t +(t+u(x)-u(y)) -\tilde d(F_t(x), F_{t+u(x)-u(y)}(y)) \nonumber \\
& = u(x)-u(y) -\tilde d(F_t(x), F_{t+u(x)-u(y)}(y)).
\end{align}

Let $\gamma$ be a minimizing geodesic connecting $x$ to $y$ and $\gamma_1= F_{t+u(x)- u(\gamma)}(\gamma)$. Then by Proposition \ref{prop-mst0} (see also the proof where a similar shift is needed)
\[ \tilde d(F_t(x), F_{t+u(x)-u(y)}(y)) \leq L(\gamma_1) \leq e^{t+u(x) -C} L(\gamma), \] where $C=\min u(\gamma)$. Clearly the right hand side of the above inequality goes to zero as $t \rightarrow -\infty$. Thus, by taking $t \rightarrow -\infty$ in \eqref{eq-dist}, we obtain 
\[ \tilde d(x, y) \geq u(x)-u(y). \]
\end{proof}

The following proposition will be helpful in the next subsection.

\begin{prop}\label{prop:sezioni2}  
Let $h\in S^2_{\rm loc}( w_{m'}\mathbb R)$ and define $f:\tilde X\to\mathbb R$ by $f:=h\circ u$. Then $f\in S^2_{\rm loc}(\tilde X)$ and
\[
|\nabla f|_{\tilde X}(x)=|\nabla h|_{w_{m'}\mathbb R}(u(x)),\qquad \tilde m-a.e.\,  x\in \tilde X.
\]
\end{prop}

\begin{proof}
The proof follows the same strategy as that of \cite[Proposition 5.29]{Gig}. 

Let $R>0$ and $\chi:\mathbb{R}\to [0,1]$ be a Lipschitz function which is compactly supported and identically $1$ on $[-R,R]$. Firstly we observe that, since the claim is a local statement, to provide a proof it is enough to show that, if $h\in W^{1,2}_{loc}(w_{m'}\mathbb{R})$ then $f (\chi\circ u)\in W^{1,2}_{loc}(\tilde{X})$ and that 
\begin{equation}
\label{eq:wug-f-vs-wug-h}
|\nabla f|_{\tilde X}(x)=|\nabla h|_{w_{m'}\mathbb R}(u(x))
\end{equation}
is valid for $\tilde{m}$-a.e. $x\in u^{-1}([-R,R])$. 

Let $h_n$ be a sequence of Lipschitz functions on $w_{m'}\mathbb{R}$ such that $h_n\to h$ and $\mathrm{lip}_{w_{m'}\mathbb{R}}h_n \to |\nabla h|_{w_{m'}\mathbb{R}}$  in $L^2(w_{m'}\mathbb{R})$. Such a sequence exists by \cite[Theorem 4.3]{Gig}. Now, we consider the functions $f_n:= (h_n\circ u)(\chi\circ u)$. Proposition \ref{prop:prodmeas} implies that $f_n\to f (\chi\circ u)$ in $L^2(\tilde{X})$. Moreover, since $u$ is $1$-Lipschitz, for $x\in u^{-1}([-R,R])$ and $n\in\mathbb{N}$, by Lemma \ref{lem-Dist},
\begin{equation}
\label{eq:lip-const-fn}
\mathrm{lip}_{\tilde{X}}(f_n)(x) = \limsup_{y\to x} \frac{\left|f_n(y) - f_n(x) \right|}{ \tilde d(x,y)}\leq \limsup_{y\to x}\frac{\left| h_n\circ u(y) - h_n\circ u(x)\right|}{|u(y)-u(x)|}= \mathrm{lip}_{w_{m'}\mathbb{R}}h_n\circ u(x).
\end{equation}
From the previous inequality, the Leibniz rule \cite[(3.9)]{Gig} and the convergence of $h_n$ we conclude that $\mathrm{lip}_{\tilde{X}}(f_n)$ is bounded in $L^2(\tilde{X})$. Therefore, passing to a subsequence if necessary, we can assume that there exists a Borel function $G:\tilde{X}\to \mathbb{R}$ such that $\mathrm{lip}_{\tilde{X}}(f_n)\to G$ weakly in $L^2(\tilde X)$.

The lower semicontinuity of minimal weak upper gradients (see the paragraph after \cite[Definition 3.8]{Gig}) and the convergence of $f_n$ to $f(\chi\circ u)$ in $L^2(\tilde{X})$ imply that $|\nabla f(\chi\circ u)|_{\tilde{X}}\leq G$ $\tilde{m}$-a.e.. Moreover by the locality of minimal weak upper gradients \cite[(3.6)]{Gig}, $|\nabla f|_{\tilde{X}}=|\nabla f(\chi\circ u)|_{\tilde{X}}$, $\tilde{m}$-a.e. on $u^{-1}([-R,R])$. Now, passing to the limit in \eqref{eq:lip-const-fn}, we obtain the $\leq$ inequality in \eqref{eq:wug-f-vs-wug-h}.\\

We now proceed to prove the other inequality in \eqref{eq:wug-f-vs-wug-h} by showing the following result, and applying it to $t=u(x')$: Let $f\in W^{1,2}(\tilde{X})$ and for $x'\in X'$ let $f^{(x')}:w_{m'}\mathbb{R}\to \mathbb{R}$ be given by $f^{(x')}(t):=f(T(x',t))$, then for $m'$-a.e. $x'$, $f^{(x')}\in S^2_{loc}(w_{m'}\mathbb{R})$ and 
\[
|\nabla f^{(x')}|_{w_{m'}\mathbb{R}}(t) \leq |\nabla f|_{\tilde{X}}(T(x',t)), \quad m'_w-\text{a.e.}\ (x',t)\in X'_w.
\]

Using that for any $x,y\in \supp( \tilde m)$ with $\pi(x)=\pi(y)$ we have $|u(x)-u(y)|= \tilde d(x,y)$, we observe the following inequality
\begin{eqnarray}
\label{eq:lip-const-f}
\mathrm{lip}_{\tilde{X}}f(x)& = & \limsup_{y\to x}\frac{\left| f(x)-f(y) \right|}{ \tilde d(x,y)} \geq \limsup_{\overset{y\to x}{\pi(y)=\pi(x)}} \frac{\left| f(x)-f(y)
 \right|}{ \tilde d(x,y)} \\
\nonumber & = & \limsup_{\overset{y\to x}{\pi(y)=\pi(x)}} \frac{\left| f^{(\pi(x))}(u(x))-f^{(\pi(x))}(u(y))
 \right|}{|u(x)-u(y)|} = \mathrm{lip}_{w_{m'}\mathbb{R}} f^{(\pi(x))}(u(x)).
\end{eqnarray}

By \cite[Theorem 4.3]{Gig}, there exists a sequence $(f_n)\subset L^2(\tilde{X})$ of Lipschitz functions such that $f_n\to f$ and $\mathrm{lip}_{\tilde{X}}(f_n)\to |\nabla f|_{\tilde{X}}$ in $L^2(\tilde{X})$. Passing to a subsequence if necessary we can further assume that $\sum_n \| f_n- f_{n+1} \|_{L^2(\tilde{X})}<\infty$ and  $\sum_n \| \mathrm{lip}_{\tilde{X}} f_n - |\nabla f|_{\tilde{X}} \|_{L^2(\tilde{X})}<\infty$. This, together with Proposition \ref{prop:prodmeas}, implies that for $m'$-a.e. $x'$,  $f_n(T(x',\cdot)) \to f(T(x',\cdot))$ and $\mathrm{lip}_{\tilde{X}}(f_n)(T(x',\cdot)) \to |\nabla f|_{\tilde{X}}(T(x', \cdot))$ in $L^2(w_{m'}\mathbb{R})$.  

We now fix such an $x'$, apply inequality \eqref{eq:lip-const-f} to the function $f_n$ on $u^{-1}(t)$ and take the limit when $n\to \infty$. Finally, we use that $|\nabla f^{(\pi(x))}|_{w_{m'}\mathbb{R}}\leq \mathrm{lip}_{w_{m'}\mathbb{R}} f^{(\pi(x))}$ (by \cite[(3.8)]{Gig}) and the lower semicontinuity of the minimal weak upper gradients to conclude. 
\end{proof}


\subsection{$W^{1,2}(\tilde X, \tilde d, \tilde m)$ is isomorphic to $W^{1,2}(X'_{\omega},d'_{\omega}, m'_{\omega})$}

The aim of this section is to show that $(\tilde X, \tilde d, \tilde m)$ and $(X'_{\omega},d'_{\omega}, m'_{\omega})$ are isomorphic. This will be achieved 
applying Proposition \ref{prop:isom}.  Thus, we only need to show that right composition with $S$ provides an isometry from $W^{1,2}(X'_w)$ to $W^{1,2}(\tilde X)$.

In Proposition \ref{prop:approximation}  we showed that $\mathcal A\cap W^{1,2}(X'_w)$ is dense in $W^{1,2}(X'_w)$. Here
\begin{align*}
\mathcal G = &\Big\{g \in S^2_{\rm loc}(X'_w) \ |\  g(x',t)=\tilde g(x')\textrm{ for some }  \tilde g\in S^2(X')\cap L^\infty(X') \Big\},\\
\mathcal H  = &\Big\{h \in S^2_{\rm loc}(X'_w)\ |\  h(x',t)=\tilde h(t)\textrm{ for some }  \tilde h\in S^2( w'
_{m'}\mathbb R)\cap L^\infty(\mathbb R) \Big\}\\
\mathcal A =  & \textrm{ algebra spanned by }\mathcal G\cup\mathcal H \subset S^2_{\rm loc}(X'_w).
\end{align*}

The proof that right composition with $S$ provides an isometry from $W^{1,2}(X'_w)$ to $W^{1,2}(\tilde X)$ is divided in the following way.

\begin{itemize}  
\item[0)]   Proposition \ref{cor-mapsGH}: For every $f\in\mathcal G$ or  $f\in\mathcal H$, we have that $f\circ S \in S^2_{\rm loc}(\tilde X)$ and $|\nabla(f\circ S)|_{\tilde X}= |\nabla f|_{X'_w}\circ S$ $\tilde m$-a.e.
\item[1)] Lemma \ref{lem-ort}: For every $g\in \mathcal G$ and $h\in \mathcal H$,  $\langle \nabla g,\nabla h\rangle_{X'_w}=0$ and 
$\langle\nabla( g\circ S),\nabla(h\circ S) \rangle_{\tilde X}=0$ hold $\tilde m-a.e.$.
\item[2)]   Proposition \ref{prop:astable}: Every $f\in\mathcal A$ satisfies $f\circ S \in S^2_{\rm loc}(\tilde X)$ and $|\nabla(f\circ S)|_{\tilde X}= |\nabla f|_{X'_w}\circ S$ $\tilde m$-a.e..
\item[3)] Proposition \ref{prop:almost}: Right composition with $S$ is a homeomorphism between $W^{1,2}(X'_w)$ and $W^{1,2}(\tilde X)$ 
\end{itemize}

\begin{prop}\label{cor-mapsGH}
The maps 
\[
\mathcal G \to S^2_{\rm loc}(\tilde X),  \,\,\,\,\, g \mapsto g \circ S,
\]
\[
\mathcal H \to  S^2_{\rm loc}(\tilde X),  \,\,\,\,\, h \mapsto h \circ S,
\]
are well defined, and satisfy
$|\nabla (g\circ S)|_{\tilde X}  = |\nabla g|_{X'_w}\circ S $ and $|\nabla (h\circ S)|_{\tilde X}= |\nabla h|_{X'_w}\circ S$  $\tilde m-a.e.$.
\end{prop}

\begin{proof}
Combining Corollary \ref{cor:sezioniprod} with a cut-off function $f$ such that $\supp(f) \subset u^{-1}[a,b]$ and Theorem
\ref{thm-DfDg2} one shows that
$g\circ S \in S^2_{\rm loc}(\tilde X)$, and $|\nabla (g\circ S)|_{\tilde X}  = |\nabla g|_{X'_w}\circ S $ $\tilde m-\text{a.e.}$

Similarly,  Corollary \ref{cor:sezioniprod} and Proposition \ref{prop:sezioni2} give
$h\circ S \in S^2_{\rm loc}(\tilde X)$  and $|\nabla (h\circ S)|_{\tilde X}= |\nabla h|_{X'_w}\circ S$  $\tilde m-\text{a.e.}$
\end{proof}

\begin{lem}\label{lem-ort}(Orthogonality relations) 
With the same notation as above, let $g\in \mathcal G$ and $h\in \mathcal H$.  Then,
\begin{equation}
\label{eq:perp1}
\langle \nabla g,\nabla h\rangle_{X'_w}=0, \qquad m'_w\text{-a.e.}
\end{equation}
and
\begin{equation}
\label{eq:perp3}
\langle\nabla( g\circ S),\nabla(h\circ S) \rangle_{\tilde X}=0, \qquad \tilde m\text{-a.e.}
\end{equation}
\end{lem}

\begin{proof}
Let $\tilde g\in S^2(X')\cap L^\infty(X')$ and $\tilde h\in S^2(w_{m'} \mathbb R)\cap L^\infty(\mathbb R)$ be such that $g(x',t)=\tilde g(x')$ and $h(x',t)=\tilde h(t)$.
Corollary \ref{cor:sezioniprod} implies
\[
|\nabla(g+h)|^2_{X'_w}(x',t)= w_d^{-2}(t) |\nabla \tilde g|_{X'}^2(x')+ |\nabla \tilde h|_{w_{m'}\mathbb R}^2(t), \qquad m'_w\text{-a.e.}\ (x',t).
\]
Using equation \eqref{eq:warpedgrad}  we get equation \eqref{eq:perp1}:
\[
\begin{split}
2 \langle \nabla g, \nabla h\rangle_{X'_w}=|\nabla(g+h)|^2_{X'_w}-|\nabla g|^2_{X'_w}-|\nabla h|^2_{X'_w}=0,\qquad m'_w\text{-a.e.} 
\end{split}
\]

To prove equation \eqref{eq:perp3}, notice that the Chain rule and the identity $ h\circ S=\tilde h\circ u$ yield
\[
\langle \nabla ( g\circ S ),\nabla( h\circ S) \rangle_{\tilde X}=\tilde h'\circ u \langle\nabla (g\circ S ),\nabla u \rangle_{\tilde X}, \qquad \tilde m\text{-a.e.}
\]
Then to conclude it is sufficient to show that 
\[
\langle\nabla (g\circ S),\nabla u \rangle_{\tilde X}=0, \qquad \tilde m\text{-a.e.}
\]

The previous equality holds because  $\tilde g\circ \pi \circ F_t = \tilde g\circ \pi$, and with a truncation argument we can see that  the following derivation rule is also valid for functions in $S^2_{\rm loc}(\tilde X)$: 
\[
\langle\nabla (g\circ S),\nabla u \rangle_{\tilde X }= \lim_{t \to 0} \frac{g\circ S \circ F_t - g\circ S}{t}= \lim_{t \to 0} \frac{\tilde g\circ \pi \circ F_t - \tilde g\circ \pi}{t}=0, \qquad \tilde m\text{-a.e.}
\]
\end{proof}

\begin{prop}\label{prop:astable} 
With the same notation as above, every $f\in\mathcal A$ satisfies $f\circ S \in S^2_{\rm loc}(\tilde X,\tilde d,\tilde m)$, and
\[
|\nabla (f\circ S)|_{\tilde X}= |\nabla f|_{X'_w}\circ S, \qquad \tilde m-a.e..
\]
\end{prop}

\begin{proof}
Let $f \in \mathcal A$. Then $f$ can be written as $f=\sum_{i\in I}g_ih_i$ for some finite set $I$, $g_i\in \mathcal G$ and $h_i\in \mathcal H$, $i\in I$.
By the infinitesimal Hilbertianity of $X'_w$,  Proposition \ref{prop-SobtoLip}  and  Corollary \ref{cor:warpHil}, we know that $m'_w$-a.e. 
\begin{equation}
\label{eq:stanchino}
\begin{split}
|\nabla f|^2_{X'_w}&=\sum_{i,j\in I} g_ig_j \langle\nabla h_i,\nabla h_j\rangle _{X'_w}+ g_ih_j \langle\nabla h_i,\nabla g_j\rangle_{X'_w}\\
&\qquad\qquad+ h_ig_j \langle\nabla g_i,\nabla h_j\rangle_{X'_w}+ h_ih_j \langle\nabla g_i,\nabla g_j\rangle_{X'_w}\\
&=\sum_{i,j\in I} g_ig_j\langle\nabla h_i,\nabla h_j\rangle_{X'_w}+h_ih_j \langle\nabla g_i,\nabla g_j\rangle_{X'_w},
\end{split}
\end{equation}
where we used  \eqref{eq:perp1} in the second step. 

Corollary \ref{cor-mapsGH} implies
\[
\begin{split}
\langle\nabla h_i,\nabla h_j\rangle_{X'_x}\circ S=\langle\nabla (h_i\circ S),\nabla (h_j\circ S)\rangle_{\tilde X},\\
\langle\nabla g_i,\nabla g_j\rangle_{X'_w}\circ S=\langle\nabla (g_i\circ S),\nabla (g_j\circ S)\rangle_{\tilde X},
\end{split}
\]
$\tilde m$-a.e. for any $i,j\in I$. Thus writing---to shorten the notation---$\bar g_i,\bar h_i$ in place of $g_i\circ S,h_i\circ S$ respectively, from \eqref{eq:stanchino} we have
\[
|\nabla f|^2_{X'_w}\circ S=\sum_{i,j\in I} \bar g_i\bar g_j \langle\nabla\bar  h_i,\nabla \bar h_j\rangle_{X'_w}+\bar h_i\bar h_j\langle\nabla \bar g_i,\nabla\bar  g_j\rangle_{X'_w}.
\]
Using  the orthogonality relation  \eqref{eq:perp3} and the fact that $\tilde X$ is infinitesimally Hilbertian we can do the same computations as in \eqref{eq:stanchino}, in reverse order, to get
\[
\begin{split}
|\nabla f|^2_{X'_w}\circ S&=\sum_{i,j\in I} \bar g_i\bar g_j\langle\nabla\bar  h_i,\nabla\bar  h_j\rangle_{\tilde X}+ \bar g_i\bar h_j\langle\nabla\bar  h_i,\nabla\bar  g_j\rangle_{\tilde X}\\
&\qquad+ \bar h_i\bar g_j\langle\nabla\bar  g_i,\nabla \bar h_j\rangle_{\tilde X}+ \bar h_i\bar h_j\langle\nabla\bar  g_i,\nabla \bar g_j\rangle_{\tilde X}=|\nabla (f\circ S)|_{\tilde X}^2,
\end{split}
\]
$\tilde m$-a.e.
\end{proof}

Recall that in Proposition \ref{prop:homeo} we defined functions $S: \tilde X \to X'_w$ and $T:X'_w \to \tilde X$ inverses of each other such that $S$ is $1$-Lipschitz and $T$ is locally Lipschitz.

\begin{prop}\label{prop:almost} 
With the same notation as above the following holds 
\begin{enumerate}[(i)]
\item  If $f\in W^{1,2}({X'_w})$ then $f\circ S \in W^{1,2}(\tilde X)$ and
\begin{equation}
\label{eq:almost1}
\||\nabla(f\circ S)|\|_{L^2(\tilde X)}   \leq  \||\nabla f|\|_{L^2({X'_w})}. 
\end{equation}
\item If  $f\circ S \in W^{1,2}(\tilde X)$  then $ f \in S_{loc}^2(X'_w)$ and each $x \in \tilde X$ has a neighborhood  $\Omega_x$ such that 
\begin{equation}
\label{eq:almost2}
L^{-1}\||\nabla f|\|_{L^2( S(\Omega_x))}   \leq   \||\nabla(f\circ S)|\|_{L^2(\Omega_x)}. 
\end{equation}
\end{enumerate}
Here $L=\Lip^{-1}(T^{-1}(x))$.
\end{prop}

\begin{proof} 
Note that $(\tilde X, d, \tilde m)$ and $X'_{w}= (X' \times _w \mathbb R, d'_w,m'_w)$ satisfy the hypotheses of 
Lemma \ref{le:contrdual}. That is, they satisfy the Sobolev to Lipschitz property, see the paragraph after \cite[Definition 4.9]{Gig} and Proposition \ref{prop-SobtoLip}. Moreover,  $T_\sharp(m'_w)=\tilde m$ and $S_\sharp \tilde m = m'_w$ by Proposition \ref{prop:prodmeas}.  To prove the first inequality recall that by Proposition \ref{prop:homeo} the map $S$ is $1$-Lipschitz. Then equation \ref{eq:almost1} follows by Lemma \ref{le:contrdual}. 

To prove the second inequality, choose $\Omega_x= T(B(T^{-1}(x), r))$ and rescale $d'_w$ by $L$. Then we get  $\Lip(T|_{B(T^{-1}(x), r)}) \leq 1$.
With this rescaling the corresponding gradient part of the Sobolev norm is scaled by $\frac1L $.  The result follows by  Lemma \ref{le:contrdual}.
\end{proof}

The main theorem of this section follows.

\begin{thm}[$(\tilde X, \tilde d, \tilde m)$ is isomorphic to $(X'_{\omega},d'_{\omega}, m'_{\omega})$]\label{thm:pitagora}  
The maps $T$ and $S$ given in Proposition \ref{prop:homeo} are isomorphisms of metric measure spaces.
\end{thm}

\begin{proof}
By the paragraph after \cite[Definition 4.9]{Gig}   $\tilde X$ has the Sobolev to Lipschitz property and by Proposition \ref{prop-SobtoLip}  and Theorem  \ref{thm:StoLipGH} $X'_w$ also has the Sobolev to Lipschitz property. Hence, it is enough to apply Proposition \ref{prop:isom}.  By Proposition \ref{prop:prodmeas} we know that  $T$ and $S$ are measure preserving.  It remains to prove that   $f\in W^{1,2}(X'_w)$ if and only if $f\circ S \in W^{1,2}(\tilde X)$ and that
\begin{equation}
\label{eq:splitting}
\||\nabla (f\circ S)|_{\tilde X}\|_{L^{2}( \tilde X)}=\||\nabla f|_{X'_w}\|_{L^{2}(X'_w)}.
\end{equation}

Let $f\in W^{1,2}(X'_w)$. By Proposition \ref{prop:approximation}  there exists a sequence $\{f_n \}\subset \mathcal A\cap W^{1,2}(X'_w)$ converging to $f$ in $W^{1,2}(X'_w)$. Then the first inequality in Proposition \ref{prop:almost}  implies that both  $f_n\circ S $, and $f\circ S$ are in $W^{1,2}(\tilde X)$, with  $f_n\circ S$ converging to $f\circ S$ in $W^{1,2}(\tilde X)$. From Proposition \ref{prop:astable} we get, 
\[
|\nabla f_n|_{X'_w}\circ S=|\nabla ( f_n\circ S)|_{\tilde X},\qquad\tilde m\text{-a.e.}
\]
Taking the $L^2$ norm of the functions in the previous equality and taking the limit as $n\to\infty$ we get \eqref{eq:splitting}.

If $f:X'_w \to \mathbb R$ is such that $f\circ S\in W^{1,2}(\tilde X_{a,b})$, the second inequality in Proposition \ref{prop:almost}  implies that 
each $x \in \tilde X$ has a neighborhood  $\Omega_x$ on which the above argument can be repeated. Thus 
\[
|\nabla f_n|_{S(\Omega_x)}\circ S=|\nabla ( f_n\circ S)|_{\Omega_x},\qquad\tilde m\text{-a.e.}
\]
By the locality of the weak upper gradient we have equality in the whole space and therefore $f \in W^{1,2}(X'_w)$.
\end{proof}


\section{$\RCDst(0,N)$-condition for $X'$}\label{sec-RCDcond}

Recall that $X'$ is an infinitesimally Hilbertian space satisfying the Sobolev to Lipschitz property. Under these conditions, \cite[Theorem 7]{EKS} implies that the validity of the Bochner inequality is equivalent to the $\RCDst$ condition. Hence, to prove that $X'$ is an $\RCDst(0,N)$-space, we will show that the weak Bochner inequality holds. 

We begin with the following technical lemma about extending test functions on $X'$ to that of $X_w'$. From this we will obtain that $(X',d',m')$ is an $\RCDst(0,N)$ space via a limiting argument. 
Denote the Laplacian operator of $X'$ by $\underline{\Delta}$.

\begin{lem}
\label{lem-regularity-of-extension}
Let $\rho \in C_0^{\infty}(\mathbb R)$ and $f\in D(\underline{\Delta})\cap L^{\infty}(X')$ be such that $\underline{\Delta}f\in W^{1,2}(X')\cap L^{\infty}(X')$. Let $\overline{f}:X'_w\to \mathbb{R}$ be defined as $\overline{f}(x,t)=f(x)$ and $\overline{\rho}:X'_w\to\mathbb{R}$ as $\overline{\rho}(x,t)=\rho(t)$. Then $\overline{f}\overline{\rho}\in D(\Delta)\cap L^{\infty}(X'_w)$ and $\Delta(\overline{f}\overline{\rho})\in W^{1,2}(X'_w)\cap L^{\infty}(X'_w)$.
\end{lem}

\begin{proof}
Clearly $\overline{f}\overline{\rho}\in L^{\infty}(X'_w, m'_w)$. 
Also, $\overline{f}\overline{\rho}\in L^2(X'_w, m_w)$ because
\[
\int_{X'_w}|\overline{f}\overline{\rho}|^2\, \mathrm{d}m'_w 
\leq \| f\|^2_{L^2(X')} \int_{\mathbb R}\rho(s)^2 w_{m'}(s)\,\mathrm{d}s.
\]

By a result of \cite{GigHan}, see Theorem \ref{def-warped-BL} in Section 2,  $\overline{f}\overline{\rho}\in W^{1,2}(X'_w)\cap L^{\infty}(X'_w)$. We will now prove that $\overline{f}\overline{\rho}\in D(\Delta)$. It is clear that $\mathrm{Test}(X'_w)\cap \mathcal{A}\neq \emptyset$. Let $\varphi\in  \mathrm{Test}(X'_w)\cap \mathcal{A}$ be given by $\varphi= \sum_i^{n}a_ih_ig_i$, where $a_i\in\mathbb{R}$, $h_i\in\mathcal{H}$ and $g_i\in \mathcal{G}$. Then,
\begin{eqnarray*}
\int_{X'_w}\left\langle \nabla (\overline{f}\overline{\rho}), \nabla \varphi \right\rangle_{X'_w}\, \mathrm{d}m'_w 
&=& \sum a_i\int_{X'_w} g_i\left\langle \nabla (\overline{f}\overline{\rho}), \nabla h_i\right\rangle_{X'_w} + h_i\left\langle \nabla (\overline{f}\overline{\rho}), \nabla g_i\right\rangle_{X'_w}\, \mathrm{d}m'_w\\
&=& \sum a_i\int_{X'_w} \left[ \overline{\rho}g_i \left\langle \nabla \overline{f}, \nabla h_i\right\rangle_{X'_w} + \overline{f}g_i\left\langle \nabla \overline{\rho}, \nabla h_i\right\rangle_{X'_w} \right. \\
&+& \left. \overline{\rho}h_i\left\langle \nabla \overline{f}, \nabla g_i\right\rangle_{X'_w}  + \overline{f}h_i\left\langle \nabla \overline{\rho}, \nabla g_i\right\rangle_{X'_w}\right] \mathrm{d}m'_w\\
&=& \sum a_i\int_{X'_w} \left( \overline{f}g_i\left\langle \nabla \overline{\rho}, \nabla h_i\right\rangle_{X'_w} + \overline{\rho}h_i\left\langle \nabla \overline{f}, \nabla g_i\right\rangle_{X'_w} \right) \mathrm{d}m'_w.
\end{eqnarray*}
Here we have used the validity of the Leibniz rule due to the regularity of the functions involved as well as the orthogonality relations. Now we note that 
$$\left\langle \nabla\overline{f}, \nabla g_i\right\rangle_{X'_w}= w_{d'}^{-2}\left\langle \nabla f, \nabla g_i\right\rangle_{X'}, \ \ \ \ \ \  \left\langle \nabla \overline{\rho}, \nabla h_i\right\rangle_{X'_w}= \rho' h_i' , $$
 as a consequence of Theorem \ref{def-warped-BL} and polarization. Therefore we obtain, 
\begin{eqnarray*}
\sum a_i\int_{X'_w} \overline{\rho}h_i\left\langle \nabla \overline{f}, \nabla g_i\right\rangle_{X'_w} \, \mathrm{d}m'_w &=& \sum a_i\int_{\mathbb{R}}\rho h_i w_{d'}^{-2}  w_{m'} \int_{X'} \left\langle \nabla \overline{f}, \nabla g_i\right\rangle_{X'} \, \mathrm{d}m'\,\mathrm{d}s\\
&=& -\sum a_i\int_{\mathbb{R}}\rho h_i w_{d'}^{-2} w_{m'}\int_{X'}g_i\underline{\Delta}f\,\mathrm{d}m'\,\mathrm{d}s,
\end{eqnarray*}
and 
\begin{eqnarray*}
	\sum a_i\int_{X'_w} \overline{f}g_i\left\langle \nabla \overline{\rho}, \nabla h_i\right\rangle_{X'_w} \, \mathrm{d}m'_w &=& \sum a_i\int_{\mathbb{R}}\rho' h_i' w_{m'} \,\mathrm{d}s \int_{X'} f g_i \, \mathrm{d}m'\\
	&=& -\sum a_i\int_{\mathbb{R}} (\rho' w_{m'})' h_i \,\mathrm{d}s  \int_{X'}f g_i \,\mathrm{d}m'.
\end{eqnarray*}

Hence, for all $\varphi\in \mathrm{Test(X'_w})\cap \mathcal{A}$, we have that
\[
\int_{X'_w}\left\langle \nabla (\overline{f}\overline{\rho}), \nabla \varphi \right\rangle_{X'_w}\, \mathrm{d}m'_w = -\int_{X'_w}\varphi [ \rho w_{d'}^{-2} \left(\underline{\Delta}f \right) + \overline{f} (w_{m'}^{-1} (\rho' w_{m'})')   ]\mathrm{d}m'_w.
\]
Here we have abused notation and denote $\underline{\Delta}f=\underline{\Delta}f\circ p_1$ and similarly $\rho w_{d'}^{-2} =( \rho w_{d'}^{-2}) \circ p_2$ etc.\ (We do so as well in what follows).
Since $\mathrm{Test(X'_w})\cap \mathcal{A}$ is dense in $W^{1,2}(X'_w)$, so by an approximation argument the previous equality holds for all $\varphi\in W^{1,2}(X'_w)$. Hence $\overline{f}\overline{\rho}\in D(\Delta)$ and 
\begin{equation} \label{lapforwarp}
\Delta(\overline{f}\overline{\rho})= \rho\left(w_{d'}^{-2}\underline{\Delta}f \right) + \overline{f} w_{m'}^{-1} (\rho' w_{m'})'. 
\end{equation} It immediately follows that $\Delta(\overline{f}\overline{\rho})\in W^{1,2}(X'_w)\cap L^{\infty}(X'_w)$.
\end{proof}

We now come to the main result of this section, the $\RCDst$-condition for $X'$. We will accomplish this by using the $\RCDst$-condition for $X'_w$. Namely we plug into Bochner inequality for $X'_w$ those test functions constructed in Lemma \ref{lem-regularity-of-extension}. The Bochner inequality for $X'$ will come out of a suitable limit.

\begin{prop}\label{prop-X1-RCD}
For all $f\in D(\underline{\Delta})$ such that $\underline{\Delta}f\in W^{1,2}(X',d',m')$ and all non-negative $g\in D(\underline{\Delta})\cap L^{\infty}(X',m')$ such that $\underline{\Delta}g\in L^{\infty}(X',m')$,  the following is satisfied:
\[
\frac{1}{2}\int_{X'}\!\underline{\Delta}g |\nabla f|^2_{X'}\, \mathrm{d}m'-\int_{X'}\!g\left\langle \nabla (\underline{\Delta}f), \nabla f\right\rangle_{X'}\, \mathrm{d}m'\geq 
 \frac{1}{N}\int_{X'}g(\underline{\Delta} f)^2\, \mathrm{d}m'.
\] 
In other words, $(X',d',m')$ is an $\RCDst(0,N)$ space.
\end{prop}

\begin{proof}
We will show the inequality holds for functions $f,g\in \mathrm{Test}(X')$ since the general case follows by the density of $\mathrm{Test}(X')$ in $W^{1,2}(X')$. With this assumption, it follows from Lemma \ref{lem-regularity-of-extension} that we can apply Bochner's inequality on $X'_w$ for the functions $\overline{f}\overline{\rho}$ and $\overline{g}\,\overline{\rho}$, that is
\begin{eqnarray*}
\frac{1}{2}\int_{X'_w}\!\Delta(\overline{g}\overline{\rho}) |\nabla (\overline{f}\overline{\rho})|^2_{X'_w}\, \mathrm{d}m'_w-\int_{X'_w}\!\overline{g}\overline{\rho}\left\langle \nabla (\Delta(\overline{f}\overline{\rho})), \nabla (\overline{f}\overline{\rho})\right\rangle_{X'_w}\, \mathrm{d}m'_w  &\\
\geq  -(N-1)\int_{X'_w}\overline{g}\overline{\rho}|\nabla(\overline{f}\overline{\rho})|^2_{X'_w}\, \mathrm{d}m'_w+ \frac{1}{N}\int_{X'_w}\overline{g}\overline{\rho}(\Delta(\overline{f}\overline{\rho}))^2\, \mathrm{d}m'_w. & & 
\end{eqnarray*}

We now compute each term of the inequality, using Theorem \ref{def-warped-BL}, the orthogonality relations and equation (\ref{lapforwarp}). We will first compute them for general functions $\rho\in C_0^{\infty}(\mathbb R)$ and specialize later for the limiting argument after some simplifications.

\begin{eqnarray*}
\int_{X'_w}\!\Delta(\overline{g}\overline{\rho}) |\nabla (\overline{f}\overline{\rho})|^2_{X'_w}\, \mathrm{d}m'_w
&  =  & \int_{X'_w} (\rho\left(w_{d'}^{-2}\underline{\Delta}g \right) + g w_{m'}^{-1} (\rho' w_{m'})')(\rho^2 w_{d'}^{-2} | \nabla f |_{X'}^2+ f^2 (\rho')^2)\, \mathrm{d}m'_w \\
& = & \int_{\mathbb{R}}\rho^3w_{d'}^{-4}w_{m'}\, \mathrm{d}s \int_{X'}(\underline{\Delta}g)|\nabla f|^2_{X'}\, \mathrm{d}m' \\
& &  + \int_{\mathbb{R}}\rho (\rho')^2 w_{d'}^{-2}w_{m'}\, \mathrm{d}s \int_{X'}f (\underline{\Delta}g) \, \mathrm{d}m' \\
& &  + \int_{\mathbb{R}}\rho^2 w_{d'}^{-2}(\rho' w_{m'})'\, \mathrm{d}s \int_{X'} g | \nabla f |_{X'}^2\, \mathrm{d}m' \\
& &  + \int_{\mathbb{R}}(\rho')^2 (\rho' w_{m'})'\, \mathrm{d}s \int_{X'} g f^2\, \mathrm{d}m'. 
\end{eqnarray*}

Similarly, 
\begin{eqnarray*}
\int_{X'_w}\!\overline{g}\overline{\rho}\left\langle \nabla (\Delta(\overline{f}\overline{\rho})), \nabla \overline{f}\overline{\rho}\right\rangle_{X'_w}\, \mathrm{d}m'_w &=& \int_{\mathbb{R}}\rho^3w_{d'}^{-4}w_{m'}\, \mathrm{d}s \int_{X'}g \left\langle \nabla \underline{\Delta}f, \nabla f\right\rangle_{X'}\, \mathrm{d}m' \hspace*{1in} \\
& &  + \int_{\mathbb{R}}\rho^2 w_{d'}^{-2}(\rho' w_{m'})'\, \mathrm{d}s \int_{X'} g | \nabla f |_{X'}^2\, \mathrm{d}m' \\
& & + \int_{\mathbb{R}}\rho \rho' (\rho w_{d'}^{-2})' w_{m'}\, \mathrm{d}s \int_{X'}g f (\underline{\Delta} f) \, \mathrm{d}m' \\
& &  + \int_{\mathbb{R}}\rho \rho' (w_{m'}^{-1}(\rho' w_{m'})')' w_{m'}\, \mathrm{d}s \int_{X'} g f^2\, \mathrm{d}m', 
\end{eqnarray*}

\begin{eqnarray*}
\int_{X'_w}\overline{g}\overline{\rho}|\nabla(\overline{f}\overline{\rho})|^2_{X'_w}\, \mathrm{d}m'_w = \int_{\mathbb{R}}\rho^3 w_{d'}^{-2} w_{m'}\, \mathrm{d}s \int_{X'} g | \nabla f |_{X'}^2\, \mathrm{d}m'  + \int_{\mathbb{R}}\rho (\rho')^2 w_{m'}\, \mathrm{d}s \int_{X'} g f^2\, \mathrm{d}m',
\end{eqnarray*}

and
\begin{eqnarray*}
\int_{X'_w}\overline{g}\overline{\rho}(\Delta(\overline{f}\overline{\rho}))^2\, \mathrm{d}m'_w & = & \int_{\mathbb{R}}\rho^3w_{d'}^{-4}w_{m'}\, \mathrm{d}s \int_{X'}g (\underline{\Delta}f)^2\, \mathrm{d}m' + \int_{\mathbb{R}} \rho (w_{m'}^{-1}(\rho' w_{m'})')^2 w_{m'}\, \mathrm{d}s \int_{X'} g f^2\, \mathrm{d}m' \\
& & + 2 \int_{\mathbb{R}} \rho^2 w_{d'}^{-2} (\rho' w_{m'})'\, \mathrm{d}s \int_{X'}g f (\underline{\Delta} f) \, \mathrm{d}m'.
\end{eqnarray*}

Now let $\rho\in C_0^{\infty}(\mathbb R)$ be a cut-off function on $\mathbb R$ such that $\rho(t)=1$ for $t\in [-1, 1]$, $\rho(t)=0$ for $|t|\geq 2$ and $0\leq \rho \leq 1$. For each $n\in \mathbb N$, set $\rho_n(t)=\rho(t+n)$. Replace $\rho$ by $\rho_n$ in all the formulas above and plug them into Bochner's inequality on $X'_w$. Using $w_{d'}(t)=e^t, \ w_{m'}(t)=e^{(N-1)t}$, we find that $$\int_{\mathbb{R}}\rho_n^3w_{d'}^{-4}w_{m'}\, \mathrm{d}s \geq \int_{-1-n}^{1-n} e^{(N-5)s} \mathrm{d}s = C(N) e^{-(N-5)n}, \ \ \ C(N)>0, $$ while all other integrals over $\mathbb R$ are of lower order. For example,
$$
 \int_{\mathbb{R}}\rho (\rho')^2 w_{d'}^{-2}w_{m'}\, \mathrm{d}s=O(e^{-(N-3)n}), \ \ \ \int_{\mathbb{R}}\rho \rho' (w_{m'}^{-1}(\rho' w_{m'})')' w_{m'}\, \mathrm{d}s=O(e^{-Nn}).
$$

Therefore, dividing every term by $\int_{\mathbb{R}}\rho_n^3w_{d'}^{-4}w_{m'}\, \mathrm{d}s$ and letting $n \rightarrow \infty$, we obtain
$$
\frac{1}{2}\int_{X'}\!\underline{\Delta}g |\nabla f|^2_{X'}\, \mathrm{d}m'-\int_{X'}\!g\left\langle \nabla (\underline{\Delta}f), \nabla f\right\rangle_{X'}\, \mathrm{d}m'\geq 
\frac{1}{N}\int_{X'}g(\underline{\Delta} f)^2\, \mathrm{d}m'
$$
which is the desired result. 
\end{proof}

Theorem~\ref{thm-main2} now follows from Theorem~\ref{thm:pitagora} and Proposition~\ref{prop-X1-RCD}.


\section{Proof of Theorems~\ref{thm-main1} and \ref{stability}}

We first adapt the ideas of Chen-Rong-Xu \cite[Lemma 4.4]{Chen-Rong-Xu} to conclude that $(\tilde X, \tilde d, \tilde m)$ is isometric to a real hyperbolic space. 
Then we prove the stability of the volume entropy when imposing a uniform lower bound on the systoles.

\begin{proof}[Proof of Theorem \ref{thm-main1}]

By Theorem  \ref{thm-main2}   and  Proposition \ref{prop-X1-RCD}, we know that $(\tilde X, \tilde{d}, \tilde m)$ is isomorphic to the warped product space $(X'_w, d'_w, m'_w)$,  with $w_{d'}(t)=e^t$ and $w_{m'}(t)= e^{(N-1)t}$, and that $(X', d', m')$ is an $\RCDst(0, N)$ space.  Thus, by Mondino-Naber \cite[Corollary 1.2]{MN}, there exists a point $y \in X'$ such that every tangent space of  $(X', {d}', m')$ at $y$ is isometric to $(\mathbb R^{k-1}, d_{Euc},   \mathcal L_{k-1},0)$ for some $k-1 \leq N$. 
 That is,  for any sequence of positive numbers $r_i \to 0$, we have that $(X',   r_i^{-1}d',    c_{r_i}^y m',  y)$ converges in the pointed measured Gromov-Hausdorff (pmGH) sense to $(\mathbb R^{k-1}, d_{Euc},   \mathcal L_{k-1},0)$, where  
$\mathcal L_{k-1}$ denotes the normalized $(k-1)$-Lebesgue measure  so that 
\begin{equation}
\int_{B(0,1)}  (1  -  |x|) d\mathcal L_{k-1}(x) =  1
\end{equation}
and the numbers $c_r^y$ are given by
\begin{equation}
c_{r}^y  = \Big( \int_{B(y,r)}  (1  -  \tfrac{1}{r}d'(y, \cdot)) d m' \Big)^{-1}.
\end{equation}

From now on we identify $(\tilde X, \tilde{d}, \tilde m)$ with $(X'_w, d'_w, m'_w)$. 
For any $t \in \mathbb R$ there is a deck transformation $\gamma_t $ 
of $\tilde{X}$ such that 
\begin{equation*}
\tilde d(\gamma_t((0,y)), (t,y))\leq {\rm diam} (X) < \infty. 
\end{equation*}
Note that in the last inequality we used that $X$ is compact. 
As the measure $\tilde m$ is equivariant,  $\gamma_t$  is an isomorphism of metric spaces, i.e. an isometry that preserves the measures. 

 We now take a sequence $t_i  \to \infty$ and a subsequence if necessary so that $\gamma_{t_i}^{-1}(t_i, y)$ converges to $\tilde p$ in $\tilde X$ (again using the compactness of $X$). 
 Then, in the pmGH sense:
\begin{equation}\label{eq-lim0}
(\tilde X,  \tilde d, \tilde m, \tilde p)= (\tilde X, \tilde d, \tilde m, \gamma_{t_i}(\tilde p)),  \ \ \   \lim_{i \to \infty} (\tilde X, \tilde d, \tilde m, \gamma_{t_i}(\tilde p))  = \lim_{i \to \infty} (\tilde X,   \tilde d, \tilde m, (t_i, y)).
\end{equation}
In particular, we have
\begin{equation}\label{eq-lim1}
(\tilde X,  \tilde d, \tilde m, \tilde p) = \lim_{i \to \infty} (\tilde X,   \tilde d, \tilde m, (t_i, y)).
\end{equation}

Now we calculate the limit in \eqref{eq-lim1}.  
For $t_i \in \mathbb R$, define $(X'_{i}, d'_{i}, m'_{i})=  (X', e^{t_i} d',  e^{(N-1)t_i}m')$. 
Consider the following sequence of positive numbers,

\begin{equation}\label{meas-constant}
\frac{e^{(N-1)t_i}}{c^y_{e^{-t_i}}} .
\end{equation}

%

After passing to a subsequence, it converges to a value $c\in[0,\infty]$.	
We will analyze the three possible cases, $c=0$, $c \in (0, \infty)$ and $c=\infty$. 

From the definition of tangent space, in the pmGH sense, 
\begin{equation}\label{eq-TanSpX'}
\lim_{i \to \infty} (X'_{i}, d'_{i}, m'_{i},y) = (\mathbb R^{k-1}, d_{Euc},   c\mathcal L_{k-1},0).
\end{equation}

The map   $(X'_w, d'_w, m'_w, (t_i,y))   \longrightarrow  ( {X'_{i}}_w,   {d'_{i}}_w,     {m'_{i}}_w    , (0,y))$
given by  $(t,x) \mapsto (t-t_i,x)$ is a pointed isometry that preserves the measure and hence
\[ 
  (X'_w, d'_w, m'_w, (t_i,y))\cong ( {X'_{i}}_w,   {d'_{i}}_w,     {m'_{i}}_w    , (0,y)).
  \]
 In combination with equation (\ref{eq-TanSpX'}) it implies that  in  the pmGH sense  (\ref{eq-lim1}) can be written as, 
\begin{align}\label{eq-lim2}
(\tilde X,  \tilde d, \tilde m, \tilde p)=  \lim_{i \to \infty} (\tilde X,  \tilde d, \tilde m, (t_i, y)) =  &  \lim_{i \to \infty}  ( {X'_{i}}_w,   {d'_{i}}_w,     {m'_{i}}_w    , (0,y)) \nonumber \\
 = & \, (\mathbb R \times _w \mathbb R^{k-1}, {d_{Euc}}_w, { c  \mathcal L_{k-1}}_w,0) \nonumber \\
 = &  (\mathbb H^{k}, d_{\mathbb H^k},  c_1{\mathcal H^{k}},0),
\end{align}
where $c_1=\frac{kc}{\omega_{k-1}}$, and $\omega_{k-1}$ denotes the volume of the unit ball in $\mathbb R^{k-1}$. The extra constant comes from the normalization of the Euclidean Lebesgue measure indicated before.

If $c=\infty$ then $ { c\mathcal L_{k-1}}$ is not locally finite and this implies that $\tilde m$ is not locally finite, which is a contradiction. 

If $c  \in (0,\infty)$ then recall that for the hyperbolic space, 
\[
h(\mathbb H^{k}, d_{\mathbb H^k},  c_1{\mathcal H^{k}}) =  k-1.
\]
Since we know that  $h(X, d, m) =  N-1$ then $k=N$.    If $c=0$ then 
\[
h(\mathbb H^{k}, d_{\mathbb H^k},  c_1{\mathcal H^{k}}) =  0, 
\]
which contradicts $h(X, d, m) =  N-1$. Hence $(\tilde X,  \tilde d, \tilde m)$ is isomorphic to $(\mathbb H^{N}, d_{\mathbb H^N},  c_1{\mathcal H^{N}})$ for some $c_1 \in (0, \infty)$, and an integer $N \ge 2$.

\end{proof}

%

Before proving Theorem~\ref{stability} we recall the definition of a systole and the following result. 

\begin{defn}
If $(X,d)$ is a compact length space that admits a universal covering $(\tilde{X}, d_{\tilde{X}})\to (X, d)$, we define the \textit{systole of} $(X,d)$ as,
\[
{\rm sys}(X,d):= \inf \{ d_{\tilde{X}}(\tilde{x}, \gamma \cdot \tilde{x}) \,\, : \,\, \tilde{x}\in  \tilde{X}, \gamma \in \bar{\pi}_{1}(X) - \{ {\rm Id} \} 
\},
\]
where $\bar{\pi}_{1}(X)$ is the group of deck tranformations of $\tilde{X}$, which is referred as the revised fundamental group of $X$ in \cite{Sormani-Wei}. 
\end{defn}

\begin{prop}[Proposition 38 \cite{Rev}]\label{reviron-prop-38}
Let $(X_i,d_i)$  be a sequence of length spaces that have a universal covering space such their systoles are uniformly bounded from below, and that converge to a length space $(Y,d_Y)$ in Gromov-Hausdorff sense. 
Then $(Y,d_Y)$ has a universal covering space and $h(X_i, d_i)$  converges to $h(Y,d_Y)$.  
\end{prop}

Now we are ready to prove the theorem. 

\begin{proof}[Proof of Theorem~\ref{stability}]
By contradiction, assume that there exists a sequence $(X_i,\sfd_i, m_i)$ of  compact  $\RCDst(-(N-1), N)$  spaces, satisfying ${\rm diam}    \le D, \ h(X_i,\sfd_i, m_i) \ge N-1 - 1/i,\  \rm{sys(X_i,d_i)} \ge s$, and such that none of the spaces $X_i$ are 
mGH close to the quotient of a  $N$-dimensional hyperbolic space.  Since the collection of $\RCDst(K, N)$ spaces, for fixed $K \in \mathbb R$ and $N \in [1, \infty)$,  are compact with respect to mGH convergence,  we can assume that  $(X_i,\sfd_i, m_i)$ mGH converges to some $\RCDst(-(N-1), N)$ space, $(X_\infty,  \sfd_\infty, m_\infty)$.  Therefore Proposition \ref{reviron-prop-38} above yields $\ h(X_\infty,  \sfd_\infty, m_\infty) \ge N-1$.   By Theorem \ref{thm-main1}, we also know that  $\ h(X_\infty,  \sfd_\infty, m_\infty) \leq N-1$. Thus,  $(X_\infty,  \sfd_\infty, m_\infty)$ attains the equality case of Theorem \ref{thm-main1}.  Then $N$ is an integer and the universal covering $(\tilde{X}_\infty,  \tilde \sfd_\infty,  \tilde  m_\infty)$ of $(X_\infty,  \sfd_\infty, m_\infty)$  is isometric to an $N$-dimensional real hyperbolic space. The convergence to $X_\infty$ is equivariant with respect to the actions of the revised fundamental groups $\bar \pi_1(X_i)$ along the sequence of spaces $X_{i}$ and $\tilde{X}_{i}$. Therefore $X_\infty$ is isometric to the quotient $\tilde{X}_{\infty} / \bar \pi_1(X_{\infty})$ of a $N$-dimensional real hyperbolic space. As the systoles are bounded below, the corresponding group actions are free and hence $X_\infty$ is isometric to an $N$-dimensional real hyperbolic manifold. This is a contradiction, and we have that $X$ is $\Psi(\epsilon|N,s,D)$ mGH close to an $N$-dimensional hyperbolic manifold.     It now follows from Theorem 6.5  of  Kapovitch-Mondino  \cite{KM} that $X$ is also bi-H\"older homeomorphic to $(X_\infty,  \sfd_\infty)$.  
\end{proof}


\bibliographystyle{amsalpha}
\bibliography{stormod}

\end{document}